\documentclass[twoside,12pt]{article}
\usepackage{geometry}
\usepackage{amsmath,amssymb,amsfonts, amsthm}
\usepackage{txfonts}
\usepackage{todonotes}
\usepackage{graphicx}
\usepackage{mathrsfs}
\usepackage{color}
\usepackage{float}
\usepackage{extarrows}
\usepackage{exscale}
\usepackage{titlesec}
\titleformat{\section}[block]{\Large\center\sc}{\arabic{section}}{0.5em}{}[] 
\usepackage{relsize}
\usepackage{color}
\usepackage{hyperref}

\usepackage[titles]{tocloft} 
\usepackage{fancyhdr}
\theoremstyle{plain}
\newtheorem{theorem}{Theorem}[section]
\newtheorem{lemma}[theorem]{Lemma}

\newtheorem{proposition}[theorem]{Proposition}

\theoremstyle{definition}
\newtheorem{definition}[theorem]{Definition}

\theoremstyle{remark}
\newtheorem{remark}[theorem]{Remark}
\newtheorem*{remark*}{Remark}

\newtheorem{claim}[theorem]{Claim}

\let\oldsection\section
\renewcommand\section{\setcounter{equation}{0}\oldsection}

\def\be{\begin{equation}}
\def\ee{\end{equation}}
\def\bes{\begin{equation*}}
\def\ees{\end{equation*}}
\def\bali{\begin{aligned}}
\def\eali{\end{aligned}}
\newcommand{\pf}{\noindent {\bf Proof. \hspace{2mm}}}

\def\al{\alpha}
\def\e{\epsilon}

\def\la{\lambda}

\def\t{\tilde}
\def\wt{\widetilde}
\def\th{\theta}

\def\dl{\delta}

\def\lt{\left}
\def\rt{\right}
\def\ls{\lesssim}

\def\i{\infty}
\def\p{\partial}
\def\f{\frac}

\def\o{\omega}

\def\s{\sqrt}
\def\q{\quad}
\def\qq{\qquad}

\def\nn{\nonumber}

\allowdisplaybreaks

\def\angt{\langle t\rangle}

\def\bN{\mathbb N}
\def\bR{\mathbb R}
\def\mH{\mathcal{H}}
\def\mG{\mathcal{G}}
\def\mK{\mathcal{K}}
\def\fH{\frak{H}}

\def\bl{\boldsymbol}
\begin{document}

\title{\normalsize\bf LONG-TIME EXISTENCE OF GEVREY-2 SOLUTIONS TO THE 3D PRANDTL BOUNDARY LAYER EQUATIONS}

\author{\normalsize\sc Xinghong Pan and Chao-Jiang Xu}

\date{}

\maketitle

\begin{abstract}  For the three dimensional Prandtl boundary layer equations, we will show that for arbitrary $M$ and sufficiently small $\epsilon$, the lifespan of the Gevrey-2 solution is at least of size $\epsilon^{-M}$ if the initial data lies in suitable Gevrey-2 spaces with size of $\epsilon$.

\medskip



\end{abstract}


\section{Introduction}

\q\ The purpose of this paper is to show the long time behavior of small Gevrey-2 solutions to the three dimensional Prandtl boundary layer equations in the domain $\{t>0,(x, y, z) \in \bR^3, z>0\}$. The equations read as follows,

\be\label{3dprandtl}
\lt\{
\bali
&\partial_{t} {u}+\left({u} \partial_{x}+{v} \partial_{y}+{w} \partial_{z}\right) {u}+\partial_{x} p=\partial_{z}^{2} {u}, \\
&\partial_{t} {v}+\left({u} \partial_{x}+{v} \partial_{y}+{w} \partial_{z}\right) {v}+\partial_{y} p=\partial_{z}^{2} {v}, \\
&\partial_{x} {u}+\partial_{y} {v}+\partial_{z} { w}=0,\\
&({u}, {v}, {w})\big|_{z=0}=0, \quad \lim_{z \rightarrow+\infty}({u}, {v})=(U(t, x, y), V(t, x, y)),\\
&({u}, {v})\big|_{t=0}=(u_0,v_0),
\eali
\rt.
\ee
where $(U(t, x, y), V(t, x, y))$ and $p(t, x, y)$ are the tangential velocity fields and pressure of the Euler flow, satisfying
\be\label{3deuler}
\left\{
\bali
&\partial_{t} U+U \partial_{x} U+V \partial_{y} U+\partial_{x} p=0, \\
&\partial_{t} V+U \partial_{x} V+V \partial_{y} V+\partial_{y} p=0.
\eali
\right.
\ee

The Prandtl equations are a degenerate Navier-Stokes equations, which was proposed by Prandtl \cite{Prandtl:1904MATHCONGRESS} in 1904 to describe the boundary layer phenomenon. Physically, the Prandtl equations bear underlying instabilities, such as the phenomenon of separation, which is related to the appearance of reverse flow in boundary layer. See \cite{Grenier:2000CPAM, GGN:2016DUKE, GN:2019ANNPDE,DM:2019PMI,SWZ:2021ADVANCES} and references therein for some instability and separation phenomenon for the boundary layer equations. Reader can see \cite{OS:1999AMMC} and references therein for more introductions on the boundary layer theory and check \cite{GN:2011CPAM} for some recent development on this topic.

Compared with the Navier-Stokes equations, the main difference of the Prandtl equations is that there is no time evolution for the vertical velocity $w$, which can only be recovered from the incompressibility condition \eqref{3dprandtl}$_3$. Also since the equations of the tangential velocity have no tangential diffusion,  $w\p_zu$ and $w\p_z v$ in the advection term will cause one order tangential derivative loss when we perform finite-order energy estimates. To show local in time well-posedness of the Prandtl equations in Sobolev spaces is even an uneasy thing.

 In two dimensional case, the first rigorous mathematical proof of the local existence in H\"older spaces dates back to Oleinik and Samokhin \cite{OS:1999AMMC}, where the so-called Crocco transform was introduced and the local well-posedness was given under a monotonic assumption in the vertical variable on the initial data. Recently, the local well-posedness result was revisited in \cite{AWXY:2015JAMS} and \cite{MW:2015CPAM} by using a direct weighted energy estimates, where a nice change of variable was introduced to overcome the one order derivative loss problem. The global weak solutions under an additional favorable sign condition on the pressure $p$ was given in Xin and Zhang \cite{XZ:2004ADVANCES}. If the initial data is a small $\e$ perturbation around the monotonic shear flow in Sobolev spaces, Xu and Zhang \cite{XZ:2017JDE} proved that the lifespan of the solutions is of size $\ln \f{1}{\e}$.

 Without  monotonicity assumption on the initial data, the related results is much recent. See E and Engquist \cite{EE:1997CPAM} for a globally  ill-posedness result and Gerard-Varet and Dormy \cite{GVD:2010JAMS} for a locally ill-posedness result  in Sobolev spaces around non-monotonic outflow. (cf. \cite{GN:2011CPAM,GVN:2012ASYMA,LY:2017JMPA} for some improvement). The result in \cite{GVD:2010JAMS} indicates that local existence in time is possible only in smooth Gevrey regularity class.  The first result in this direction is due to Sammartino and Caflisch \cite{SC:1998CMP}, where the local well-posedness in analytical setting (corresponding to Gevrey-1 class) was established by using the abstract Cauchy-Kowalewski theorem. Later, the analyticity on the normal variable was removed in \cite{LCS:2003SIAM}, and Kukavica and Vicol in \cite{KV:2013CMS} gave an energy-based proof. When the data is an $\e$ size in analytical spaces, authors in \cite{ZZ:2016JFA} showed that the lifespan of the analytical solution is of size $\e^{-4/3}$. This result was extended to an almost global-in-time existence in Ignatova and Vicol \cite{IV:2016ARMA}, and global-in-time existence in Paicu and Zhang \cite{PZ:2021ARMA}. To extend the analyticity results to more generalized Gevrey class is uneasy. Under the assumption that the data has only a non-degenerate critical point in the vertical variable for each fixed tangential variable, the local well-posedness of the two dimensional Prandtl equations  in Gevrey-7/4 class was proved in G\'{e}rard-Varet and Masmoudi \cite{GVM:2015ASENS}. See \cite{LY:2020JEMS} for extensions to Gevrey-2 class, where the exponent $2$ is optimal in view of the instability mechanism indicated in \cite{GVD:2010JAMS}.  More recently, the single non-degenerate critical point assumption was removed in Dietert and G\'{e}rard-varet \cite{DGV:2019ANNPDE}. Most recently, global existence of Gevrey-2 small solutions was shown in Wang, Wang and Zhang \cite{WWZ:2021ARXIV} for the two dimensional Prandtl equations.

 In three dimensional case, there are little well-posedness results for the Prandtl equations in Sobolev spaces except  Liu, Wang and Yang \cite{LWY:2017ADVANCES} where the local well-posedness result was given under some flow-structure constraints in addition to the monotonic assumption. Without monotonic assumptions, the above analytical well-posedness results are both valid for the two and three dimensional case except for the global existence of 3D analytic solutions. By introducing a tangentially polynomial weight to the energy functional, the global existence of small analytical solutions for the three dimensional axially symmetric Prandtl equations was given in Pan and Xu \cite{PX:2022ARXIV}. To relax the analyticity to more generalized Gevrey smoothness is not easy especially for the three dimensional Prandtl equations. As far as the authors know, only most recently, the local well-posedness result in Gevrey-2 spaces for the three dimensional Prandtl equations was solved in \cite{LMY:2022CPAM} by introducing some new cancellations. Such techeniques are used in \cite{LiXY:2022} to prove the global well-posedness of Gevrey-2 solutions for a Prandtl model derived from MHD in the Prandtl-Hartmann regime.

 Until now, as far as the authors' knowledge, there isn't any result concerning on the long time behavior of solutions for the three dimensional Prandtl equations in generalized Gevrey spaces rather than the analytical setting. This is our preliminary interest of this paper.  The main purpose of this paper is to show that for arbitrary $M$ and sufficiently small $\epsilon$, if we consider the outflow $(U,V)$ in \eqref{3deuler} is zero, then the lifespan of the Gevrey-2 solution to the three dimensional Prandtl equations \eqref{3dprandtl} is of size $\epsilon^{-M}$ if the initial data lie in suitable Gevrey-2 spaces with size of $\epsilon$. Also in one of our forthcoming papers, we will show the global existence of Gevrey-2 solutions for the 3D axially symmetric Prandtl equations.

When the outflow $(U,V)\equiv 0$, the Prandtl equations \eqref{3dprandtl} degenerate to
\be\label{3dprandtl1}
\lt\{
\bali
&\partial_{t} {u}+\left({u} \partial_{x}+{v} \partial_{y}+{w} \partial_{z}\right) {u}=\partial_{z}^{2} {u}, \\
&\partial_{t} {v}+\left({u} \partial_{x}+{v} \partial_{y}+{w} \partial_{z}\right) {v}=\partial_{z}^{2} {v}, \\
&\partial_{x} {u}+\partial_{y} {v}+\partial_{z} { w}=0,\\
&({u}, {v}, {w})\big|_{z=0}=0, \quad \lim_{z \rightarrow+\infty}({u}, {v})=(0, 0),\\
&(u,v)\big|_{t=0}=(u_0,v_0).
\eali
\rt.
\ee

Before stating the main result of this paper, we need to introduce some notations especially for the Gevrey-2 space in which the solution $(u,v)$ lie. We use $\|f(t)\|_{L^p}$ for $1\leq p\leq +\i$ to denote the usual spacial $L^p$ norm in $\bR^3_+$ and $\|f(t)\|_{L^p(\o)}$ to denote the weighted $L^p$ norm for $1\leq p\leq+\i$ with the weight $\o$ as follows.
\bes
 \|f(t)\|^p_{L^p(\o)}:=\int_{\bR^3_+} |f(t,x,y,z)|^p\omega dxdydz.
\ees
\begin{definition}\label{defgsigma}
 For $\nu\in (0,1]$, denote $\th_{\nu}(t,z):=\exp\lt\{\nu\f{z^2}{8\angt}\rt\}$. Let $\al=(\al_1,\al_2)\in\bN^2$ be the two dimensional multi-index and denote $\p^\al_h=\p^{\al_1}_x\p^{\al_2}_y$. For a function $f(t,x,y,z)$, which are smooth in the tangential variables $(x,y)$, define the weighted Gevrey-$\sigma$ norm $\|\cdot\|_{G^\sigma_{\tau,\nu}}$ by
 \be\label{Gsigma}
\|f\|^2_{G^\sigma_{\tau,\nu}}:=\sum_{j\in\bN}\f{\tau^{2(j+1)}}{(j!)^{2\sigma}}\sup_{|\al|=j}\lt\|\th_{\nu}\p^\al_h f\rt\|^2_{L^2}=\sum_{j\in\bN}\f{\tau^{2(j+1)}}{(j!)^{2\sigma}}\sup_{|\al|=j}\lt\|\p^\al_h f\rt\|^2_{L^2(\th_{2\nu})}.
\ee
\end{definition}

We will look for solutions of \eqref{3dprandtl} in the $G^\sigma_{\tau,\nu}$ spaces defined above.  We need the initial data satisfying the following compatibility conditions at $z=0$.
\be\label{compa}
\p^{2k}_z(u_0,v_0)\big|_{z=0}=0,\q \text{for}\q k=0,1.
\ee

The main result of this paper is the following.
\begin{theorem}\label{thmain0}
  For any fixed $M\geq 2$, $\tau_0> 0$, suppose that $\p^k_z(u_0,v_0)$ belong to $G^2_{2\tau_0,1}$ for $k=0,1,2,3$, and satisfy the compatibility condition \eqref{compa}. Then there exist three constants $c_{\tau_0,M}$, $C_{\tau_0,M}$ and $\e_0$, such that for any $0<\e\leq \e_0$, if
\be\label{initialdata}
\sum^3_{k=0}\|\p^k_z(u_0,v_0)\|_{G^2_{2\tau_0,1}}\leq \e,
\ee
then for any $t\in(0,c_{\tau_0,M}\e^{-M}]$, system \eqref{3dprandtl1} has a unique smooth solution satisfying,
\be\label{solutionesti}
\sum^3_{k=0}\angt^{\f{k-1}{2}}\|\p^k_z(u,v)(t)\|_{G^2_{\f{1}{2}\tau_0,\f{1}{2}}}\leq  C_{\tau_0,M}\angt^{-\f{10M-1}{8M}}.
\ee
Here, $c_{\tau_0,M}$ and $C_{\tau_0,M}$ are two constants depending only on $\tau_0$ and $M$, which are relatively small and large when $M$ approaches infinity.
\end{theorem}
Two remarks follow.
\begin{itemize}
\item Actually, the result in Theorem \ref{thmain0} can also be easily applied to the Gevrey-$\sigma$ space $G^\sigma_{\tau,\nu}$ with $\sigma\in [1,2]$. The proof is essentially the same with the case $\sigma=2$. For simplicity, we only show the optimal case $\sigma=2$ and leave the details to the interested reader.
\item Here in the three dimensional case, we can not obtain a global existence result in Gevrey-2 spaces as shown in Wang, Wang and Zhang \cite{WWZ:2021ARXIV} for the two dimensional Prandtl equations. The reason is that in three dimensional case, the decay rate of the lower order Gevrey-2 norms of the tangential velocities is almost power $-3/4$, which is $1/2$-order slower than that of the two dimensional case.  This slower decay is not enough to ensure the global existence.
\end{itemize}

Proof of Theorem \ref{thmain0} is a direct consequence of Theorem \ref{thmain} stated in Section \ref{secmain}, which is a more precise and detailed version of our result. We will spare some time to give the proof of Theorem \ref{thmain0} when we finish the statement of Theorem \ref{thmain}. \qed

\section{Notations and detailed statement of the main theorem} \label{secmain}

\subsection{Notations}

For $0<\kappa$ and a $t$-dependent function $\tau=\tau(t)>0$,  define
\bes
M_{j,\kappa}:= \f{\tau^{j+1}(j+1)^\kappa}{(j!)^2}.
\ees
 For a function $f$ and $j\in\bN$, define
\bes
f_{j,\kappa,x}:=M_{j,\kappa} \p^j_x f, \q {f}_{j,\kappa,y}:=M_{j,\kappa} \p^j_y f, \q \text{and}\q f_{\al,\kappa}:=M_{|\al|,\kappa} \p^\al_h f.
\ees

Let $\th_{\nu}(t,z)$ be the weighted function in Definition \ref{defgsigma} and we simply denote $\th_1$ by $\th$. It is easy to see that for $\al,\,\beta\in\bR$, $\th_{\al+\beta}=\th_\al\cdot \th_\beta$. Now for $\nu\in(0,1]$, define another weighted Gevrey-2 norm $\|\cdot\|_{X_{\tau,\kappa,\nu}}$ as following.
\be\label{Gevnorm2}
\|f\|^2_{X_{\tau,\kappa,\nu}}=\sum_{j\in\bN}\sup_{|\al|=j}\|f_{\al,\kappa}\th_{\nu}\|^2_{L^2}
:=\sum_{j\in\bN}\sup_{|\al|=j}\|f_{\al,\kappa}\|^2_{L^2(\th_{2\nu})}.
\ee
When $\nu=1$, we abbreviate \eqref{Gevnorm2} as
\bes
\|f\|^2_{X_{\tau,\kappa}}:=\sum_{j\in\bN}\sup_{|\al|=j}\|f_{\al,\kappa}\|^2_{L^2(\th_2)}=\|f\|^2_{X_{\tau,\kappa,1}}.
\ees

Here we remark that the replaced Gevrey-2 norm $\|\cdot\|_{X_{\tau,\kappa,\nu}}$ is more suitable than the Gevrey-2 norm $\|\cdot\|_{G^2_{\tau,\nu}}$ for our later energy estimates and for proof of the theorem. Actually, it is easy to see that $\|\cdot\|_{G^2_{\tau,\nu}}=\|\cdot\|_{X_{\tau,0,\nu}}$.

By using Fourier transform on the tangential variables $x$ and $y$, we can see that
\bes
\f{1}{2}\lt(\|f_{j,\kappa,x}\|^2_{L^2(\th_{2\nu})}+\|f_{j,\kappa,y}\|^2_{L^2(\th_{2\nu})}\rt)\leq \sup_{|\al|=j}\|f_{\al,\kappa}\|^2_{L^2(\th_{2\nu})}\leq \|f_{j,\kappa,x}\|^2_{L^2(\th_{2\nu})}+\|f_{j,\kappa,y}\|^2_{L^2(\th_{2\nu})}.
\ees
So the $\|\cdot\|_{X_{\tau,\kappa,\nu}}$ norm defined in \eqref{Gevnorm2} is equivalent to the following
\bes
\|f\|^2_{\widetilde{X}_{\tau,\kappa,\nu}}:=\sum_{j\in\bN}\lt(\|f_{j,\kappa,x}\|^2_{L^2(\th_{2\nu})}+\|f_{j,\kappa,y}\|^2_{L^2(\th_{2\nu})}\rt).
\ees

The equivalent norms $\|f\|^2_{{X}_{\tau,\kappa}}$ and $\|f\|^2_{\widetilde{X}_{\tau,\kappa}}$ will be used alternatively throughout the rest of this paper.

Next, we will give the Gevrey radius $\tau(t)$ in the definition of the Gevrey-2 norm in \eqref{Gevnorm2}, which has a positive lower bound in our constrained time interval.

{\noindent\bf Choosing of the Gevrey radius $\tau$}.

For any fixed $\dl\in (0,\f{1}{100}]$ and $\tau_0>0$, we choose
\be\label{gevreyradius}
\tau(t):=\tau_0-\la\dl^{-1}\s{\e}\tau_0\lt(\angt^{\dl}-1\rt),
\ee
where $\la$ is a large constant, independent of $\e$ and will be determined later. For sufficiently small $\e$, we assume that $\la\dl^{-1}\s{\e}<1/2$ and set
\be\label{timelifespan}
 t \leq \lt(\f{1}{2\la\dl^{-1}\s{\e}}\rt)^{\f{1}{\dl}}:=T_0.
\ee
Under the constraint \eqref{timelifespan}, we can obtain
\be\label{timelifespan0}
\angt^\dl\leq \lt(1+\f{1}{2\la\dl^{-1}\s{\e}}\rt)\leq \f{1}{\la\dl^{-1}\s{\e}},
\ee
which indicate that
\be\label{radiequi}
\f{1}{2}\tau_0\leq \tau(t)\leq \tau_0.
\ee
Taking $t$ derivative of \eqref{gevreyradius} shows that
\bes
{\tau}'(t)=-\la\s{\e} \tau_0 \angt^{\dl-1}.
\ees
Denote $\la\s{\e}\eta(t):=-\f{\tau'(t)}{\tau(t)}$, then we have
\be\label{radigain}
 \angt^{\dl-1} \leq \eta(t)\leq 2\angt^{\dl-1}.
\ee
Next, we will restrict $t$ to be any time according to \eqref{timelifespan}.
\begin{remark}
 The constant $\la$ will be chosen to be dependent on $\dl$ and $\tau_0$ and will  approach to infinity as $\dl\rightarrow 0$. So there exists a constant $c_\dl$ such that
 \bes
 T_0=c_\dl \e^{-\f{1}{2\dl}}.
 \ees
\end{remark}

Throughout the paper, $C_{a,b,c,...}$ denotes a positive constant depending on $a,\,b,\, c,\,...$ which may be different from line to line. Dependence on the initial Gevrey radius $\tau_0$ is default, we will denote $C_{\tau_0}$ by $C$ for simplicity. We also apply $A\lesssim_{a,b,c,\cdots} B$ to denote $A\leq C_{a,b,c,...}B$. For a two dimensional multi-index $\al=(\al_1,\al_2)\in\bN^2$, we write $\p^\al_{h}=\p^{\al_1}_{x}\p^{\al_2}_{y}$ and
$\p^k_h=\{\p^\al_h\big | |\al|=k\}$. For a norm $\|\cdot\|$, we use $\|(f,g,\cdots)\|$ to denote $\|f\|+\|g\|+\cdots$. For a function $f(t,x,y,z)$ and $1\leq p,q \leq +\i$,  define
\bes
\|f\|_{L^p_hL^q_z}:=\lt(\int^{+\i}_0 \lt(\int_{\bR^2} |f|^p dxdy\rt)^{q/p}dz \rt)^{1/q}.
\ees
If $p=q$, we simply write it as $\|f\|_{L^p}$ and besides, if $p=q=2$, we will simply denote it as $\|f\|$. We use $[A, B]=AB-BA$ to denote the commutator of $A$ and $B$. $\langle\cdot , \cdot\rangle_{\o}$ denote weighted $L^2$ inner product with respect to spacial variables, which means for $f$ and $g$
\[\langle f , g\rangle_{\o}:=\int_{\bR^3_+} f(x,y,z) g(x,y,z)\o dxdydz.
\]

\subsection{Detailed statement of the main theorem}

Before presenting the more precise and detailed version of the main theorem, we need to introduce two good unknowns $(g,\t{g})$, which are set to control the lower order Gevrey-2 norms of $(u,v)$.  Define
\bes
g:=\p_z u+\f{z}{2\langle t\rangle} u, \q \t{g}:=\p_z v+\f{z}{2\langle t\rangle} v,
\ees
then we have the following theorem.
\begin{theorem}\label{thmain}

For any fixed $\tau_0>0$, $\dl\in(0,\f{1}{100}]$, there exist constants  $c_{\dl}$, $C_{\dl}$ and $\e_0$, such that for any $\e\leq \e_0$, if
\begin{align}
&\|(u,v)(0)\|_{X_{\tau_0, 13+\f{2}{\dl}}}\leq \e,\label{conditionu}\\
&\|(g,\t{g})(0)\|_{X_{\tau_0,12}}+\s{\dl}\|\p_z(g,\t{g})(0)\|_{X_{\tau_0,10}}+\dl\|\p^2_z(g,\t{g})(0)\|_{X_{\tau_0,8}}\leq \e,\label{conditiong}
\end{align}
then system \eqref{3dprandtl1} have a solution $(u,v,w)$ satisfying for any $t\in \lt[0,c_{\dl}\e^{-\f{1}{2\dl}}\rt]$, such that
\begin{align}
&\angt^{\f{1-\dl}{4}}\|(u,v)(t)\|_{X_{\tau,12+\f{2}{\dl}}}\leq C_{\dl}\e,\label{uvesti}\\
&\|(g,\t{g})(t)\|_{X_{\tau,12}}+\s{\dl}\angt^{1/2}\|\p_z (g,\t{g})(t)\|_{X_{\tau,10}}+\dl\angt\|\p^2_z(g,\t{g})(t)\|_{X_{\tau,8}}\leq C_{\dl}\e \angt^{-\f{5-\dl}{4}}.\label{gesti}
\end{align}
Here $\e_0$, $c_{\dl}$ and $C_{\dl}$ are three constants, depending on $\tau_0$ and $\dl$. We write $c_{\dl}$ and $C_{\dl}$ with $\dl$ subscript to emphasize its dependence on $\dl$.

\end{theorem}

\begin{remark}
 This result shows that for any $M>1$, by choosing sufficiently small $\dl$, the lifespan of the solution can be size of $\e^{-M}$ if the initial data is of size $\e$. Here, The constant $c_{\dl}$ is small while $C_{\dl}$ is large with respect to $\dl$. Actually in our proof, we will see that
\bes
\lim_{\dl\rightarrow 0} c_{\dl}=0,\q \lim_{\dl\rightarrow 0} C_{\dl}=+\i.
\ees
This is the main obstacle which prevent us to obtain the almost global existence of Gevrey-2 solutions.
\end{remark} \qed
\begin{remark}
For the proof of Theorem \ref{thmain}, by using the local-wellposedness in \cite{LMY:2022CPAM} and continuity argument, we only need to show a closed a priori estimate in Gevrey-2 spaces, which will be achieved in Section \ref{secproof}.
\end{remark} \qed

{\noindent\bf Proof of Theorem \ref{thmain0} based on Theorem \ref{thmain}}

First we verify initial conditions \eqref{conditionu} and \eqref{conditiong} in Theorem \ref{thmain} based on the assumption \eqref{initialdata} in Theorem \ref{thmain0}.

Let us choose $\dl=(2M)^{-1}$ in Theorem \ref{thmain}. Then
\begin{align}
&\|(u_0,v_0)\|_{X_{\tau_0, 13+\f{2}{\dl}}}=\lt(\sum_{j\in\bN}\sup_{|\al|=j} M^2_{j,13+4M}\|e^{\f{z^2}{8}}\p^\al_h(u_0,v_0)\|^2_{L^2(\bR^3_+)}\rt)^{1/2}\nn\\
=& \lt[\sum_{j\in\bN}\lt(\f{\tau_0^{(j+1)}(j+1)^{13+4M}}{(j!)^2}\sup_{|\al|=j}\|e^{\f{z^2}{8}}\p^\al_h(u_0,v_0)\|_{L^2(\bR^3_+)}\rt)^2\rt]^{1/2} \text{ using \eqref{Gsigma}}\nn\\
\leq& \lt[\sum_{j\in\bN}\lt(\f{(j+1)^{13+4M}}{2^{j+1}}\rt)^2\rt]^{1/2}G^2_{2\tau_0,1} \leq C_M \e. \text{ using \eqref{initialdata}}\nn
\end{align}
Moreover, by using inequalities \eqref{poincare} and \eqref{poincare1} in Lemma \ref{poincare}, we obtain that by using \eqref{initialdata}
\begin{align}
&\|(g,\t{g})(0)\|_{X_{\tau_0,12}}+(2M)^{-1/2}\|\p_z(g,\t{g})(0)\|_{X_{(\tau_0,10}}+(2M)^{-1}\|\p^2_z(g,\t{g})(0)\|_{X_{\tau_0,8}}\nn\\
\leq& C\lt(\|\p_z(u_0,v_0)\|_{X_{\tau_0,12}}+(2M)^{-1/2}\|\p^2_z(u_0,v_0)\|_{X_{\tau_0,10}}+(2M)^{-1}\|\p^3_z(u_0,v_0)\|_{X_{\tau_0,8}}\rt)\nn\\
\leq & C_M\sum^3_{k=1}\|\p^k_z(u_0,v_0)\|_{G^2_{2\tau_0,1}}\leq C_M\e. \nn
\end{align}
We have shown that \eqref{conditionu} and \eqref{conditiong} are guarantied by replacing $\e$ with $\t{\e}:=C_M\e$. Also from the proof of \eqref{egu1p}, \eqref{ug2}, \eqref{ug3} and \eqref{ug4} in Appendix, we see that
\begin{align}
&\sum^3_{k=0}\angt^{-\f{k-1}{2}}\|\p^k_z(u,v)(t)\|^2_{G^2_{\f{1}{2}\tau_0,\f{1}{2}}}=\sum^3_{k=0}\angt^{\f{k-1}{2}}\|\p^k_z(u,v)(t)\|^2_{X_{\f{1}{2}\tau_0,0,\f{1}{2}}}\nn\\
\leq& C_{\tau_0,M}\lt(\|(g,\t{g})(t)\|_{X_{\tau,12}}+\angt^{1/2}\|\p_z (g,\t{g})(t)\|_{X_{\tau,10}}+\angt\|\p^2_z(g,\t{g})(t)\|_{X_{\tau,8}}\rt). \label{estiequiv}
\end{align}
Then using the result of Theorem \ref{thmain}, we see that there exists three constants $c_{\tau_0,M}$ and $C_{\tau_0,M}$ and $\e_0$, such that,
for any $t\in(0,c_{\tau_0,M}\e^{-M}]$, system \eqref{3dprandtl1} has a unique smooth solution satisfying
\begin{align}
&\sum^3_{k=0}\angt^{-\f{k-1}{2}}\|\p^k_z(u,v)(t)\|^2_{G^2_{\f{1}{2}\tau_0,\f{1}{2}}}\leq C_{\tau_0,M}\t{\e} \angt^{-\f{5-(2M)^{-1}}{4}}=C_{\tau_0,M}{\e} \angt^{-\f{10M-1}{8M}},\label{estiequiv1}
\end{align}
which is \eqref{solutionesti}. At the last line of \eqref{estiequiv1}, we have used \eqref{estiequiv} and \eqref{gesti}. \qed

\section{Closed a priori estimates: proof to Theorem \ref{thmain}}\label{secproof}

\q\ In this section, we will give a closed a priori estimate in Gevrey-2 spaces, which indicates the validity of Theorem \ref{thmain} by combining the local well-posedness results and continuity argument. Our strategy is the following. First, in Section \ref{sec3.1}, we will introduce some auxiliary functions, which originate from the ones in \cite{DGV:2019ANNPDE} and \cite{LMY:2022CPAM}, where similar auxiliary functions are introduced to obtain local well-posedness of Gevrey-2 solutions for the two and three dimensional Prandtl equations. However, here we need some modifications so that they can be applied to obtain the long time behavior of the solution. Then in Section \ref{sec3.2}, we will introduce some linearly good unknowns, which are some linear combinations of the unknowns, auxiliary functions and their derivatives. They are set to achieve fast decay of lower order Gevrey-2 norms for the unknowns and auxiliary functions. In Section \ref{sec3.3}, we make a priori assumptions on the linearly good unknowns and based on the a priori assumptions, we will give a series of a priori estimates for the unknowns, the  auxiliary functions  and the linearly good unknowns in Section \ref{sec3.4}. At last, in Section \ref{sec3.5}, by applying the a priori estimates in Section \ref{sec3.4}, we can achieve closed energy estimates in Gevrey-2 spaces in the time interval $[0,T_0]$.

\subsection{Introduction of auxiliary functions}\label{sec3.1}

First we introduce the following two auxiliary functions $\mH$ and $\wt{\mathcal{H}}$  by

\be\label{auxih}
\lt\{
\begin{aligned}
&\lt[\partial_{t} +\left({u} \p_x+v \partial_y+w\p_z \right)-\partial_{z}^{2} \rt]\int^{+\i}_z \mathcal{H} d\bar{z}=\s{\e} \angt^{\dl-1} \p_x w,\\
& \mathcal{H}|_{t=0}=0,\q \p_z \mathcal{H}|_{z=0}=0,\q \mathcal{H}|_{z\rightarrow+\i}=0.
\end{aligned}
\rt.
\ee
\be\label{auxith}
\lt\{
\bali
&\lt[\partial_{t} +\left({u} \p_x+v \partial_y+w\p_z \right)-\partial_{z}^{2} \rt]\int^{+\i}_z \mathcal{\wt{H}}d\bar{z}=\s{\e} \angt^{\dl-1} \p_y w,\\
& \mathcal{\wt{H}}|_{t=0}=0,\q \p_z \mathcal{\wt{H}}|_{z=0}=0,\q \mathcal{\wt{H}}|_{z\rightarrow+\i}=0.
\eali
\rt.
\ee

The existence of $\mH$ and $\wt{\mH}$ follows the standard linear parabolic theory.  This two auxiliary functions are inspired by Dietert and G\'{e}rard-Varet \cite{DGV:2019ANNPDE}  and  Li-Masmoudi-Yang \cite{LMY:2022CPAM} where similar auxiliary functions are constructed to prove the local well-posedness of the 2D and 3D Prandtl equations in Gevrey-2 spaces. The main differences are the following.

1. In \eqref{auxih} and \eqref{auxith}, we define the auxiliary functions by $\int^{+\i}_z \mathcal{H} d\bar{z}$ and $\int^{+\i}_z \wt{\mathcal{H}} d\bar{z}$ instead of $\int^{z}_0 \mathcal{H} d\bar{z}$ and $\int^{z}_0 \wt{\mathcal{H}} d\bar{z}$ respectively is to ensure that $\mH$ and $\wt{\mH}$ decay fast enough at $z$ infinity.

2. The time-dependent coefficient on the righthand of \eqref{auxih} and \eqref{auxith} $\s{\e} \angt^{\dl-1}$ is specially designed to match with $\eta(t)$ in \eqref{radigain}, which can ensure closing of Gevrey-2 energy defined for $\mH$ and $\wt{\mH}$.
\begin{remark}
Here we remark that
\be\label{antimh}
\int^\i_{0}\mH d\bar{z}=0.
\ee
Actually by letting $z=0$ in \eqref{auxih} and using the boundary condition of $\p_z \mH$ and $w$ on $z=0$, we can achieve that
\bes
\lt[\partial_{t} +\left({u} \p_x+v \partial_y \right)\rt]\int^{+\i}_0 \mathcal{H} d\bar{z}=0.
\ees
Combining the fact that $\mH|_{t=0}=0$, we can obtain \eqref{antimh}  from the above transport equation.
\end{remark}
However, as shown in  Li-Masmoudi-Yang in \cite{LMY:2022CPAM}, the above two auxiliary functions are not enough to show the well posedness of the 3D Prandtl equations in the Gevrey-2 space. More auxiliary functions are needed to seek for new cancellations to overcome the one order derivative loss problem for the 3D much more complicated couple system. We introduce the following other four auxiliary functions.
\bes
\lt\{
\bali
&\mG:=\p_x u+\f{\angt^{1-\dl}}{\s{\e} } \p_z u\int^{+\i}_z \mathcal{H}d\bar{z},\q \wt{\mG}:=\p_y u+\f{\angt^{1-\dl}}{\s{\e} } \p_z u\int^{+\i}_z \wt{\mathcal{H}}d\bar{z},\\
&\mK:=\p_x v+\f{\angt^{1-\dl}}{\s{\e} } \p_z v\int^{+\i}_z \mathcal{H} d\bar{z},\q \wt{\mK}:=\p_y v+\f{\angt^{1-\dl}}{\s{\e} } \p_z v\int^{+\i}_z \wt{\mathcal{H}}d\bar{z},
\eali
\rt.
\ees

These auxiliary functions are initiated by  Li, Masmoudi and Yang in \cite{LMY:2022CPAM}, where similar four auxiliary functions are introduced to seek for new cancellations and local well-posedness in Gevrey-2 energy spaces  are achieved by combining the aforementioned auxiliary functions $\mH$ and $\wt{\mH}$. Also to obtain the long time behavior of the solution, we make some modification for the original auxiliary functions in \cite{LMY:2022CPAM}.  These four auxiliary functions will help achieve the $\f{1}{2}-$order derivative loss of $\mH$ and $\wt{\mH}$ and enable us to close the long-time Gevrey-2 energy of $\mH$ and $\wt{\mH}$. Then the long-time Gevrey-2 energy for the solution $(u,v)$ will be obtained with the help of $\mH$ and $\wt{\mH}$. Here we make a brief introduction for the idea of proof.

Take \eqref{auxih} for example. Applying $-\p_z$ to \eqref{auxih}, we can obtain, as shown in \eqref{auxih1}, that
\begin{align}
&\lt[\partial_{t} +\left({u} \p_x+v \partial_y+w\p_z \right)-\partial_{z}^{2} \rt] \mathcal{H}\nn\\
=&\s{\e} \angt^{\dl-1}   \lt(\p_x\mathcal{G}+\p_y\mathcal{K}\rt)-(\p_z\p_xu+\p_z\p_yv)\int^{+\i}_z \mathcal{H}d\bar{z}+(\p_xu+\p_yv)\mathcal{H}\label{preh1}\\
:=&\s{\e} \angt^{\dl-1}  \lt(\p_x\mathcal{G}+\p_y\mathcal{K}\rt)+\text{l.o.t.}.\nn
\end{align}
Here from the equation of $\mH$ and previous results in \cite{DGV:2019ANNPDE}, we consider $\mH$ have the same order as $\p_h (u,v)$, so the term l.o.t doesn't have derivative loss. However, if we view $\p_x\mG$  separately as two terms, then there will be one order derivative loss for $\p^2_x u$ in $\p_x\mG$. By computation, we can see that

\bes
\bali
&\lt[\partial_{t} +\left({u}\partial_{x}+v \partial_{y}+w\p_z\right) -\partial_{z}^{2} \rt]\p_x\mathcal{G}\\
=&-\lt[\p_xu\p_x\mG+\p_xv\p_y\mG+\p_x w\p_z \mG\rt]\\
&+\p_x\lt\{-\lt[(\p_xu)^2+\p_xv\p_yu\rt]+\f{\p_yv\p_zu-\p_yu\p_zv}{\s{\e}\angt^{\dl-1}}\int^\i_z\mathcal{H}d\bar{z}+\f{2\p^2_zu}
{\s{\e}\angt^{\dl-1}}\mathcal{H}\rt\}\\
:=&\text{terms involving }\p^2_h(u,v)+ \text{l.o.t.}.
\eali
\ees
The same is for $\p_y \mK$.  Then inserting this into \eqref{preh1}, we can see that

\be\label{mHgevrey}
\lt[\partial_{t} +\left({u} \p_x+v \partial_y+w\p_z \right)-\partial_{z}^{2} \rt]^2\mathcal{H}=\text{terms involving } \p^2_h(u,v)+\text{l.o.t.}.
\ee
Here $\p^2_h(u,v)$ have the same order with $\p_h \mH$. The above equation for $\mH$ indicates that we can perform energy estimates in Gevrey-$\sigma$ spaces with $\sigma\in[1,2]$ for $\mH$ as indicated in the toy model displayed in Li, Masmoudi and Yang \cite{LMY:2022CPAM}. While $\mG$ has the same order as $\p^{-1/2}_{h}\mH$. So if we define the energy functional of $\mH,\, \wt{\mH}$ as $\|(\mH,\wt{\mH})\|_{X_{\tau,\kappa}}$, then correspondingly, we need to define the
energy functionals of $\mG,\, \wt{\mG},\,\mK,\, \wt{\mH} $ as $\|(\mG,\, \wt{\mG},\,\mK,\, \wt{\mH})\|_{X_{\tau,\kappa+1}}$ and also  the energy functionals of $u,\, v$ as $\|(u,\, v)\|_{X_{\tau,\kappa+2}}$.

When we obtain from \eqref{mHgevrey} the Gevrey-2 energy norm estimate, from \eqref{uauxi}, we see that
\bes
\bali
&\lt[\partial_{t} +\left({u}\partial_{x}+v \partial_{y}+w\p_z\right) -\partial_{z}^{2} \rt]\lt(\p^j_xu+\f{\p_zu}{\s{\e}\angt^{\dl-1}} \int^\i_z \p^{j-1}_x\mathcal{H}d\bar{z}\rt)=\text{l.o.t.}.
\eali
\ees
No derivative loss is for the equation. After performing Gevrey-2 norm estimate for the above equation and then combining the Gevrey-2 energy estimates of $\mH$ and $\wt{\mH}$,  the Gevrey-2 norm estimate of $(u,v)$ follows.

In order to obtain much long lifespan of the solutions, the equations of $(u,v)$ and $(\mH,\wt{\mH})$ are not enough to obtain much faster time decay estimate. Next, we will introduce the following four linearly good unknowns to catch much faster decay to the lower order Gevrey-2 energy of the $(u,v)$ and $(\mH,\wt{\mH})$.

\subsection{The linearly good unknowns}\label{sec3.2}

Inspired by the good unknown in \cite{IV:2016ARMA} and \cite{PZ:2021ARMA}, we define
\be\label{defng}
\lt\{
\bali
&g:=\p_z u+\f{z}{2\langle t\rangle} u, \q \t{g}:=\p_z v+\f{z}{2\langle t\rangle} v,\\
& \fH:=\mH-\f{z}{2\angt}\int^\i_z\mH d\bar{z},\q \wt{\fH}:=\wt{\mH}-\f{z}{2\angt}\int^\i_z\wt{\mH} d\bar{z}.
\eali
\rt.
\ee
These four linearly good unknowns are set to dig out the sufficiently fast decay rate for the lower order Gevrey-2 norms of the solutions $(u,v)$ and $(\mH,\t{\mH})$, which ensure closing of energy estimates for all the quantities mentioned above in our constrained time $t\in (0,T_0]$. As shown in \eqref{gesti}, we see that the lower order Gevrey-2 norm of $(g,\t{g})$ has a decay rate of almost $-5/4$ order with respect to time, which will induce almost
$-3/4$ order decay rate of the lower order Gevrey-2 norm of $(u,v)$, see \eqref{solutiongevrey} in Lemma \ref{llowpoint}. Based on the almost
$-3/4$ order decay of $(u,v)$, we can see that $(\fH, \t{\fH})$ have the same almost $-3/4$ order decay for the lower order Gevrey-2 norm, which indicates almost $-3/4$ order decay of the lower order Gevrey-2 norm for the auxiliary functions $(\mH,\wt{\mH})$. See also \eqref{solutiongevrey} in Lemma \ref{llowpoint}.

%

\subsection{A prior assumptions}\label{sec3.3}

\q\ Later for simplification of notations, we denote
\bes
\bl{u}=(u,v),\q \bl{\mH}=(\mH,\wt{\mH}),\q \bl{\mG}=(\mG,\wt{\mG},\mK,\wt{\mK}),\q \bl{g}=(g,\t{g}),\q \bl{\fH}=(\fH,\wt{\fH}).
\ees

We will first make a priori assumptions for the good unknowns as follows. We assume that there exists some constant $C_\ast$, (to be determined later), such that
\be\label{gassump}
\bali
&\|\bl{g}(t)\|_{X_{\tau,12}}+\s{\dl}\angt^{1/2}\|\p_z \bl{g}(t)\|_{X_{\tau,10}}+\dl\angt\|\p^2_z\bl{g}(t)\|_{X_{\tau,8}}\leq C_\ast\e \angt^{-\f{5-\dl}{4}},\\
&\|\bl{\fH}(t)\|_{X_{\tau,9,7/8}}+\angt^{1/2}\|\p_z\bl{\fH}(t)\|_{X_{\tau,7,7/8}}\leq C_\ast\e \angt^{-\f{3-\dl}{4}}.
\eali
\ee

Under the a prior assumption \eqref{gassump}, we first have the following a priori estimates based on the relations between $\bl{u}$ and $\bl{g}$, and between $\bl{\mH}$ and $\bl{\fH}$ respectively.

\begin{lemma}\label{llowpoint}

Under the assumption \eqref{gassump}, we have the following a priori estimates. For any $0\leq\nu<1$,
{\small
\be\label{solutiongevrey}
\bali
&\angt^{-1/2}\lt\|\bl{u}\rt\|_{X_{\tau,12,\nu}}+\lt\|\p_z \bl{u}\rt\|_{X_{\tau,12,\nu}}+\s{\dl} \angt^{\f{1}{2}}\lt\|\p^2_z\bl{u}\rt\|_{X_{\tau,10,\nu}}
+\dl \angt\lt\|\p_z^3\bl{u}\rt\|_{X_{\tau,8,\nu}}\ls_\nu C_\ast\e \angt^{-\f{5-\dl}{4}},\\
&\lt\|\bl{\mH}\rt\|_{X_{\tau,9,3/4}}+\angt^{1/2}\lt\| \p_z\bl{\mH}\rt\|_{X_{\tau,7,3/4}}\ls C_\ast\e \angt^{-\f{3-\dl}{4}},\\
&\lt\|\bl{\mG}\rt\|_{X_{\tau,9,3/4}}+\s{\dl}\angt^{1/2}\lt\|\p_z\bl{\mG}\rt\|_{X_{\tau,7,3/4}}\ls C_\ast\e \angt^{-\f{3-\dl}{4}}.
\eali
\ee
}
Sobolev embedding will imply the following finite order $L^\i$ a priori estimates. For $k\leq 50$, any $0\leq\nu<1$,
\be\label{lowpoint}
\bali
&\angt^{-\f{1}{4}}\lt\|\th_{\nu} \p^k_h \bl{u}\rt\|_{L^\i}+\s{\dl}\angt^{\f{1}{4}}\lt\|\th_{\nu} \p^k_h \p_z \bl{u}\rt\|_{L^\i}+\dl\angt^{\f{3}{4}}\lt\|\th_{\nu} \p^k_h \p^2_z \bl{u}\rt\|_{L^\i}\ls C_\ast\e \angt^{-\f{5-\dl}{4}},\\
&\angt^{-\f{3}{4}}\lt\|\th_{\nu} \p^k_h {w}\rt\|_{L^\i} \ls C_\ast\e\angt^{-\f{5-\dl}{4}}.
\eali
\ee

\end{lemma}

Proof of this Lemma will be presented in Appendix. \qed

Based on the a priori assumptions in \eqref{gassump} and the a priori estimates in Lemma \ref{llowpoint}, we can derive a series of estimates as follows, which is based on performing weighted energy estimates to equations of auxiliary functions, the unknowns and the good unknowns.

\subsection{A priori estimates of auxiliary functions, the unknowns and the good unknowns}\label{sec3.4}

\q\ For simplification of notations, let $\kappa=10+\f{2}{\dl}$ in the following. For auxiliary functions $\bl{\mH}$, we have the following estimate.
\begin{proposition}[Gevrey-2 estimates of $\bl{\mH}$]\label{pfirstauxi}

For any fixed $\tau_0>0$, $\dl\in(0,\f{1}{100}]$, under the assumption of \eqref{gassump}, for sufficiently small $\e$, there exits a constant $C_\dl$ such that for any $t\in (0,T_0]$, we have the following estimate
\begin{align}
&\angt^{\f{1-\dl}{2}} \lt\|\bl{\mH}(t)\rt\|^2_{X_{\tau,\kappa}}+\dl\int^{T_0}_0\angt^{\f{1-\dl}{2}}\lt\|\p_z\bl{\mH}(t)\rt\|^2_{X_{\tau,\kappa}}dt+\la\s{\e} \int^{T_0}_0\angt^{\f{1-\dl}{2}}{\eta}(t) \lt\|\bl{\mH}(t)\rt\|^2_{X_{\tau,\kappa+1/2}}dt\nn\\
\leq  & C_\dl\f{C^2_\ast}{\la}\s{\e} \int^{T_0}_0 \angt^{-\f{1-\dl}{2}} \lt(\lt\|\bl{\mH}(t)\rt\|^2_{X_{\tau,\kappa+1/2}}+\lt\|\bl{u}(t)\rt\|^2_{X_{\tau,\kappa+5/2}} \rt)dt\label{mhest}\\
 &+C_\dl\f{C^2_\ast}{\la}\s{\e} \int^{T_0}_0 \angt^{\f{1-\dl}{2}} \lt(\lt\|\p_z\bl{\mH}(t)\rt\|^2_{X_{\tau,\kappa}}+\lt\|\p_z\bl{u}(t)\rt\|^2_{X_{\tau,\kappa+2}} \rt)dt\nn\\
 &+C_\dl \s{\e} \int^{T_0}_0\angt^{\f{1-\dl}{2}}{\eta}(t) \lt\|\bl{\mG}(t)\rt\|^2_{X_{\tau,\kappa+3/2}}dt.\nn
\end{align}
\end{proposition}
For auxiliary functions $\bl{\mG}$, we have the following estimate.
\begin{proposition}[Gevrey-2 estimates of $\bl{\mG}$ ]\label{psecondauxi}

For any fixed $\tau_0>0$, $\dl\in(0,\f{1}{100}]$, under the assumption of \eqref{gassump}, for sufficiently small $\e$, there exits a constant $C_\dl$ such that for any $t\in (0,T_0]$, we have the following estimate
\begin{align}
&\angt^{\f{1-\dl}{2}} \lt\|\bl{\mG}(t)\rt\|^2_{X_{\tau,\kappa+1}}+\dl\int^{T_0}_0\angt^{\f{1-\dl}{2}}\lt\|\p_z\bl{\mG}(t)\rt\|^2_{X_{\tau,\kappa+1}}dt+\la\s{\e} \int^{T_0}_0\angt^{\f{1-\dl}{2}}{\eta}(t) \lt\|\bl{\mG}(t)\rt\|^2_{X_{\tau,\kappa+3/2}}dt\nn\\
\leq  &C_\dl\|\bl{u}(0)\|_{X_{\tau,\kappa+3}}+C_\dl\f{C^2_\ast}{\la}\s{\e} \int^{T_0}_0 \angt^{-\f{1-\dl}{2}} \lt(\lt\|\bl{\mH}(t)\rt\|^2_{X_{\tau,\kappa+1/2}}+\lt\|\bl{\mG}(t)\rt\|^2_{X_{\tau,\kappa+3/2}}+\lt\|\bl{u}(t)\rt\|^2_{X_{\tau,\kappa+5/2}} \rt)dt\label{mgest}\\
 &+C_\dl\f{C^2_\ast}{\la}\s{\e} \int^{T_0}_0 \angt^{\f{1-\dl}{2}} \lt(\lt\|\p_z\bl{\mH}(t)\rt\|^2_{X_{\tau,\kappa}}+\lt\|\p_z\bl{\mG}(t)\rt\|^2_{X_{\tau,\kappa+1}}+\lt\|\p_z\bl{u}(t)\rt\|^2_{X_{\tau,\kappa+2}} \rt)dt.\nn
\end{align}
\end{proposition}
For the unknown functions $\bl{u}$, we have the following estimate.
\begin{proposition}[Gevrey-2 estimates of $\bl{u}$]\label{punknown}
For any fixed $\tau_0>0$, $\dl\in(0,\f{1}{100}]$, under the assumption of \eqref{gassump}, for sufficiently small $\e$, there exits a constant $C_\dl$ such that for any $t\in (0,T_0]$, we have the following estimate
\begin{align}
&\angt^{\f{1-\dl}{2}} \lt\|\bl{u}(t)\rt\|^2_{X_{\tau,\kappa+2}}+\dl\int^{T_0}_0\angt^{\f{1-\dl}{2}}\lt\|\p_z\bl{u}(t)\rt\|^2_{X_{\tau,\kappa+2}}dt+\la\s{\e} \int^{T_0}_0\angt^{\f{1-\dl}{2}}{\eta}(t) \lt\|\bl{u}(t)\rt\|^2_{X_{\tau,\kappa+5/2}}dt\nn\\
\leq  &C_\dl\|\bl{u}(0)\|_{X_{\tau,\kappa+2}}+C_\dl C^2_\ast\|\bl{\mH}\|^2_{X_{\tau,\kappa}}+C_\dl\f{C^2_\ast}{\la}\s{\e} \int^{T_0}_0 \angt^{-\f{1-\dl}{2}} \lt(\lt\|\bl{u}(t)\rt\|^2_{X_{\tau,\kappa+5/2}}+\lt\|\bl{\mH}(t)\rt\|^2_{X_{\tau,\kappa+1/2}}\rt)dt\label{uesti}\\
     &+C_\dl\f{C^2_\ast}{\la}\s{\e} \int^{T_0}_0 \angt^{\f{1-\dl}{2}} \lt(\lt\|\p_z\bl{u}(t)\rt\|^2_{X_{\tau,\kappa+2}}+\lt\|\p_z\bl{\mH}(t)\rt\|^2_{X_{\tau,\kappa}}\rt)dt.\nn
\end{align}
\end{proposition}
Next, we will give the Gevrey-2 estimates of the  good unknowns $\bl{g}$ and $\bl{\fH}$. Denote
\bes
\kappa_0=12,\ \kappa_1=10,\ \kappa_2=8,\text{ and }, \kappa_3=9,\ \kappa_4=7.
\ees
For $\bl{g}$, we have the following estimate.
\begin{proposition}[Gevrey-2 estimates of $\bl{g}$]\label{pgoodunknown}

For any fixed $\tau_0>0$, $\dl\in(0,\f{1}{100}]$, under the assumption of \eqref{gassump}, for sufficiently small $\e$, there exits a constant $C_\dl$ such that for any $t\in (0,T_0]$, we have the following estimate.
\begin{itemize}
\item[i.] For the good unknown: $\bl{g}$,
\begin{align}
&\angt^{\f{5-\dl}{2}} \lt\| \bl{g} (t)\rt\|^2_{X_{\tau,\kappa_0}}+\dl\int^{T_0}_0\angt^{\f{5-\dl}{2}}\lt\|\p_z\bl{g}(t)\rt\|^2_{X_{\tau,\kappa_0}}dt+\la\s{\e} \int^{T_0}_0\angt^{\f{5-\dl}{2}}{\eta}(t) \lt\|\bl{g}(t)\rt\|^2_{X_{\tau,\kappa_0+1/2}}dt\nn\\
\leq & C_\dl \lt\| \bl{g} (0)\rt\|^2_{X_{\tau_0,\kappa_0}}+C_\dl \f{C^2_\ast}{\la} \e^{3/2}\int^{T_0}_0 \angt^{\f{1-\dl}{2}}\lt\|\p_z\bl{u}(t)\rt\|^2_{X_{\tau,\kappa+2}}dt\label{goodun1}\\
&  +C_\dl \f{C^2_\ast}{\la} \e^{3/2}\int^{T_0}_0 \lt(\angt^{\f{3+\dl}{2}}\lt\|\bl{g}(t)\rt\|^2_{X_{\tau,\kappa_0+1/2}}+\angt^{\f{5-\dl}{2}}\lt\|\p_z\bl{g}(t)\rt\|^2_{X_{\tau,\kappa_0}}\rt)dt.\nn
\end{align}
\item[ii.] For the first order $z$-derivative of the good unknown: $\p_z\bl{g}$,
\begin{align}
&\angt^{\f{7-\dl}{2}} \lt\| \p_z\bl{g}(t)\rt\|^2_{X_{\tau,\kappa_1}}+\dl\int^{T_0}_0\angt^{\f{7-\dl}{2}}\lt\|\p^2_z\bl{g}(t)\rt\|^2_{X_{\tau,\kappa_1}}dt+\la\s{\e} \int^{T_0}_0\angt^{\f{7-\dl}{2}}{\eta}(t) \lt\|\p_z\bl{g}(t)\rt\|^2_{X_{\tau,\kappa_1+1/2}}dt\nn\\
\leq & C_\dl \lt\| \p_z\bl{g}(0)\rt\|^2_{X_{\tau,\kappa_1}}+\int^{T_0}_0 \angt^{\f{5-\dl}{2}}\lt\|\p_z\bl{g}(t)\rt\|^2_{X_{\tau,\kappa_1}}dt\label{goodun2}\\
      &+C_\dl\f{C^2_\ast}{\la} \e^{3/2}\int^{T_0}_0 \lt(\angt^{\f{3+\dl}{2}}\lt\|\bl{g}(t)\rt\|^2_{X_{\tau,\kappa_0+1/2}}+\angt^{\f{5-\dl}{2}}\lt\|\p_z\bl{g}(t)\rt\|^2_{X_{\tau,\kappa_0}}\rt)dt.\nn
\end{align}
\item[iii.] For the second order $z$-derivative of the good unknown: $\p^2_z\bl{g}$,
\begin{align}
&\angt^{\f{9-\dl}{2}} \lt\| \p^2_z\bl{g}(t)\rt\|^2_{X_{\tau,\kappa_2}}+\dl\int^T_0\angt^{\f{9-\dl}{2}}\lt\|\p^3_z\bl{g}(t)\rt\|^2_{X_{\tau,\kappa_2}}dt+\la\s{\e} \int^T_0\angt^{\f{9-\dl}{2}}{\eta}(t) \lt\|\p^2_z\bl{g}(t)\rt\|^2_{X_{\tau,\kappa_3+1/2}}dt\nn\\
\leq &C_\dl \lt\| \p^2_z\bl{g}(0)\rt\|^2_{X_{\tau_0,\kappa_2}}+\int^{T_0}_0 \angt^{\f{7-\dl}{2}}\lt\|\p^2_z\bl{g}(t)\rt\|^2_{X_{\tau,\kappa_2}}dt.\label{goodun3}\\
    &+C_\dl\f{C^2_\ast}{\la} \e^{3/2}\int^{T_0}_0 \lt(\angt^{\f{3+\dl}{2}}\lt\|\bl{g}(t)\rt\|^2_{X_{\tau,\kappa_0+1/2}}+\angt^{\f{5-\dl}{2}}\lt\|\p_z\bl{g}(t)\rt\|^2_{X_{\tau,\kappa_0}}++\angt^{\f{7-\dl}{2}}\lt\|\p^2_z\bl{g}(t)\rt\|^2_{X_{\tau,\kappa_1}}\rt)dt.\nn
\end{align}
\end{itemize}
\end{proposition}

For $\bl{\fH}$, we have the following estimate.
\begin{proposition}[Gevrey-2 estimates of $\bl{\fH}$]\label{pgoodunknown1}

For any fixed $\tau_0>0$, $\dl\in(0,\f{1}{100}]$, under the assumption of \eqref{gassump}, for sufficiently small $\e$, there exits a constant $C_\dl$ such that for any $t\in (0,T_0]$, we have the following estimate.
\begin{itemize}
\item[iv.] For the good unknown: $\bl{\fH}$,
\begin{align}
&\angt^{\f{3-\dl}{2}} \lt\| \bl{\fH} (t)\rt\|^2_{X_{\tau,\kappa_3,7/8}}+\int^{T_0}_0\angt^{\f{3-\dl}{2}}\lt\|\p_z\bl{\fH}(t)\rt\|^2_{X_{\tau,\kappa_3,7/8}}dt+\la\s{\e} \int^{T_0}_0\angt^{\f{3-\dl}{2}}{\eta}(t) \lt\|\bl{\fH}(t)\rt\|^2_{X_{\tau,\kappa_3+1/2,7/8}}dt\nn\\
\leq  &C_\dl \f{C^2_\ast}{\la} \e^{3/2}\int^{T_0}_0 \lt(\angt^{\f{1+\dl}{2}}\lt\|\bl{\fH}(t)\rt\|^2_{X_{\tau,\kappa_3+1/2,7/8}}+\angt^{\f{3-\dl}{2}}\lt\|\p_z\bl{\fH}(t)\rt\|^2_{X_{\tau,\kappa_3,7/8}}
+\angt^{-\f{1-\dl}{2}}\lt\|\bl{\mH}(t)\rt\|^2_{X_{\tau,\kappa}}\rt)dt\label{goodun4}\\
        &+\f{\s{\e}}{\la}\int^{T_0}_0 \angt^{\f{3+\dl}{2}}\lt\|\bl{g}(t)\rt\|^2_{X_{\tau,\kappa_0+1/2}}dt.\nn
\end{align}
\item[v.] For the first order $z$-derivative of the good unknown: $\p_z\bl{\fH}$,
\begin{align}
&\angt^{\f{5-\dl}{2}} \lt\| \p_z\bl{\fH}(t)\rt\|^2_{X_{\tau,\kappa_4,7/8}}+\int^{T_0}_0\angt^{\f{5-\dl}{2}}\lt\|\p^2_z\bl{\fH}(t)\rt\|^2_{X_{\tau,\kappa_4,7/8}}dt+\la\s{\e} \int^{T_0}_0\angt^{\f{5-\dl}{2}}{\eta}(t) \lt\|\p_z\bl{\fH}(t)\rt\|^2_{X_{\tau,\kappa_4+1/2,7/8}}dt\nn\\
\leq  &C_\dl \f{C^2_\ast}{\la} \e^{3/2}\int^{T_0}_0 \lt(\angt^{\f{1+\dl}{2}}\lt\|\bl{\fH}(t)\rt\|^2_{X_{\tau,\kappa_3+1/2,7/8}}+\angt^{\f{3-\dl}{2}}\lt\|\p_z\bl{\fH}(t)\rt\|^2_{X_{\tau,\kappa_3,7/8}}\rt)dt\label{goodun5}\\
        &+\f{\s{\e}}{\la}\int^{T_0}_0 \angt^{\f{3+\dl}{2}}\lt\|\bl{g}(t)\rt\|^2_{X_{\tau,\kappa_0+1/2}}dt+\int^{T_0}_0\angt^{\f{3-\dl}{2}}\lt\|\p_z\bl{\fH}(t)\rt\|^2_{X_{\tau,\kappa_4,7/8}}dt.\nn
\end{align}
\end{itemize}
\end{proposition}

\subsection{End proof of Theorem \ref{thmain}}\label{sec3.5}

\q\ Based on the a priori estimates from Proposition \ref{pfirstauxi} to Proposition \ref{pgoodunknown1}, we can derive the validity of Theorem \ref{thmain}. Since the local well-posedness of Gevrey-2 solutions has already been shown in Li-Masnoudi-Yang \cite{LMY:2022CPAM}, by continuity argument, we only need to show that under the a priori assumption \eqref{gassump}, by choosing suitably large $C_\ast$, we can obtain that
\be\label{gassumpf}
\bali
&\|\bl{g}(t)\|_{X_{\tau,12}}+\s{\dl}\angt^{1/2}\|\p_z \bl{g}(t)\|_{X_{\tau,10}}+\dl\angt\|\p^2_z\bl{g}(t)\|_{X_{\tau,8}}\leq \f{1}{2}C_\ast\e \angt^{-\f{5-\dl}{4}},\\
&\|\bl{\fH}(t)\|_{X_{\tau,9,7/8}}+\angt^{1/2}\|\p_z\bl{\fH}(t)\|_{X_{\tau,7,7/8}}\leq \f{1}{2}C_\ast\e \angt^{-\f{3-\dl}{4}}.
\eali
\ee
And to prove the validity of \eqref{uvesti} and \eqref{gesti}.

Set
\bes
\la=C^4_\ast.
\ees
For $C_\ast$ sufficiently large, adding \eqref{mhest} in Proposition \ref{pfirstauxi} and \eqref{mgest} in Proposition \ref{psecondauxi} together, we can obtain that  that
\begin{align}
&\angt^{\f{1-\dl}{2}} \lt\|\bl{\mH}(t)\rt\|^2_{X_{\tau,\kappa}}+\dl\int^{T_0}_0\angt^{\f{1-\dl}{2}}\lt\|\p_z\bl{\mH}(t)\rt\|^2_{X_{\tau,\kappa}}dt+C^4_\ast\s{\e} \int^{T_0}_0\angt^{\f{1-\dl}{2}}{\eta}(t) \lt\|\bl{\mH}(t)\rt\|^2_{X_{\tau,\kappa+1/2}}dt\nn\\
\ls_\dl  & \lt\|u(0)\rt\|^2_{X_{\tau_0,\kappa+3}}+C^{-2}_\ast\s{\e} \int^{T_0}_0 \lt(\angt^{-\f{1-\dl}{2}}\lt\|\bl{u}(t)\rt\|^2_{X_{\tau,\kappa+5/2}}+\angt^{\f{1-\dl}{2}}\lt\|\p_z\bl{u}(t)\rt\|^2_{X_{\tau,\kappa+2}}\rt)dt.\nn
\end{align}
For $C_\ast$ sufficiently large, combining the above inequality with \eqref{uesti} in Proposition \ref{punknown}, we can obtain that
\begin{align}
&\angt^{\f{1-\dl}{2}} \lt(\lt\|\bl{\mH}(t)\rt\|^2_{X_{\tau,\kappa}}+\lt\|\bl{u}(t)\rt\|^2_{X_{\tau,\kappa+2}}\rt)+\dl\int^{T_0}_0\angt^{\f{1-\dl}{2}}\lt(\lt\|\p_z\bl{\mH}(t)\rt\|^2_{X_{\tau,\kappa}}+\lt\|\p_z\bl{u}(t)\rt\|^2_{X_{\tau,\kappa+2}}\rt)dt\nn\\
&+C^4_\ast\s{\e} \int^{T_0}_0\angt^{\f{1-\dl}{2}}{\eta}(t) \lt(\lt\|\bl{\mH}(t)\rt\|^2_{X_{\tau,\kappa+1/2}}+\lt\|\bl{u}(t)\rt\|^2_{X_{\tau,\kappa+5/2}}\rt)dt\leq C_\dl \lt\|\bl{u}(0)\rt\|^2_{X_{\tau_0,\kappa+3}}\leq C_\dl\e^2. \label{combesti}
\end{align}
From the above inequality, estimate \eqref{uvesti} is proven.

Now multiplying a small constant $c\dl$ to \eqref{goodun2} and a much smaller constant $c\dl^2$ to \eqref{goodun3}, and adding the resulted equations to \eqref{goodun1}, by letting $\e$ be sufficiently small, we can achieve that
\begin{align}
&\angt^{\f{5-\dl}{2}} \lt\| \bl{g} (t)\rt\|^2_{X_{\tau,\kappa_0}}+\dl\angt^{\f{7-\dl}{2}} \lt\| \p_z\bl{g} (t)\rt\|^2_{X_{\tau,\kappa_1}}+\dl^2\angt^{\f{9-\dl}{2}} \lt\| \p^2_z\bl{g} (t)\rt\|^2_{X_{\tau,\kappa_2}}\nn\\
&C^4_\ast\s{\e} \int^{T_0}_0\angt^{\f{5-\dl}{2}}{\eta}(t) \lt\|\bl{g}(t)\rt\|^2_{X_{\tau,\kappa_0+1/2}}dt\leq C_\dl \e^{3/2}\int^{T_0}_0 \angt^{\f{1-\dl}{2}}\lt\|\p_z\bl{u}(t)\rt\|^2_{X_{\tau,\kappa+2}}dt. \label{combesti1}
\end{align}
Inserting \eqref{combesti} into \eqref{combesti1}, for sufficiently small $\e$, we can achieve that for some constant $C_\dl$,
\begin{align}
&\angt^{\f{5-\dl}{2}} \lt\| \bl{g} (t)\rt\|^2_{X_{\tau,\kappa_0}}+\dl\angt^{\f{7-\dl}{2}} \lt\| \p_z\bl{g} (t)\rt\|^2_{X_{\tau,\kappa_1}}+\dl^2\angt^{\f{9-\dl}{2}} \lt\| \p^2_z\bl{g} (t)\rt\|^2_{X_{\tau,\kappa_2}}\nn\\
&+\la\s{\e} \int^{T_0}_0\angt^{\f{5-\dl}{2}}{\eta}(t) \lt\|\bl{g}(t)\rt\|^2_{X_{\tau,\kappa_0+1/2}}dt\leq  C_\dl\e^2. \label{combesti2}
\end{align}
Then we can obtain \eqref{gesti} and the first one of \eqref{gassumpf} by letting $C_\ast$ large enough such that $16C_\dl\leq C^2_\ast$.

At last, from \eqref{goodun4} and \eqref{goodun5} in Proposition \ref{pgoodunknown1}, by letting $\e$ is sufficiently small, we can obtain that
\begin{align}
&\angt^{\f{3-\dl}{2}} \lt\| \bl{\fH} (t)\rt\|^2_{X_{\tau,\kappa_3,7/8}}+\angt^{\f{5-\dl}{2}} \lt\| \p_z\bl{\fH} (t)\rt\|^2_{X_{\tau,\kappa_4,7/8}}\nn\\
\leq  &C_\dl \e^{3/2}\int^{T_0}_0 \angt^{-\f{1-\dl}{2}}\lt\|\bl{\mH}(t)\rt\|^2_{X_{\tau,\kappa}}dt+C^{-4}_\ast\s{\e}\int^{T_0}_0 \angt^{\f{3+\dl}{2}}\lt\|\bl{g}(t)\rt\|^2_{X_{\tau,\kappa_0+1/2}}dt.\nn
\end{align}
By using \eqref{combesti} and \eqref{combesti2}, and remembering \eqref{radigain}, we can obtain there exists a constant $C_\dl$ such that
\begin{align}
&\angt^{\f{3-\dl}{2}} \lt\| \bl{\fH} (t)\rt\|^2_{X_{\tau,\kappa_3,7/8}}+\angt^{\f{5-\dl}{2}} \lt\| \p_z\bl{\fH} (t)\rt\|^2_{X_{\tau,\kappa_4,7/8}}\leq  C_\dl \e^2, \nn
\end{align}
which shows the second one of \eqref{gassumpf} by letting $C_\ast$ large enough.  \qed

We still need to show the proof of Proposition \ref{pfirstauxi} to Proposition \ref{pgoodunknown1}. Before that, we give two useful lemmas, which will be frequently used in later estimates.

\subsection{Preliminary lemmas}

The first one is a weighted Poinc\'{a}re inequality. We state as follows.
\begin{lemma}\label{lpoincare}
Let $f$ be a function belonging to $H^1$ in $z$ variable, which decays to zero sufficiently fast as $z\rightarrow +\i$. Then for $0\leq\nu\leq1$ we have
\be\label{poincare}
 \f{\nu}{2\angt}\|f\|^2_{L^2(\th_{2\nu})}\leq\|\p_zf\|^2_{L^2(\th_{2\nu})},
\ee
and
\be\label{poincare1}
\f{\nu}{4\angt}\|f\|^2_{L^2(\th_{2\nu})}+\f{\nu^2}{16}\lt\|\f{z}{\angt} f\rt\|^2_{L^2(\th_{2\nu})}\leq \| \p_zf\|^2_{L^2(\th_{2\nu})}.
\ee
\end{lemma}

Proof of this lemma is similar with \cite[Lemma 2.5]{WWZ:2021ARXIV}, where the case $\nu=1$ is handled. The essential idea is to performing integration by parts on $z$ variable. Here we make some details for completeness.

\begin{proof}
Since when $\nu=0$, \eqref{poincare} and \eqref{poincare1} are obvious, we will assume that $0< \nu\leq 1$ in the following. Using integration by parts on $z$, we have
\begin{align}
&\int^\i_0 f^2(x,y,z)e^{\f{\nu z^2}{4\angt}}dz\nn\\
=&-\int^\i_0 z\p_z\lt(f^2(x,y,z)e^{\f{\nu z^2}{4\angt}}\rt)dz\nn\\
=&-2\int^\i_0 zf(x,y,z)\p_zf(x,y,z)e^{\f{\nu z^2}{4\angt}}dz-\f{\nu}{2\angt}\int^\i_0 z^2f^2(x,y,z)e^{\f{\nu z^2}{4\angt}}dz.\nn
\end{align}
Then by using Cauchy inequality to the first term of the right hand of the above equality, we can obtain that
\begin{align}
&\int^\i_0 f^2(x,y,z)e^{\f{\nu z^2}{4\angt}}dz+\f{\nu}{2\angt}\int^\i_0 z^2f^2(x,y,z)e^{\f{\nu z^2}{4\angt}}dz\nn\\
\leq &\f{\nu}{2\angt}\int^\i_0 z^2f^2(x,y,z)e^{\f{\nu z^2}{4\angt}}dz+\f{2\angt}{\nu}\int^\i_0 (\p_zf)^2(x,y,z)e^{\f{\nu z^2}{4\angt}}dz,\nn
\end{align}
and
\begin{align}
&\int^\i_0 f^2(x,y,z)e^{\f{\nu z^2}{4\angt}}dz+\f{\nu}{2\angt}\int^\i_0 z^2f^2(x,y,z)e^{\f{\nu z^2}{4\angt}}dz\nn\\
\leq &\f{\nu}{4\angt}\int^\i_0 z^2f^2(x,y,z)e^{\f{\nu z^2}{4\angt}}dz+\f{4\angt}{\nu}\int^\i_0 (\p_zf)^2(x,y,z)e^{\f{\nu z^2}{4\angt}}dz.\nn
\end{align}
Then after absorbing the first term of the righthand of the above two inequalities, we can have
\begin{align}
\f{\nu}{2\angt}\int^\i_0 f^2(x,y,z)e^{\f{\nu z^2}{4\angt}}dz\leq\int^\i_0 (\p_zf)^2(x,y,z)e^{\f{\nu z^2}{4\angt}}dz, \label{ppoin1}
\end{align}
and
\begin{align}
&\f{\nu}{4\angt}\int^\i_0 f^2(x,y,z)e^{\f{\nu z^2}{4\angt}}dz+\lt(\f{\nu}{4\angt}\rt)^2\int^\i_0 z^2f^2(x,y,z)e^{\f{\nu z^2}{4\angt}}dz\nn\\
\leq &\int^\i_0 (\p_zf)^2(x,y,z)e^{\f{\nu z^2}{4\angt}}dz.  \label{ppoin2}
\end{align}
After integrating \eqref{ppoin1} and \eqref{ppoin2} on $x,y$, we can obtain \eqref{poincare1}.
\end{proof}

Next, we give a lemma to show  product estimates of the Gevrey-2 norm.
\begin{lemma}\label{lemproduct}
Let $\kappa>0$ and $0<\nu\leq 1$. For smooth functions $f$ and $g$, which decay fast enough at $z$ infinity, we have the following product estimates.
\begin{align}
&\|fg\|^2_{X_{\tau,\kappa,\nu}}\ls_{\nu,\kappa} \angt^{1/2}\lt(\|f\|^2_{X_{\tau,5,\f{\nu+1}{4}}}\|\p_z g\|^2_{X_{\tau,\kappa,\f{\nu+1}{4}}}+\|\p_z f\|^2_{X_{\tau,\kappa,\f{\nu+1}{4}}}\|g\|^2_{X_{\tau,5,\f{\nu+1}{4}}}\rt),\label{product1}\\
&\|fg\|^2_{X_{\tau,\kappa,\nu}}\ls_{\nu,\kappa}  \angt^{1/2}\lt(\|f\|^2_{X_{\tau,5,\f{\nu+1}{4}}}\|\p_z g\|^2_{X_{\tau,\kappa,\f{\nu+1}{4}}}+\| f\|^2_{X_{\tau,\kappa,\f{\nu+1}{4}}}\|\p_zg\|^2_{X_{\tau,5,\f{\nu+1}{4}}}\rt),\label{product2}\\
&\|fg\|^2_{X_{\tau,\kappa,\nu}}\ls_{\nu,\kappa} \angt^{1/2}\lt(\|\p_zf\|^2_{X_{\tau,5,\f{\nu+1}{4}}}\| g\|^2_{X_{\tau,\kappa,\f{\nu+1}{4}}}+\|f\|^2_{X_{\tau,\kappa,\f{\nu+1}{4}}}\|\p_zg\|^2_{X_{\tau,5,\f{\nu+1}{4}}}\rt).\label{product3}
\end{align}
\end{lemma}
\begin{proof}
We only present the proof of \eqref{product1}, since the other two are simiar. For simplicty, we write $f_{k,\kappa}$ to denote $f_{k,x,\kappa}$ or $f_{k,y,\kappa}$ and $\p^k_h$ to denote $\p^k_x$ or $\p^k_y$ if no confusion is caused. First, by using Leibniz formula, we see that

\begin{align}
(fg)_{j,\kappa}=&M_{j,\kappa}\p^j_h(fg)\nn\\
  =& \sum_{0\leq k\leq j }\f{M_{j,\kappa}}{M_{k,\kappa} M_{j-k,\kappa}}\lt(j\atop k\rt)|f_{k,\kappa}||g_{j-k,\kappa}|\nn\\
     =&\sum_{0\leq k\leq j }\f{1}{\tau}\lt(\f{j+1}{(k+1)(j-k+1)}\rt)^\kappa\lt(j\atop k\rt)^{-1}|f_{k,\kappa}||g_{j-k,\kappa}|.\nn
\end{align}
Then by using \eqref{radiequi}, we can obtain that
{\small
\be\label{keydiff}
\bali
|(fg)_{j,\kappa}|\ls &\sum^{[(j+1)/2]}_{k=0} (k+1)^{-\kappa} |f_{k,\kappa}| |g_{j-k,\kappa}|+\sum^{j}_{k=[(j+1)/2]+1}  (j-k+1)^{-\kappa} |f_{k,\kappa}| |g_{j-k,\kappa}|.
\eali
\ee
}
Using Minkowski inequality and H\"{o}lder inequality, we have
{\small
\bes
\bali
\|(fg)_{j,\kappa}\|_{L^2(\th_{2\nu})} \leq &\sum^{[(j+1)/2]}_{k=0} \|(k+1)^{-\kappa}f_{k,\kappa}g_{j-k,\kappa}\|_{L^2(\th_{2\nu})}+\sum^{j}_{k=[(j+1)/2]+1}  \|(j-k+1)^{-\kappa}f_{k,\kappa}g_{j-k,\kappa}\|_{L^2(\th_{2\nu})}\\
\leq&\sum^{[(j+1)/2]}_{k=0} \|(k+1)^{-\kappa} f_{k,\kappa}\|_{L^\i_h L^2_z(\th_\nu)}\|\th_{\f{\nu}{2}}g_{j-k,\kappa}\|_{L^2_hL^\i_z}\\
   &+\sum^{j}_{k=[(j+1)/2]+1} \|f_{k,\kappa}\th_{\f{\nu}{2}}\|_{L^2_hL^\i_z}\|(j-k+1)^{-\kappa}g_{j-k,\kappa}\|_{L^\i_hL^2_z(\th_{\nu})}.
\eali
\ees
}
Then using the following discrete Young's convolution inequality
\be\label{disyoung}
\sum^\i_{j=0}(\sum^j_{k=0}a_kb_{j-k})^2\leq \lt(\sum^{\i}_{k=0} a_k\rt)^2\lt(\sum^{\i}_{k=0} b^2_k\rt),
\ee
we can obtain that
{\small
\begin{align}
\|fg\|^2_{X_{\tau,\kappa,\nu}}=&\sum_{j\in\bN}\|(fg)_{j,\kappa}\|^2_{L^2(\th_{2\nu})}\nn\\
\leq&\lt(\sum^{\i}_{k=0}(k+1)^{-\kappa}\|f_{k,\kappa}\|_{L^\i_hL^2(\th_\nu)}\rt)^2
\lt(\sum^{\i}_{k=0}\|\th_{\nu/2}g_{k,\kappa}\|^2_{L^2_hL^\i_z}\rt)\label{prod1}\\
  &+\lt(\sum^{\i}_{k=0}\|\th_{\nu/2}f_{k,\kappa}\|^2_{L^2_hL^\i_z}\rt)
\lt(\sum^{\i}_{k=1}(k+1)^{-\kappa}\|g_{k,\kappa}\|_{L^\i_hL^2_v(\th_{\nu})}\rt)^2.\nn
\end{align}
}

Before continuing estimates, we give three weighted Sobolev embedding inequalities which will be frequently used later on. For any $f$, decaying fast enough at $z$ infinity, and $0\leq \nu<1$, by using the Sobolev embedding in $(x,y)$ variables, first we have
\be\label{sob1}
\bali
\lt\| f_{k,\kappa}\rt\|_{L^\i_hL^2_z(\th_\nu)}\leq \sum^{2}_{i=0}\lt\|\p^i_hf_{k,\kappa}\rt\|_{L^2_hL^2_z(\th_\nu)}\ls \sum^{k+2}_{i=k} (k+1)^4 \|f_{i,\kappa}\|_{L^2(\th_\nu)}.
\eali
\ee
Then using the fundamental formula of calculus, H\"{o}lder inequality and property of $\th$, we can obtain that
\be\label{sob2}
\bali
&\lt\|\th_{\nu/2}f_{k,\kappa}\rt\|_{L^2_hL^\i_z}\\
=&\lt\|\th_{\nu/2}\int^\i_z \p_zf_{k,\kappa}d \bar{z}\rt\|_{L^2_hL^\i_z}\leq\lt\|\int^\i_z \th_{\f{\nu-1}{4}}\th_{\f{\nu+1}{4}}\p_zf_{k,\kappa}d \bar{z}\rt\|_{L^2_hL^\i_z}\ls_\nu \angt^{1/4}\lt\|\th_{\f{\nu+1}{4}}\p_z f_{k,\kappa}\rt\|_{L^2}.
\eali
\ee
Combining the above two, we can achieve that
\be\label{sobolev3}
\bali
\lt\|\th_{\nu/2}f_{k,\kappa}\rt\|_{L^\i_hL^\i_z}\ls_\nu& \angt^{1/4}\lt\|\th_{\f{\nu+1}{4}}\p_z f_{k,\kappa}\rt\|_{L^\i_hL^2_z}\ls\angt^{1/4}\sum^{k+2}_{i=k} (k+1)^4 \|\th_{\f{\nu+1}{4}}\p_z f_{i,\kappa}\|_{L^2}.\nn
\eali
\ee
Inserting \eqref{sob1} and \eqref{sob2} into \eqref{prod1} and by using discrete Cauchy inequality, we can obtain that

{\small
\begin{align}
\|fg\|^2_{X_{\tau,\kappa,\nu}}\ls &\angt^{1/2}\lt(\sum^{\i}_{k=0}(k+1)^{-\kappa+4}\| f_{k,\kappa}\|_{L^2(\th_\nu)}\rt)^2
\sum^{\i}_{k=0}\|\p_zg_{k}\|^2_{L^2(\th_{\f{\nu+1}{2}})}\nn\\
  &+\angt^{1/2}\sum^{\i}_{k=0}\|\p_z f_{k,\kappa}\|^2_{L^2(\th_{\f{\nu+1}{2}})}
\lt(\sum^{\i}_{k=0}(k+1)^{-\kappa+4}\|g_{k,\kappa}\|_{L^2(\th_\nu)}\rt)^2dt\nn\\
\ls &\angt^{1/2}\sum^{\i}_{k=0}(k+1)^{-2\kappa+10}\| f_{k,\kappa}\|^2_{L^2(\th_{\f{\nu+1}{2}})}
\sum^{\i}_{k=0}\|\p_zg_{k}\|^2_{L^2(\th_{\f{\nu+1}{2}})}\nn\\
  &+\angt^{1/2}\sum^{\i}_{k=0}\|\p_z f_{k,\kappa}\|^2_{L^2(\th_{\f{\nu+1}{2}})}
\sum^{\i}_{k=0}(k+1)^{-2\kappa+10}\|g_{k,\kappa}\|^2_{L^2(\th_{\f{\nu+1}{2}})}dt,\nn
\end{align}
}
which is \eqref{product1}.

\end{proof}

\section{Estimates of auxiliary functions $\bl{\mH}$}\label{secmh}

In this section, we give the proof of Proposition \ref{pfirstauxi}. We only proceed with the estimate for $\mH$, while the estimate for $\wt{\mH}$ follows the same line. First we derive the equation for $\mH$.
\subsection{Derivation of the equation of $\mH$ and its linear estimate}

 Applying $-\p_z$ to \eqref{auxih} and using the incompressibility, we can have
\begin{align}
&\lt[\partial_{t} +\left({u} \p_x+v \partial_y+w\p_z \right)-\partial_{z}^{2} \rt] \mathcal{H}\nn\\
=&\s{\e} \angt^{\dl-1}  (\p^2_xu+\p_x\p_y v)+\p_z u\p_x\int^{+\i}_z \mathcal{H}d\bar{z}+\p_z v\p_y\int^{+\i}_z \mathcal{H}d\bar{z}-\p_zw \mathcal{H}\nn\\
=&\s{\e} \angt^{\dl-1}   \p_x(\p_x u+\f{\angt^{1-\dl}}{\s{\e} } \p_z u\int^{+\i}_z \mathcal{H}d\bar{z})-\p_z\p_xu\int^{+\i}_z \mathcal{H}d\bar{z}\nn\\
&+\s{\e} \angt^{\dl-1}  \p_y(\p_x v+\f{ \angt^{1-\dl}}{\s{\e}} \p_z v\int^{+\i}_z \mathcal{H}d\bar{z})-\p_z\p_yv\int^{+\i}_z \mathcal{H}d\bar{z} \label{auxih1}\\
&+(\p_xu+\p_yv)\mathcal{H}\nn\\
=&\s{\e} \angt^{\dl-1}   \lt(\p_x\mathcal{G}+\p_y\mathcal{K}\rt)-(\p_z\p_xu+\p_z\p_yv)\int^{+\i}_z \mathcal{H}d\bar{z}+(\p_xu+\p_yv)\mathcal{H}\nn\\
:=&\s{\e} \angt^{\dl-1}  \lt(\p_x\mathcal{G}+\p_y\mathcal{K}\rt)+H,\nn
\end{align}
where $H$ is defined as
\bes
H:=-(\p_z\p_xu+\p_z\p_yv)\int^{+\i}_z \mathcal{H}d\bar{z}+(\p_xu+\p_yv)\mathcal{H}.
\ees

From Section \ref{secmh} to Section \ref{secu}, we set $\kappa=10+\f{2}{\dl}$ and $M_{j,\kappa}$ is abbreviated to $M_j$. Also for a function $f$, $f_j$ denotes $f_{j,x,\kappa}$, $f_{j,y,\kappa}$ or $f_{\al,\kappa}$ for $|\al|=j$ if no confusion is caused.

From the equation of $\mH$ in \eqref{auxih1}, we first have the following linear estimate.

\begin{lemma}\label{llinearmH}
Under the assumption of Proposition \ref{pfirstauxi}, for sufficiently small $\e$, we have the following estimate
\begin{align}
&\angt^{\f{1-\dl}{2}}\|\mathcal{H}(t)\|^2_{X_{\tau,\kappa}}+\dl\int^{T_0}_0\angt^{\f{1-\dl}{2}}\|\p_z \mathcal{H}(t)\|^2_{X_{\tau,\kappa}}dt+\la\s{\e} \int^{T_0}_0\angt^{\f{1-\dl}{2}}{\eta}(t)\|\mathcal{H}(t)\|^2_{X_{\tau,\kappa+1/2}}dt\nn\\
 \leq &C_\dl\s{\e} \int^{T_0}_0\angt^{-\f{1-\dl}{2}}\lt\|\bl{\mG}\rt\|^2_{X_{\tau,\kappa+3/2}}dt+C_\dl C_\ast\e\int^{T_0}_0\angt^{\f{1-\dl}{2}}{\eta}(t)\|\mathcal{H}(t)\|^2_{X_{\tau,\kappa+1/2}}dt\nn\\
    &+\f{C_\dl}{\la\s{\e}}\int^{T_0}_0\angt^{\f{3-3\dl}{2}}\sum^\i_{j=0}(j+1)^{-1}\|M_j[u\p_x+v\p_y+w\p_z, \p^j_x]\mathcal{H}\|^2_{L^2(\th_2)}dt\label{auxilh6}\\
    &+\int^{T_0}_0\angt^{\f{1-\dl}{2}}\sum^\i_{j=0}|\lt\langle H_{j},{\mathcal{H}_{j}}\rt\rangle_{\th_2}|dt.\nn
\end{align}
\end{lemma}

\pf Applying $M_j\p^j_x$ to \eqref{auxih1} implies that
\be\label{auxih2}
\bali
&\partial_{t} \mathcal{H}_{j}+\la\s{\e}\eta(t)(j+1)\mathcal{H}_{j}+\left({u}\partial_{x}+v \partial_{y}+w\p_z\right) {\mathcal{H}_{j}}-\partial_{z}^{2} \mathcal{H}_{j}\\
=&M_j[u\p_x+v\p_y+w\p_z, \p^j_x]\mathcal{H}+\s{\e} \angt^{\dl-1}  \lt(\p_x\mathcal{G}_{j}+\p_y\mathcal{K}_{j}\rt)+H_{j}.
\eali
\ee
Then multiplying \eqref{auxih2} by $\mathcal{H}_{j}\th_2$ and integrating the resulted equation in $\bR^3_+$ indicate that

\be\label{auxilh3}
\bali
&\lt\langle \lt[\partial_{t} +\la\s{\e}{\eta}(t)(j+1)+\left({u}\partial_{x}+v \partial_{y}+w\p_z\right) -\partial_{z}^{2} \rt] {\mathcal{H}_{j}}, {\mathcal{H}_{j}}(t)\rt\rangle_{\th_2}\\
=&\lt\langle H_{j},{\mathcal{H}_{j}}\rt\rangle_{\th_2}+\s{\e} \angt^{\dl-1} \lt\langle  \lt(\p_x\mathcal{G}_{j}+\p_y\mathcal{K}_{j}\rt),{\mathcal{H}_{j}}\rt\rangle_{\th_2}\\
&+\lt\langle M_j[u\p_x+v\p_y+w\p_z, \p^j_x]\mathcal{H},\mathcal{H}_{j}\rt\rangle_{\th_2}.
\eali
\ee

For $\th_2:=e^{\f{z^2}{4\angt}}$, we have
\bes
\bali
-\f{\p_t\th_2}{\th_2}=\f{z^2}{4\angt},\q -\f{\p_z\th_2}{\th_2}=-\f{z}{2\angt},\q -\f{\p^2_z\th_2}{\th_2}=-\f{1}{2\angt}-\f{z^2}{4\angt}.
\eali
\ees
Integration by parts indicate that the left hand of \eqref{auxilh3} satisfies
\begin{align}
&\lt\langle \lt[\partial_{t} +\la\s{\e}{\eta}(t)(j+1)+\left({u}\partial_{x}+v \partial_{y}+w\p_z\right) -\partial_{z}^{2} \rt]\mathcal{H}_{j}, \mathcal{H}_{j}(t)\rt\rangle_{\th_2}\nn\\
=&\f{1}{2}\f{d}{dt}\|{\mathcal{H}_{j}}(t)\|^2_{L^2(\th_2)}+\|\p_z \mathcal{H}_{j}(t)\|^2_{L^2(\th_2)}-\f{1}{4\angt}\| \mathcal{H}_{j}(t)\|^2_{L^2(\th_2)}\label{hj5}\\
 &+(j+1)\la\s{\e}{\eta}(t)\|{\mathcal{H}_{j,x}}(t)\|^2_{L^2(\th_2)}-\langle \f{z}{4\angt}w,{\mathcal{H}^2_{j}}\rangle_{\th_2}.\nn
\end{align}
From \eqref{poincare} in Lemma \ref{lpoincare}, we have
\bes
\|\p_z \mathcal{H}_{j}(t)\|^2_{L^2(\th_2)}\geq \f{1}{2\angt}\| \mathcal{H}_{j}(t)\|^2_{L^2(\th_2)}.
\ees
Inserting the above inequality and \eqref{hj5} into \eqref{auxilh3} implies that
{\small
\be\label{auxilh4}
\bali
&\f{d}{dt}\|\mathcal{H}_{j}(t)\|^2_{L^2(\th_2)}+\dl\|\p_z \mathcal{H}_{j}(t)\|^2_{L^2(\th_2)}+\f{1-\dl}{2\angt}\| \mathcal{H}_{j}(t)\|^2_{L^2(\th_2)}+2(j+1)\la\s{\e}{\eta}(t)\|\mathcal{H}_{j}(t)\|^2_{L^2(\th^2)}\\
\leq &2\s{\e} \angt^{\dl-1} \lt\langle \lt(\p_x\mathcal{G}_{j}+\p_y\mathcal{K}_{j}\rt) ,{\mathcal{H}_{j}}\rt\rangle_{\th^2}+\langle \f{z}{2\angt}w,\mathcal{H}^2_{j}\rangle_{\th_2}\\
 &+2\lt\langle M_j[u\p_x+v\p_y+w\p_z, \p^j_x]\mathcal{H},\mathcal{H}_{j,x}\rt\rangle_{\th^2}+2\lt\langle H_{j},{\mathcal{H}_{j}}\rt\rangle_{\th_2}.
\eali
\ee
}
Multiplying \eqref{auxilh4} by $\angt^{\f{1-\dl}{2}}$ and then integrating from $0$ to $t$ for any $t\in (0,T_0]$, we can achieve that
{\small
\be\label{auxilh5}
\bali
&\angt^{\f{1-\dl}{2}}\|\mathcal{H}_{j}(t)\|^2_{L^2(\th_2)}+\dl\int^{T_0}_0\angt^{\f{1-\dl}{2}}\|\p_z \mathcal{H}_{j}(t)\|^2_{L^2(\th_2)}dt+2\la\s{\e}(j+1) \int^{T_0}_0\angt^{\f{1-\dl}{2}}{\eta}(t)\|\mathcal{H}_{j}(t)\|^2_{L^2(\th_2)}dt\\
\leq& 2\s{\e}\int^{T_0}_0\angt^{-\f{1-\dl}{2}}|\lt\langle \p_h\bl{\mathcal{G}}_{j},{\mathcal{H}_{j}}\rt\rangle_{\th_2}|dt+\int^{T_0}_0\angt^{\f{1-\dl}{2}}\langle \f{z}{2\angt}w,\mathcal{H}^2_{j}\rangle_{\th_2}dt\\
  &+2\int^{T_0}_0\angt^{\f{1-\dl}{2}}|\lt\langle M_j[u\p_x+v\p_y+w\p_z, \p^j_x]\mathcal{H},\mathcal{H}_{j}\rt\rangle_{\th_2}|dt+2\int^{T_0}_0\angt^{\f{1-\dl}{2}}|\lt\langle H_{j},{\mathcal{H}_{j}}\rt\rangle_{\th_2}|dt.
\eali
\ee
}
Using Cauchy inequality, we can get
{\small
\bes
\bali
&\text{The first and third term of the right hand of \eqref{auxilh5}}\\
\leq & \int^{T_0}_0\f{\angt^{\f{1-\dl}{2}}}{(j+1)\la\s{\e}{\eta}(t)}\lt(2\|M_j[u\p_x+v\p_y+w\p_z, \p^j_x]\mathcal{H}\|^2_{L^2(\th_2)}\rt)dt\\
     &+\s{\e}\int^{T_0}_0 (j+1)^{-1}\angt^{-\f{1-\dl}{2}}\lt\|\p_h\bl{\mathcal{G}}_{j}\rt\|^2_{L^2(\th_2)}dt
     +(j+1)\la\s{\e}\int^{T_0}_0\angt^{\f{1-\dl}{2}}{\eta}(t)\|\mathcal{H}_{j}(t)\|^2_{L^2(\th_2)}dt.
\eali
\ees
}
 Noting that $\p_h\bl{\mathcal{G}}_{j}\thickapprox_\dl (j+1)^2\bl{\mathcal{G}}_{j+1}$ and the a priori estimate in \eqref{lowpoint}, then inserting the above inequality into \eqref{auxilh5} and summing the resulted equation over $j\in\bN$, we can obtain
{\small
\begin{align}
&\angt^{\f{1-\dl}{2}}\|\mathcal{H}(t)\|^2_{X_{\tau,\kappa}}+\dl\int^{T_0}_0\angt^{\f{1-\dl}{2}}\|\p_z \mathcal{H}(t)\|^2_{X_{\tau,\kappa}}dt+\la\s{\e} \int^{T_0}_0\angt^{\f{1-\dl}{2}}{\eta}(t)\|\mathcal{H}(t)\|^2_{X_{\tau,\kappa+1/2}}dt\nn\\
\leq &C_\dl\s{\e}\int^{T_0}_0\angt^{-\f{1-\dl}{2}}\lt\|\bl{\mathcal{G}}\rt\|^2_{X_{\tau,\kappa+3/2}}dt+C_\dl C_\ast\e\int^{T_0}_0\angt^{-\f{1-\dl}{2}}\|\mathcal{H}(t)\|^2_{X_{\tau,\kappa}}dt\nn\\
    &+\f{C_\dl}{\la\s{\e}}\int^{T_0}_0\angt^{\f{3-3\dl}{2}}\sum^\i_{j=0}(j+1)^{-1}\|M_j[u\p_x+v\p_y+w\p_z, \p^j_x]\mathcal{H}\|^2_{L^2(\th_2)}dt\nn\\
    &+\int^{T_0}_0\angt^{\f{1-\dl}{2}}\sum^\i_{j=0}\lt|\lt\langle H_{j},{\mathcal{H}_{j}}\rt\rangle_{\th_2}\rt|dt,\nn
\end{align}
which is \eqref{auxilh6}. \qed

\subsection{Estimates of nonlinear terms for $\mH$}

Now we go to estimate the nonlinear terms on the righthand of \eqref{auxilh6}, we have the following Lemma.
\begin{lemma}\label{lnonlinearmH}
Under the assumption in \eqref{gassump}, for sufficiently small $\e$, we have the following estimate
\begin{align}
&\f{1}{\la\s{\e}} \int^{T_0}_0\angt^{\f{3-3\dl}{2}}\sum^{\i}_{j=0}(j+1)^{-1}\|M_j[u\p_x+v\p_y+w\p_z, \p^j_x]\mathcal{H}\|^2_{L^2(\th^2)}dt\nn\\
\ls_\dl&  \f{C^2_\ast}{\la}\s{\e} \int^{T_0}_0  \angt^{-\f{1-\dl}{2}}\lt(\lt\|\bl{\mH}\rt\|^2_{X_{\tau,\kappa+1/2}}+\lt\|\bl{u}\rt\|^2_{X_{\tau,\kappa+5/2}}\rt)dt+\f{C^2_\ast}{\la}\s{\e} \int^{T_0}_0  \angt^{\f{1-\dl}{2}}\lt(\lt\|\p_z\bl{\mH}\rt\|^2_{X_{\tau,\kappa}}+\lt\|\p_z\bl{u}\rt\|^2_{X_{\tau,\kappa+2}}\rt)dt,\label{nonmH}
\end{align}
and
\be\label{nonmH1}
\bali
&\int^{T_0}_0\angt^{\f{1-\dl}{2}}\sum^{\i}_{j=0}\lt|\lt\langle H_{j}, \mathcal{H}_j \rt\rangle_{\th_2}\rt|dt \leq \f{\dl}{2}\int^{T_0}_0\angt^{\f{1-\dl}{2}}\|\p_z\mH\|^2_{X_{\tau,\kappa}}dt+C_\dl\f{C^2_\ast}{\la}\s{\e}\int^{T_0}_0\angt^{-\f{1-\dl}{2}} \|\mH\|^2_{X_{\tau,\kappa}} dt.
\eali
\ee
\end{lemma}

Combining estimates in Lemma \ref{llinearmH} and Lemma \ref{lnonlinearmH}, we finish the proof of Proposition \ref{pfirstauxi}. Now we give the proof of Lemma \ref{lnonlinearmH}.

\pf Recall that
\bes
\bali
&M_j[u\p_x+v\p_y+w\p_z, \p^j_x]\mathcal{H}\\
=&-M_j\sum^j_{k=1} \lt(j\atop k\rt) \lt(\p^k_x u\p^{j-k+1}_x\mathcal{H}+\p^k_x v\p^{j-k}_x\p_y\mathcal{H}+\p^k_x w\p^{j-k}_x\p_z\mathcal{H}\rt)\\
:=& I^1_{j}+I^2_{j}+I^3_{j}.
\eali
\ees

We will estimate $I^i$ $(i=1,2,3)$ term by term.

{\noindent\bf Estimate of term $I^1$ and $I^2$.}

Since $I_1$ and $I_2$ share the same estimate, we only care about term $I_1$ and handling of term $I_2$ follows the same line. For term $I^1_{j}$, noting that when $1\leq k\leq \lt[\f{j+1}{2}\rt]\leq j$, we have
\bes
\lt(j\atop k\rt)^{-1}\leq (j+1)^{-1}.
\ees
Then similar as derivation of \eqref{keydiff}, we have

\be\label{termi1}
\bali
|I^1_{j}|\leq &\sum^{[(j+1)/2]}_{k=1} (k+1)^{-\kappa} |u_{k}| (j-k+1)^{-1}|\p_x\mathcal{H}_{j-k}|+\sum^{j}_{k=[(j+1)/2]+1}  (j-k+1)^{-\kappa} |u_{k}| |\p_x\mathcal{H}_{j-k}|.
\eali
\ee
By using \eqref{termi1}, similar derivation as \eqref{product2} in Lemma \ref{lemproduct}, we can obtain that
\begin{align}
&\sum_{j\in\bN}(j+1)^{-1}\|I^1_{j}(t)\|^2_{L^2(\th_2)}dt\nn\\
\ls& \angt^{1/2}\lt(\|\p_z u\|^2_{X_{\tau,5,1/2}}\|\p_x\mathcal{H}\|^2_{X_{\tau,\kappa-3/2,1/2}}+\|\p_z u\|^2_{X_{\tau,\kappa-1/2,1/2}}\|\p_x\mathcal{H}\|^2_{X_{\tau,5,1/2}}\rt) \label{termi1first}\\
\ls& \angt^{1/2}\lt(\|\p_z u\|^2_{X_{\tau,5,1/2}}\|\mathcal{H}\|^2_{X_{\tau,\kappa+1/2,1/2}}+\|\p_z u\|^2_{X_{\tau,\kappa+2,1/2}}\|\mathcal{H}\|^2_{X_{\tau,7,1/2}}\rt).\nn
\end{align}
Using a priori estimates in \eqref{solutiongevrey}}, we can obtain that
\be\label{termi1second}
\bali
&\f{1}{\la\s{\e}}\int^{T_0}_0\angt^{\f{3-3\dl}{2}}\sum_{j\in\bN}(j+1)^{-1}\|(I^1_{j},I^2_{j})(t)\|^2_{L^2(\th_2)}dt\\
\leq&\f{C^2_\ast\e^2}{\la\s{\e}} \int^{T_0}_0\lt(\angt^{-\f{1-\dl}{2}}\|\mH\|^2_{X_{\tau,\kappa+1/2}}+\angt^{\f{1-\dl}{2}}\|\p_zu\|^2_{X_{\tau,\kappa+2}}\rt)dt\\
\leq&\f{C^2_\ast}{\la}\s{\e} \int^{T_0}_0\lt(\angt^{-\f{1-\dl}{2}}\|\mH\|^2_{X_{\tau,\kappa+1/2}}+\angt^{\f{1-\dl}{2}}\|\p_zu\|^2_{X_{\tau,\kappa+2}}\rt)dt.
\eali
\ee

{\noindent\bf Estimate of term $I^3$.}

For term $I^3$, we have
\bes
\bali
|I^3_j|\ls \sum^{[(j+1)/2]}_{k=1} (k+1)^{-\kappa} (j-k+1)^{-1}|w_{k}| |\p_z\mathcal{H}_{j-k}|+\sum^{j}_{k=[(j+1)/2]+1}  (j-k+1)^{-\kappa} |w_{k}| |\p_z\mathcal{H}_{j-k}|.
\eali
\ees
Using the above inequality, similar as \eqref{termi1first} and using the incompressibility, we can obtain
\begin{align}
&\sum_{j\in\bN}(j+1)^{-1}\|I^3_{j}(t)\|^2_{L^2(\th_2)}dt\nn\\
\ls& \angt^{1/2}\lt(\|\p_z w\|^2_{X_{\tau,5,1/2}}\|\p_z\mathcal{H}\|^2_{X_{\tau,\kappa-3/2,1/2}}+\|\p_z w\|^2_{X_{\tau,\kappa-1/2,1/2}}\|\p_z\mathcal{H}\|^2_{X_{\tau,5,1/2}}\rt) \label{termi3first}\\
\ls& \angt^{1/2}\lt(\|(u,v)\|^2_{X_{\tau,7,1/2}}\|\p_z\mathcal{H}\|^2_{X_{\tau,\kappa}}+\|(u,v)\|^2_{X_{\tau,\kappa+5/2}}\|\p_z\mathcal{H}\|^2_{X_{\tau,5,1/2}}\rt).\nn
\end{align}
Then using the a priori estimates in \eqref{solutiongevrey} and smallness of $\e$, we can obtain that
\begin{align}
&\f{1}{\la\s{\e}}\int^{T_0}_0\angt^{\f{3-3\dl}{2}}\sum^{\i}_{j=0}(j+1)^{-1}\|I^3_j(t)\|^2_{L^2(\th^2)}dt\nn\\
\ls& \f{C^2_\ast}{\la} \s{\e}\int^{T_0}_0 \lt(\angt^{\f{1-\dl}{2}}\|\p_z\mH\|^2_{X_{\tau,\kappa}}+\angt^{-\f{1-\dl}{2}}\|(u,v)\|^2_{X_{\tau,\kappa+5/2}}\rt) dt.\label{termi3second}
\end{align}
Combining estimates in \eqref{termi1second} and \eqref{termi3second}, we can obtain \eqref{nonmH}.

{\noindent\bf Estimate of term involving  $H_{j}$.}

Next we estimate term involving  $H_{j}$. Recall
\bes
\bali
H_{j}=&M_j\p^j_x\lt(-(\p_z\p_xu+\p_z\p_yv)\int^{+\i}_z \mathcal{H}d\bar{z}+(\p_xu+\p_yv)\mathcal{H}\rt).
\eali
\ees
First, by integrating by parts on $z$ and using H\"{o}lder inequality, we can have
\begin{align}
&\lt|\lt\langle H_{j}, \mathcal{H}_j \rt\rangle_{\th_2}\rt|\nn\\
= &\lt|\lt\langle M_j \p^j_x\lt[ (\p_x u+\p_yv)\int^\i_z\mH d\bar{z}\rt] ,  \p_z\mathcal{H}_j+\mH_j \f{z}{2\angt}  \rt\rangle_{\th_2}\rt|\nn\\
\ls & \lt\|M_j \p^j_x\lt[(\p_x u+\p_yv)\int^\i_z\mH d\bar{z}\rt] \rt\|_{L^2(\th^2)}\|\p_z H_j\|_{L^2(\th_2)}.\nn
\end{align}
At the last line, we have used the fact in \eqref{poincare1}. Then using Cauchy inequality, we have
\bes
\bali
&\int^{T_0}_0\angt^{\f{1-\dl}{2}}\sum^{\i}_{j=0}\lt|\lt\langle H_{j}, \mathcal{H}_j \rt\rangle_{\th_2}\rt|dt\\
  \leq &\f{\dl}{2}\int^{T_0}_0\angt^{\f{1-\dl}{2}}\|\p_z\mH\|^2_{X_{\tau,\kappa}}dt+C_\dl\int^{T_0}_0\angt^{\f{1-\dl}{2}}\lt\|(\p_x u+\p_yv)\int^\i_z\mH d\bar{z} \rt\|^2_{X_{\tau,\kappa}} dt.
\eali
\ees
By applying \eqref{product2} in Lemma \ref{lemproduct} and using the a priori estimate in \eqref{solutiongevrey}, we obtain that
\begin{align}
&\int^{T_0}_0\angt^{\f{1-\dl}{2}}\sum^{\i}_{j=0}\lt|\lt\langle H_{j}, \mathcal{H}_j \rt\rangle_{\th_2}\rt|dt\nn\\
  \leq &\f{\dl}{2}\int^{T_0}_0\angt^{\f{1-\dl}{2}}\|\p_z\mH\|^2_{X_{\tau,\kappa}}dt+C_\dl \int^{T_0}_0\angt^{\f{2-\dl}{2}}\lt(\|\bl{u}\|^2_{X_{\tau,7,1/2}} \|\mH\|^2_{X_{\tau,\kappa}} +\|\bl{u}\|^2_{X_{\tau,\kappa+2}} \|\mH\|^2_{X_{\tau,5,1/2}}\rt) dt\nn\\
  \leq &\f{\dl}{2}\int^{T_0}_0\angt^{\f{1-\dl}{2}}\|\p_z\mH\|^2_{X_{\tau,\kappa}}dt+C_\dl C^2_\ast\e^2\int^{T_0}_0 \lt(\angt^{-\f{1}{2}}\|\mH\|^2_{X_{\tau,\kappa}}+\angt^{-\f{1-2\dl}{2}}\|\bl{u}\|^2_{X_{\tau,\kappa+2}}\rt)dt\nn\\
  \leq &\f{\dl}{2}\int^{T_0}_0\angt^{\f{1-\dl}{2}}\|\p_z\mH\|^2_{X_{\tau,\kappa}}dt+C_\dl \f{C^2_\ast}{\la}\s{\e}\int^{T_0}_0 \angt^{-\f{1-\dl}{2}}\lt(\|\mH\|^2_{X_{\tau,\kappa+1/2}}+\|\bl{u}\|^2_{X_{\tau,\kappa+5/2}}\rt)dt,\nn
\end{align}
which is \eqref{nonmH1}. \qed

\section{Estimates of auxiliary functions $\bl{\mG}$ }\label{secmg}

We only estimate $\mG$ since the other three follow the same line.

\subsection{Derivation of the equation of $\mG$ and its linear estimate}

 From the first equation of \eqref{3dprandtl}, we can obtain that
\be\label{mg1}
\bali
&\partial_{t} \p_x u+\left({u}\partial_{x}+v \partial_{y}+w\p_z\right) \p_x u-\partial_{z}^{2} \p_x u=-\p_x w\p_z u- (\p_x u)^2-\p_x v\p_y u.
\eali
\ee
Multiplying $\p_z u$ to the first equation of \eqref{auxih} indicates that
\be\label{mgfirst}
\bali
&\lt[\partial_{t} +\left({u}\partial_{x}+v \partial_{y}+w\p_z\right) -\partial_{z}^{2} \rt]\lt(\p_zu \int^\i_z \mathcal{H}d\bar{z}\rt)\\
=&\s{\e}\angt^{\dl-1}\p_x w\p_z u+\lt[\p_yv\p_zu-\p_yu\p_zv\rt]\int^\i_z \mathcal{H}d\bar{z}+2(\p^2_zu) \mathcal{H}.
\eali
\ee
Then by multiplying $\angt^{1-\dl}\e^{-1/2}$ to \eqref{mgfirst}, we can obtain that
\be\label{mg2}
\bali
&\lt[\partial_{t} +\left({u}\partial_{x}+v \partial_{y}+w\p_z\right) -\partial_{z}^{2} \rt]\lt(\f{\angt^{1-\dl}\p_zu}{\s{\e}} \int^\i_z \mathcal{H}d\bar{z}\rt)-\p_t\lt( \f{\angt^{1-\dl}}{\s{\e}}\rt) \lt(\p_z u \int^\i_z \mathcal{H}d\bar{z}\rt)\\
=&\p_x w\p_z u+\f{\angt^{1-\dl}\lt[\p_yv\p_zu-\p_yu\p_zv\rt]}{\s{\e}}\int^\i_z \mathcal{H}d\bar{z}+\f{2\angt^{1-\dl}(\p^2_zu)}{\s{\e}}\mathcal{H}.
\eali
\ee
Then adding \eqref{mg1} and \eqref{mg2} together implies that
\be\label{mg3}
\bali
&\lt[\partial_{t} +\left({u}\partial_{x}+v \partial_{y}+w\p_z\right) -\partial_{z}^{2} \rt]\mathcal{G}\\
=&-\lt[(\p_xu)^2+\p_xv\p_yu\rt]+\f{\angt^{1-\dl}\lt[\p_yv\p_zu-\p_yu\p_zv\rt]}{\s{\e}}\int^\i_z \mathcal{H}d\bar{z}+\f{2\angt^{1-\dl}(\p^2_zu)}{\s{\e}}\mathcal{H}\\
:=& K^1+K^2+K^3.
\eali
\ee
\begin{lemma}\label{lmg}
Under the assumption in \eqref{gassump}, for sufficiently small $\e$, we have the following estimate
{\small
\begin{align}
&\angt^{\f{1-\dl}{2}}\|\mathcal{G}(t)\|^2_{X_{\tau,\kappa+1}}+\dl\int^{T_0}_0\angt^{\f{1-\dl}{2}}\|\p_z \mathcal{G}(t)\|^2_{X_{\tau,\kappa+1}}dt+\la \s{\e} \int^{T_0}_0\angt^{\f{1-\dl}{2}}{\eta}(t)\|\mathcal{G}(t)\|^2_{_{X_{\tau,\kappa+3/2}}}dt\nn\\
\ls_\dl& \|\mathcal{G}(0)\|^2_{X_{\tau_0,\kappa+1}}+C_\ast\e \int^{T_0}_0\angt^{\f{1-\dl}{2}}{\eta}(t)\|\mathcal{G}(t)\|^2_{_{X_{\tau,\kappa+3/2}}}dt\nn\\
       &+\f{1}{\la\s{\e}}\int^{T_0}_0\angt^{\f{3-3\dl}{2}}\sum^\i_{j=0}(j+1)\|M_j[u\p_x+v\p_y+w\p_z, \p^j_x]\mathcal{G}\|^2_{L^2(\th_2)}dt\label{mglinear}\\
   &+\f{1}{\la\s{\e}}\int^{T_0}_0\angt^{\f{3-3\dl}{2}}\|(K^1,K^2)\|^2_{X_{\tau,\kappa+1/2}}dt+\int^{T_0}_0\angt^{\f{1-\dl}{2}}\sum^\i_{j=0}(j+1)^2\lt|\lt\langle K^3_{j}, \mG_{j} \rt\rangle_{\th_2}\rt|dt.\nn
\end{align}
}
\end{lemma}
\pf Applying $M_j\p^j_x$ to \eqref{mg3} and still denoting $\mG_{j,x}$ by $\mG_{j}$, we can obtain that
\bes
\bali
&\partial_{t} \mathcal{G}_{j}+\la\s{\e}\eta(t)(j+1)\mathcal{G}_{j}+\left({u}\partial_{x}+v \partial_{y}+w\p_z\right) {\mathcal{G}_{j}}-\partial_{z}^{2} \mathcal{G}_{j}\\
=&M_j[u\p_x+v\p_y+w\p_z, \p^j_x]\mathcal{G}+K^1_j+K^2_j+K^3_j.
\eali
\ees

Performing energy estimates similar as \eqref{auxilh5} and using the a priori estimates in \eqref{lowpoint}, we can obtain that
\bes
\bali
&\angt^{\f{1-\dl}{2}}(j+1)^2\|\mathcal{G}_{j}(t)\|^2_{L^2(\th_2)}+\dl(j+1)^2\int^{T_0}_0\angt^{\f{1-\dl}{2}}\|\p_z \mathcal{G}_{j}(t)\|^2_{L^2(\th_2)}dt\\
&+\la\s{\e}(j+1)^3 \int^{T_0}_0\angt^{\f{1-\dl}{2}}{\eta}(t)\|\mathcal{G}_{j}(t)\|^2_{L^2(\th_2)}dt\\
\ls_\dl & (j+1)^2\|\mathcal{G}_{j}(0)\|^2_{L^2(\th_2)}+  C_\ast\e(j+1)^2 \int^{T_0}_0\angt^{-\f{1-\dl}{2}}{\eta}(t)\|\mathcal{G}_{j}(t)\|^2_{L^2(\th_2)}dt\\
   &+\f{1}{\la\s{\e}}\int^{T_0}_0\angt^{\f{3-3\dl}{2}}(j+1)\lt(\|(K^1_{j},K^2_j)(t)\|^2_{L^2(\th_2)}+\|M_j[u\p_x+v\p_y+w\p_z, \p^j_x]\mathcal{G}\|^2_{L^2(\th_2)}\rt)dt\\
   &+\int^{T_0}_0\angt^{\f{1-\dl}{2}}(j+1)^2\lt|\lt\langle K^3_{j}, \mG_{j} \rt\rangle_{L^2(\th_2)}\rt|dt.
\eali
\ees
Summing the above inequality over $j\in \bN$ indicates \eqref{mglinear}. \qed

\subsection{Estimates of nonlinear terms for $\mG$}

\begin{lemma}\label{lmgnonlinear}
Under the assumption in \eqref{gassump}, for sufficiently small $\e$, we have the following estimate
{\small
\begin{align}
&\f{1}{\la\s{\e}}\int^{T_0}_0\angt^{\f{3-3\dl}{2}}\sum^\i_{j=1}(j+1)\|M_j[u\p_x+v\p_y+w\p_z, \p^j_x]\mathcal{G}\|^2_{L^2(\th_2)}dt\nn\\
 &+\f{1}{\la\s{\e}}\int^{T_0}_0\angt^{\f{3-3\dl}{2}}\|(K^1,K^2)\|^2_{X_{\tau,\kappa+1/2}}dt+\int^{T_0}_0\angt^{\f{1-\dl}{2}}\sum^\i_{j=0}(j+1)^2\lt|\lt\langle K^3_{j}, \mG_{j} \rt\rangle_{\th_2}\rt|dt\nn\\
\ls_\dl&\f{C^2_\ast}{\la}\s{\e} \int^{T_0}_0 \angt^{-\f{1-\dl}{2}} \lt(\lt\|\bl{\mH}(t)\rt\|^2_{X_{\tau,\kappa+1/2}}+\lt\|\bl{\mG}(t)\rt\|^2_{X_{\tau,\kappa+3/2}}+\lt\|\bl{u}(t)\rt\|^2_{X_{\tau,\kappa+5/2}} \rt)dt\label{mgnonlinear}\\
 &+\f{C^2_\ast}{\la}\s{\e} \int^{T_0}_0 \angt^{\f{1-\dl}{2}} \lt(\lt\|\p_z\bl{\mH}(t)\rt\|^2_{X_{\tau,\kappa}}+\lt\|\p_z\bl{\mG}(t)\rt\|^2_{X_{\tau,\kappa+1}}+\lt\|\p_z\bl{u}(t)\rt\|^2_{X_{\tau,\kappa+2}} \rt)dt.\nn
\end{align}
}
\end{lemma}
Combining estimates in Lemma \ref{lmg} and Lemma \ref{lmgnonlinear}, we finish the proof of Proposition \ref{psecondauxi}. Now we give the proof of Lemma \ref{lmgnonlinear}.

\pf First, using Leibniz formula, we see that
\bes
\bali
&M_j[u\p_x+v\p_y+w\p_z, \p^j_x]\mathcal{H}\\
=&-M_j\sum^j_{k=1} \lt(j\atop k\rt) \lt(\p^k_x u\p^{j-k+1}_x\mathcal{G}+\p^k_x v\p^{j-k}_x\p_y\mathcal{G}+\p^k_x w\p^{j-k}_x\p_z\mathcal{G}\rt)\\
:=& L^1_{j}+L^2_{j}+L^3_{j}.
\eali
\ees
Similar as \eqref{termi1second}, using \eqref{gassump} and the a priori estimates in Lemma \ref{llowpoint}, for sufficiently small $\e$, we have
\begin{align}
&\f{1}{\la\s{\e}}\int^{T_0}_0\angt^{\f{3-3\dl}{2}}\sum_{j\in\bN}(j+1)\|(L^1_{j},L^2_{j})(t)\|^2_{L^2(\th_2)}dt\nn\\
\ls_\dl &\f{1}{\la\s{\e}} \int^{T_0}_0\angt^{\f{4-3\dl}{2}}\lt(\|\p_z u\|^2_{X_{\tau,5,1/2}}\|\p_x\mathcal{G}\|^2_{X_{\tau,\kappa-1/2,1/2}} +\| \p_zu\|^2_{X_{\tau,\kappa+1/2,1/2}}\|\p_x\mathcal{G}\|^2_{X_{\tau,5,1/2}}\rt)dt\label{terml1f}\\
 \ls_\dl& \f{C^2_\ast}{\la}\e^{3/2}\int^{T_0}_0\lt(\angt^{-\f{1-\dl}{2}}\|\mathcal{G}\|^2_{X_{\tau,\kappa+3/2}}+\angt^{\f{1-\dl}{2}}\|\p_zu\|^2_{X_{\tau,\kappa+2}}\rt)dt\nn\\
\ls_\dl & \f{C^2_\ast}{\la}\s{\e}\int^{T_0}_0\lt(\angt^{-\f{1-\dl}{2}}\|\mathcal{G}\|^2_{X_{\tau,\kappa+3/2}}+\angt^{\f{1-\dl}{2}}\|\p_zu\|^2_{X_{\tau,\kappa+2}}\rt)dt.\nn
\end{align}
Similar same as \eqref{termi3first}, by using the incompressibility, we can obtain that
\begin{align}
&\sum_{j\in\bN}(j+1)\|L^3_{j}(t)\|^2_{L^2(\th_2)}dt\nn\\
\ls& \angt^{1/2}\lt(\|\p_z w\|^2_{X_{\tau,5,1/2}}\|\p_z\mathcal{G}\|^2_{X_{\tau,\kappa-1/2,1/2}}+\|\p_z w\|^2_{X_{\tau,\kappa+1/2,1/2}}\|\p_z\mathcal{G}\|^2_{X_{\tau,5,1/2}}\rt) \label{terml3first}\\
\ls& \angt^{1/2}\lt(\|\bl{u}\|^2_{X_{\tau,7,1/2}}\|\p_z\mathcal{G}\|^2_{X_{\tau,\kappa+1}}+\|\bl{u}\|^2_{X_{\tau,\kappa+5/2}}\|\p_z\mathcal{G}\|^2_{X_{\tau,5,1/2}}\rt).\nn
\end{align}
Then using \eqref{terml3first} and the a priori estimates in Lemma \ref{llowpoint}, for sufficiently small $\e$, we see that
\begin{align}
&\f{1}{\la \s{\e}}\int^{T_0}_0\angt^{\f{3-3\dl}{2}}\sum_{j\in\bN}(j+1)\|L^3_j(t)\|^2_{L^2(\th_2)}dt\nn\\
\leq&\f{1}{\la\s{\e}} \int^{T_0}_0\angt^{\f{4-3\dl}{2}}\lt(\|\bl{u}\|^2_{X_{\tau,7,1/2}}\|\p_z\mathcal{G}\|^2_{X_{\tau,\kappa+1}}+\|\bl{u}\|^2_{X_{\tau,\kappa+5/2}}\|\p_z\mathcal{G}\|^2_{X_{\tau,5,1/2}}\rt)dt\label{terml3second}\\
\ls& \f{C^2_\ast}{\la}\s{\e}\int^{T_0}_0\lt(\angt^{\f{1-\dl}{2}}\|\p_z\mathcal{G}\|^2_{X_{\tau,\kappa+1}}+\angt^{-\f{1-\dl}{2}}\|\bl{u}\|^2_{X_{\tau,\kappa+5/2}}\rt)dt.\nn
\end{align}

\noindent{\bf Estimates of $K^1_{j}$.}

Remembering the representation of $K^1$ in \eqref{mg3} and using the product estimate in \eqref{product3} and a priori estimates \eqref{solutiongevrey} in Lemma \ref{llowpoint}, we have
\begin{align}
&\f{1}{\la\s{\e}}\int^{T_0}_0\angt^{\f{3-3\dl}{2}}\|K^1_{j}(t)\|^2_{X_{\tau,\kappa+1/2}}dt= \f{1}{\la\s{\e}}\int^{T_0}_0\angt^{\f{3-3\dl}{2}}\|\p_h\bl{u}\p_h\bl{u}\|^2_{X_{\tau,\kappa+1/2}}dt\nn\\
\ls_\dl& \f{1}{\la\s{\e}}\int^{T_0}_0\angt^{\f{4-3\dl}{2}}\|\p_z\p_h\bl{u}\|^2_{X_{\tau,5,1/2}}\|\p_h\bl{u}\|^2_{X_{\tau,\kappa+1/2,1/2}}dt
\ls_\dl\f{1}{\la\s{\e}}\int^{T_0}_0\angt^{\f{4-3\dl}{2}}\|\p_z\bl{u}\|^2_{X_{\tau,7,1/2}}\|\bl{u}\|^2_{X_{\tau,\kappa+5/2}}dt\label{termk1}\\
\ls_\dl &\f{C^2_\ast}{\la}\s{\e}\int^{T_0}_0\angt^{-\f{1-\dl}{2}}\|\bl{u}\|^2_{X_{\tau,\kappa+5/2}}dt.\nn
\end{align}

\noindent{\bf Estimates of $K^2_{j}$.}

Remembering the representation of $K^2$ in \eqref{mg3} and using the product estimate in \eqref{product2} and \eqref{product3}, we have
{\small
\begin{align}
&\|K^2_{j}(t)\|^2_{X_{\tau,\kappa+1/2}}= \angt^{2-2\dl}\e^{-1}\lt\|\p_h\bl{u}\p_z\bl{u}\int^\i_z\mH d\bar{z}\rt\|^2_{X_{\tau,\kappa+1/2}}dt\nn\\
\ls_\dl & \angt^{5/2-2\dl}\e^{-1}\lt(\lt\|\p_z\lt(\p_z\bl{u}\int^\i_z\mH d\bar{z}\rt)\rt\|^2_{X_{\tau,5,1/2}}\|\p_h\bl{u}\|^2_{X_{\tau,\kappa+1/2,1/2}}+\|\p_z \p_h\bl{u} \|^2_{X_{\tau,5,1/2}}\|\p_z\bl{u}\int^\i_z\mH\|^2_{X_{\tau,\kappa+1/2,1/2}}\rt)\nn\\
\ls & \angt^{3-2\dl}\e^{-1}\lt(\|\p^2_z\bl{u}\|^2_{X_{\tau,5,3/8}}\|\mH\|^2_{X_{\tau,5,3/8}}\|\bl{u}\|^2_{X_{\tau,\kappa+5/2,1/2}}\rt.\nn\\
    & +\|\p_z\bl{u}\|^2_{X_{\tau,7,3/4}}\lt(\|\p_z\bl{u}\|^2_{X_{\tau,5,3/8}}\|\mH\|^2_{X_{\tau,\kappa+1/2,3/8}}
    +\|\mH\|^2_{X_{\tau,5,3/8}}\|\p_z\bl{u}\|^2_{X_{\tau,\kappa+1/2,3/8}}\rt).\nn
\end{align}
}
Then using the a priori estimates in \eqref{solutiongevrey}, for sufficiently small $\e$, we can obtain that
\begin{align}
&\f{1}{\la\s{\e}}\int^{T_0}_0\angt^{\f{3-3\dl}{2}}\sum^{\i}_{j=0}(j+1)\|K^2_{j}(t)\|^2_{L^2(\th^2)}dt\nn\\
\ls_\dl & \f{C^2_\ast}{\la}\s{\e}\int^{T_0}_0\lt(\angt^{-\f{1-\dl}{2}}\lt(\|\bl{\mH}\|^2_{X_{\tau,\kappa+1/2}}+\|\bl{u}\|^2_{X_{\tau,\kappa+5/2}}\rt)
+\angt^{\f{1-\dl}{2}}\|\p_z\bl{u}\|^2_{X_{\tau,\kappa+1/2}}\rt)dt.\label{termk2f}
\end{align}

\noindent{\bf Estimates of $K^3_{j}$}

First we claim that for $0<\nu\leq 1$ and $\kappa>0$,
\be\label{product4}
\sum^\i_{j=0}\|(fg)_{j,\kappa}\|^2_{L^2_hL^1_z(\th_{2\nu})}\ls \|f\|^2_{X_{\tau,5,\nu}}\|g\|^2_{X_{\tau,\kappa,\nu}}+\|f\|^2_{X_{\tau,\kappa,\nu}}\|g\|^2_{X_{\tau,5,\nu}}.
\ee
The proof of \eqref{product4} is essentially the same as that in Lemma \ref{lemproduct}, here we omit the details.

For term $K^3$, we decompose it as the following.
\be
\bali
K^3_{j}=& 2\f{\angt^{1-\dl}}{\s{\e}}\p^2_z u \mathcal{H}_{j}+2\f{\angt^{1-\dl}}{\s{\e}}\sum_{1\leq k\leq j }\f{M_j}{M_k M_{j-k}}\lt(j\atop k\rt)\p^2_z u_{k}\mathcal{H}_{j-k}\nn\\
      :=&K^3_{j,\operatorname{low}}+K^3_{j,\operatorname{other}}.\nn
\eali
\ee
Then by using Cauchy inequality and the a priori estimates in Lemma \ref{llowpoint}, we have
\begin{align}
&\int^{T_0}_0 \angt^{\f{1-\dl}{2}}\sum^\i_{j=0}(j+1)^2\lt|\langle K^3_{j,\operatorname{low}}, \mathcal{G}_j\rangle_{\th_2}\rt|dt\nn\\
\leq &\f{C_\dl}{\la\e^{3/2}}\int^{T_0}_0 \angt^{\f{7-7\dl}{2}}\| \p^2_z u\|^2_{L^\i}\lt\|\mathcal{H}\rt\|^2_{X_{\tau,\kappa+1/2}}dt+\f{\la\s{\e}}{2}\int^{T_0}_0 \angt^{-\f{1-\dl}{2}}\lt\| \mathcal{G}\rt\|^2_{X_{\tau,\kappa+3/2}}dt\label{termk3f}\\
\leq &C_\dl\f{ C^2_\ast}{\la}\s{\e}\int^{T_0}_0 \angt^{-\f{1-\dl}{2}}\lt\|\mathcal{H}\rt\|^2_{X_{\tau,\kappa+1/2}}dt+\f{\la\s{\e}}{2}\int^{T_0}_0 \angt^{-\f{1-\dl}{2}}\lt\| \mathcal{G}\rt\|^2_{X_{\tau,\kappa+3/2}}dt.\nn
\end{align}
By using integration by parts on $z$, we have that
\begin{align}
&|\langle K^3_{j,\operatorname{other}}, \mathcal{G}_j\rangle_{\th_2}|\nn\\
\leq &2\f{\angt^{1-\dl}}{\s{\e}}\lt|\lt\langle\sum_{1\leq k\leq j }\f{M_j}{M_k M_{j-k}}\lt(j\atop k\rt)\p_z u_{k}\p_z\mathcal{H}_{j-k},\mathcal{G}_j \rt\rangle_{\th_2}\rt| \nn\\
     &+2\f{\angt^{1-\dl}}{\s{\e}}\lt|\lt\langle\sum_{1\leq k\leq j }\f{M_j}{M_k M_{j-k}}\lt(j\atop k\rt)\p_z u_{k}\mathcal{H}_{j-k},\p_z\mathcal{G}_j+\f{z}{2\angt}\mathcal{G}_j \rt\rangle_{\th_2}\rt|.\nn
\end{align}
Using H\"{o}lder inequality and \eqref{sob2}, \eqref{poincare1}, we can obtain that
\begin{align}
&\sum^\i_{j=0}(j+1)^2|\langle K_{3,j,\operatorname{other}}, \mathcal{G}_j\rangle_{\th^2}|\nn\\
\ls &2\f{\angt^{1-\dl}}{\s{\e}}\sum^\i_{j=0}(j+1)^2\lt\|\sum_{1\leq k\leq j }\f{M_j}{M_k M_{j-k}}\lt(j\atop k\rt)\p_z u_{k}\p_z\mathcal{H}_{j-k}\rt\|_{L^2_hL^1_z(\th_{3/2})}\angt^{1/4}\| \p_z\mG_j\|_{L^2(\th)}\label{k3four}\\
 &+2\f{\angt^{1-\dl}}{\s{\e}}\sum^\i_{j=0}(j+1)^2\lt\|\sum_{1\leq k\leq j }\f{M_j}{M_k M_{j-k}}\lt(j\atop k\rt)\p_z u_{k}\mathcal{H}_{j-k}\rt\|_{L^2(\th_{2})}\| \p_z\mG_j\|_{L^2_z(\th_{2})}.\nn
\end{align}
The same as before, it is easy to see that
\begin{align}
\lt|\sum_{1\leq k\leq j }\f{M_j}{M_k M_{j-k}}\lt(j\atop k\rt)\p_z u_{k}\p_z\mathcal{H}_{j-k}\rt|\ls &\sum^{[(j+1)/2]}_{k=1} (k+1)^{-\kappa}|\p_z u_{k}|(j-k+1)^{-1}\lt|\p_z\mathcal{H}_{j-k}\rt|\nn\\
&+\sum^{j}_{k=[(j+1)/2]+1}  (j-k+1)^{-\kappa} |\p_z u_{k}|\lt|\p_z\mathcal{H}_{j-k}\rt|.\nn
\end{align}
Using Minkowski inequality, Sobolev embedding, discrete young inequality in \eqref{disyoung}, and a priori estimates in \eqref{solutiongevrey} and \eqref{lowpoint}, we have
\begin{align}
&\sum^\i_{j=0}(j+1)^2\lt\|\sum_{1\leq k\leq j }\f{M_j}{M_k M_{j-k}}\lt(j\atop k\rt)\p_z u_{k}\p_z\mathcal{H}_{j-k}\rt\|^2_{L^2_hL^1_z(\th_{3/2})}\nn\\
   \ls&\| \p_z u_{k}\|^2_{X_{\tau,5,3/4}}\lt\|\p_z\mathcal{H}\rt\|^2_{X_{\tau,\kappa,3/4}}+ \|\p_z u\|^2_{X_{\tau,\kappa+1,3/4}}\|\p_z\mathcal{H}\|^2_{X_{\tau,5,3/4}}\label{termk3f}\\
\ls& C^2_\ast\e^2\angt^{-\f{5-3\dl}{2}} \lt(\lt\|\p_z\mathcal{H}\rt\|^2_{X_{\tau,\kappa}} +\|\p_z u\|^2_{X_{\tau,\kappa+2}}\rt),\nn
\end{align}
and similar as the proof of Lemma \ref{lemproduct}, we have
{\small
\be\label{termk3second}
\bali
&\sum^\i_{j=0}(j+1)^2\lt\|\sum_{1\leq k\leq j }\f{M_j}{M_k M_{j-k}}\lt(j\atop k\rt)\p_z u_{k}\mathcal{H}_{j-k}\rt\|^2_{L^2(\th_{3/2})}\\
   \ls&\angt^{1/2}\| \p_z u_{k}\|^2_{X_{\tau,5,7/16}}\lt\|\p_z\mathcal{H}\rt\|^2_{X_{\tau,\kappa,7/16}}+\angt^{1/2} \|\p_z u\|^2_{X_{\tau,\kappa+1,7/16}}\|\p_z\mathcal{H}\|^2_{X_{\tau,5,7/16}}\\
   \ls& C^2_\ast\e^2\angt^{-\f{4-3\dl}{2}}\lt(\lt\|\p_z\mathcal{H}\rt\|^2_{X_{\tau,\kappa}}+  \|\p_z u\|^2_{X_{\tau,\kappa+2}}\rt).
\eali
\ee
}
By using Young inequality to \eqref{k3four} and  inserting \eqref{termk3f} and \eqref{termk3second} into  the resulted inequality, we can obtain that
{\small
\begin{align}
&\int^{T_0}_0 \angt^{\f{1-\dl}{2}}\sum^\i_{j=0}(j+1)^2|\langle K^3_{j,\operatorname{other}}, \mathcal{G}_j\rangle_{\th^2}|dt\nn\\
\leq &\f{\dl}{4}\int^{T_0}_0 \angt^{\f{1-\dl}{2}}\| \p_z\mG\|^2_{X_{\tau,\kappa+1}}dt\nn\\
&+\f{C_\dl}{\e}\int^{T_0}_0\angt^{\f{6-5\dl}{2}} \sum^\i_{j=0}(j+1)^2\lt\|\sum_{1\leq k\leq j }\f{M_j}{M_k M_{j-k}}\lt(j\atop k\rt)\p_z u_{k}\p_z\mathcal{H}_{j-k}\rt\|^2_{L^2_hL^1_z(\th^{3/2})}dt \label{termk3five}\\
 &+\f{C_\dl}{\e}\int^{T_0}_0 \angt^{\f{5-5\dl}{2}}\sum^\i_{j=0}(j+1)^2\lt\|\sum_{1\leq k\leq j }\f{M_j}{M_k M_{j-k}}\lt(j\atop k\rt)\p_z u_{k}\mathcal{H}_{j-k}\rt\|^2_{L^2(\th_{2})}dt\nn\\
\leq &\f{\dl}{4}\int^{T_0}_0 \angt^{\f{1-\dl}{2}}\| \p_z\mG\|_{X_{\tau,\kappa+1}}dt+C_\dl {C^2_\ast\e}\int^{T_0}_0 \angt^{\f{1-\dl}{2}}\lt(\lt\|\p_z\mathcal{H}\rt\|^2_{X_{\tau,\kappa}}+\|\p_z u\|^2_{X_{\tau,\kappa+2}}\rt)dt.\nn
\end{align}
}

Combining the estimates in \eqref{terml1f}, \eqref{terml3second}, \eqref{termk1}, \eqref{termk2f} and \eqref{termk3five}, we can obtain \eqref{mgnonlinear}. \qed


\section{Estimates of the unknowns $\bl{u}$}\label{secu}

We only derive estimates of $u$ since the estimate of $v$ follows the same line. Due to the fact that the equation of ${u}_j$ have one order derivative lose, direct Gevrey-2 energy estimates on the equations of ${u}_j$ do not work. Instead, we will use another alternative quantity $\psi_j$, defined as
\be\label{defpsi}
\psi_j:=u_j+\f{\angt^{1-\dl}\p_zu}{\s{\e}} M_j\int^\i_z \p^{j-1}_x\mathcal{H}d\bar{z},
\ee to perform energy estimates, which has no derivative lose.  Combining estimates of $\mH_{j-1}$ and $\psi_j$, we can achieve estimates of $u_j$. Then Proposition \ref{punknown} follows.

\subsection{The equation of an equivalent quantity and its linear estimate}

First, we derive the equation of $\psi_j$. From the first equation of \eqref{3dprandtl}, we can obtain that
\be\label{u1}
\bali
&\partial_{t} \p^j_x u+\left({u}\partial_{x}+v \partial_{y}+w\p_z\right) \p^j_x u-\partial_{z}^{2} \p^j_x u\\
=&[u\p_x+v\p_y+w\p_z, \p^j_x]u\\
=&-\p^j_x w\p_z u-\sum^j_{k=1} \lt(j\atop k\rt) \lt(\p^k_x u\p^{j-k+1}_x u+\p^k_x v\p^{j-k}_x\p_y u\rt)-\sum^{j-1}_{k=1}\lt(j\atop k\rt)\p^k_x w\p^{j-k}_x\p_zu\\
:=&-\p^j_x w\p_z u+O^1+O^2+O^3.
\eali
\ee
Also by applying $\p_z u\p^{j-1}_x$ to the first equation of \eqref{auxih}, we can obtain that
\bes
\bali
&\lt[\partial_{t} +\left({u}\partial_{x}+v \partial_{y}+w\p_z\right) -\partial_{z}^{2} \rt]\lt(\p_zu \int^\i_z \p^{j-1}_x\mathcal{H}d\bar{z}\rt)\\
=&\s{\e}\angt^{\dl-1}\p^j_x w\p_z u-\p_zu\sum^{j-1}_{k=1} \lt(j-1\atop k\rt) \lt(\p^k_x u\int^\i_z\p^{j-k}_x \mathcal{H}d\bar{z}+\p^k_x v\int^\i_z \p^{j-1-k}_x\p_y\mathcal{H}d\bar{z} \rt)\\
 &+\p_zu\sum^{j-1}_{k=1}\lt(j-1\atop k\rt)\p^k_x w\p^{j-1-k}_x\mathcal{H}\\
 &+\lt[\p_yv\p_zu-\p_yu\p_zv\rt]\int^\i_z\p^{j-1}_x \mathcal{H}d\bar{z}+2(\p^2_zu)\p^{j-1}_x \mathcal{H}.
\eali
\ees
Divided by $\s{\e}\angt^{\dl-1}$ from the above equation, we can obtain that
\begin{align}
&\lt[\partial_{t} +\left({u}\partial_{x}+v \partial_{y}+w\p_z\right) -\partial_{z}^{2} \rt]\lt(\f{\angt^{1-\dl}\p_zu}{\s{\e}} \int^\i_z \p^{j-1}_x\mathcal{H}d\bar{z}\rt)\nn\\
&-\p_t\lt(\f{\angt^{1-\dl}}{\s{\e}}\rt)\lt(\p_z u \int^\i_z \p^{j-1}_x\mathcal{H}d\bar{z}\rt)\nn\\
=&\p^j_x w\p_z u-\f{\angt^{1-\dl}\p_zu}{\s{\e}}\sum^{j-1}_{k=1} \lt(j-1\atop k\rt) \lt(\p^k_x u\int^\i_z\p^{j-k}_x \mathcal{H}d\bar{z}+\p^k_x v\int^\i_z \p^{j-1-k}_x\p_y\mathcal{H}d\bar{z} \rt) \label{u2}\\
 &+\f{\angt^{1-\dl}\p_zu}{\s{\e}}\sum^{j-1}_{k=1}\lt(j-1\atop k\rt)\p^k_x w\p^{j-1-k}_x\mathcal{H}\nn\\
 &+\f{\lt[\p_yv\p_zu-\p_yu\p_zv\rt]}{\s{\e}\angt^{\dl-1}}\int^\i_z\p^{j-1}_x \mathcal{H}d\bar{z}+\f{2(\p^2_zu)}{\s{\e}\angt^{\dl-1}}\p^{j-1}_x \mathcal{H}\nn\\
&:=\p^j_x w\p_z u+\sum^5_{i=1}P^i.\nn
\end{align}

Then add \eqref{u2} and \eqref{u1} together implies that
\be\label{uauxi}
\bali
&\lt[\partial_{t} +\left({u}\partial_{x}+v \partial_{y}+w\p_z\right) -\partial_{z}^{2} \rt]\lt(\p^j_xu+\f{\p_zu}{\s{\e}\angt^{\dl-1}} \int^\i_z \p^{j-1}_x\mathcal{H}d\bar{z}\rt)\\
=&\p_t\lt(\f{1}{\s{\e}\angt^{\dl-1}}\rt) \lt(\p_z u \int^\i_z \p^{j-1}_x\mathcal{H}d\bar{z}\rt)+\sum^3_{i=1}O^i+\sum^5_{i=1}P^i.
\eali
\ee
Remembering the definition of $\psi_j$ in \eqref{defpsi}, by multiplying $M_j$ to \eqref{uauxi}, we can obtain that
\be\label{psifirst}
\bali
&\lt[\partial_{t} +\la\s{\e}(j+1)+\left({u}\partial_{x}+v \partial_{y}+w\p_z\right) -\partial_{z}^{2} \rt]\psi_j\\
=&M_j\p_t\lt(\f{\angt^{1-\dl}}{\s{\e}}\rt) \lt(\p_z u \int^\i_z \p^{j-1}_x\mathcal{H}d\bar{z}\rt)+\sum^3_{i=1}O^i_j+\sum^5_{i=1}P^i_j.
\eali
\ee
where $O^i_j:=M_j O^i$ and $P^i_j:=M_j P^i$.

There is no derivative loss for the equation \eqref{psifirst}. For $\al\geq 0$, denote
\bes
 \|\psi\|^2_{X_{\tau,\kappa+\al}}:=\sum^\i_{j=0}(j+1)^{2\al}\|\psi_j\|^2_{L^2(\th_2)}.
\ees
We have the following linear estimate.

\begin{lemma}\label{lpsi}
Under the assumption in \eqref{gassump}, for sufficiently small $\e$, we have the following estimate
{\small
\begin{align}
&\angt^{\f{1-\dl}{2}}\|\psi(t)\|^2_{X_{\tau,\kappa+2}}+\dl\int^{T_0}_0\angt^{\f{1-\dl}{2}}\|\p_z \psi(t)\|^2_{X_{\tau,\kappa+2}}dt+\la \s{\e} \int^{T_0}_0\angt^{\f{1-\dl}{2}}{\eta}(t)\|\psi(t)\|^2_{_{X_{\tau,\kappa+5/2}}}dt\nn\\
\ls_\dl& \|{u}(0)\|^2_{X_{\tau_0,\kappa+2}}+C_\ast\e \int^{T_0}_0\angt^{\f{1-\dl}{2}}{\eta}(t)\|\psi(t)\|^2_{_{X_{\tau,\kappa+5/2}}}dt\nn\\
       &+\f{1}{\la\s{\e}}\int^{T_0}_0\angt^{\f{3-3\dl}{2}}\sum^\i_{j=0}(j+1)^3\lt\|(\sum^3_{i=1}O^i_j+\sum^5_{i=1}P^i_j)\rt\|^2_{L^2(\th_2)}dt\label{psilinear}\\
   &+\int^{T_0}_0\angt^{\f{1-3\dl}{2}} \sum^\i_{j=0}(j+1)^4\lt\langle M_j\f{\p_zu}{\s{\e}}\int^\i_z \p^{j-1}_x\mathcal{H}d\bar{z}, \psi_j\rt\rangle_{\th_2}dt.\nn
\end{align}
}
\end{lemma}

\pf Performing energy estimates for \eqref{psifirst} similar as \eqref{auxilh5} and using Cauchy inequality, we can have
\begin{align}
&\angt^{\f{1-\dl}{2}}(j+1)^4\|\psi_j(t)\|^2_{L^2(\th_2)}+\dl(j+1)^4\int^{T_0}_0\angt^{\f{1-\dl}{2}}\|\p_z \psi_j(t)\|^2_{L^2(\th_2)}dt\nn\\
&+\la\s{\e}(j+1)^5 \int^{T_0}_0\angt^{\f{1-\dl}{2}}{\eta}(t)\|\psi_j(t)\|^2_{L^2(\th_2)}dt\nn\\
\ls_\dl &(j+1)^4\int^{T_0}_0\angt^{\f{1-\dl}{2}}\langle \f{z}{2\angt}w,\psi^2_j(t)\rangle_{\th^2}dt+\f{1}{\la\s{\e}}\int^{T_0}_0\angt^{\f{3-3\dl}{2}}(j+1)^3\|O^i_j,P^i_j \|^2_{L^2(\th_2)}dt\label{psi1}\\
&+(j+1)^4\int^{T_0}_0\angt^{\f{1-3\dl}{2}} \lt\langle M_j\f{\p_zu}{\s{\e}}\int^\i_z \p^{j-1}_x\mathcal{H}d\bar{z}, \psi_j\rt\rangle_{\th_2}dt.\nn
\end{align}
Then summing the \eqref{psi1} over $j\in\bN$ and using \eqref{lowpoint} to bound the first term of the righthand, we can achieve \eqref{psilinear}. \qed

\subsection{Estimates of nonlinear terms }

\begin{lemma}\label{lpsinonlinear}
Under the assumption in \eqref{gassump}, for sufficiently small $\e$, we have the following estimate
{\small
\begin{align}
&\f{1}{\la\s{\e}}\int^{T_0}_0\angt^{\f{3-3\dl}{2}}\sum^\i_{j=0}(j+1)^3\lt\|(\sum^3_{i=1}O^i_j+\sum^5_{i=1}P^i_j)\rt\|^2_{L^2(\th_2)}dt\nn\\
&+\int^{T_0}_0\angt^{\f{1-3\dl}{2}} \sum^\i_{j=0}(j+1)^4\lt\langle M_j\f{\p_zu}{\s{\e}}\int^\i_z \p^{j-1}_x\mathcal{H}d\bar{z}, \psi_j\rt\rangle_{\th_2}dt\nn\\
\leq &C_\dl\f{C^2_\ast}{\la}\s{\e} \int^{T_0}_0 \angt^{-\f{1-\dl}{2}} \lt(\lt\|\bl{u}(t)\rt\|^2_{X_{\tau,\kappa+5/2}}+\lt\|\bl{\mH}(t)\rt\|^2_{X_{\tau,\kappa+1/2}}\rt)dt\label{psinon}\\
     &+C_\dl \f{C^2_\ast}{\la}\s{\e} \int^{T_0}_0 \angt^{\f{1-\dl}{2}} \lt(\lt\|\p_z\bl{u}(t)\rt\|^2_{X_{\tau,\kappa+2}}+\lt\|\p_z\bl{\mH}(t)\rt\|^2_{X_{\tau,\kappa}}\rt)dt+\f{\la\s{\e}}{3}\int^{T_0}_0\angt^{-\f{1-\dl}{2}}\|\psi\|^2_{X_{\tau,\kappa+5/2}}dt.\nn
\end{align}
}
\end{lemma}
\pf
First, using H\"{o}lder inequality, a priori estimates in \eqref{lowpoint} and noting that $(j+1)^2M_j\thickapprox_\dl M_{j-1}$, we have
\begin{align}
&\int^{T_0}_0\angt^{\f{1-3\dl}{2}}\sum^\i_{j=0}(j+1)^4\lt\langle \f{\p_zu}{\s{\e}}\int^\i_z M_j\p^{j-1}_x\mathcal{H}d\bar{z}, \psi_j\rt\rangle_{\th^2}dt\nn\\
\leq &\int^{T_0}_0\angt^{\f{1-3\dl}{2}}\sum^\i_{j=0}(j+1)^2\lt\|\f{\p_zu}{\s{\e}}\rt\|_{L^\i_h L^2_z(\th^2)}\lt\|\int^\i_z \mathcal{H}_{j-1}d\bar{z}\rt\|_{L^2_h L^\i_z}\|\psi_j\|_{L^2(\th_{2})}dt\nn\\
\ls_\dl &\int^{T_0}_0\angt^{\f{3-6\dl}{4}}\sum^\i_{j=0}(j+1)^2\lt\|\f{\p_zu}{\s{\e}}\rt\|_{L^\i_h L^2_z(\th^2)}\lt\| \mathcal{H}_{j-1}\rt\|_{L^2(\th_2)}\|\psi_j\|_{L^2(\th_{2})}dt\label{psi2}\\
\leq & C_\dl \f{ C^2_\ast\s{\e}}{\la}\int^{T_0}_0\angt^{-\f{1-\dl}{2}}\|\mH\|^2_{X_{\tau,\kappa+1/2}}dt+\f{\la\s{\e}}{3}\int^{T_0}_0\angt^{-\f{1-\dl}{2}}\|\psi\|^2_{X_{\tau,\kappa+5/2}}dt.\nn
\end{align}

{\noindent\bf Estimates of $O^1_j,\, O^2_j$.}

Since $O^1_j$ and $O^2_j$ share the same estimate, we only care about $O^1_j$. Noting that
\bes
\bali
|O^1_j|\leq &\sum^{[(j+1)/2]}_{k=1} (k+1)^{-\kappa} (j-k+1)^{-1}|u_{k}| |\p_h u_{j-k}|+\sum^{j}_{k=[(j+1)/2]+1}  (j-k+1)^{-\kappa} |u_{k}| |\p_h u_{j-k}|,
\eali
\ees
then similar as product estimates in \eqref{product1} to \eqref{product3} and using the a priori estimates in Lemma \ref{llowpoint}, we have
\begin{align}
&\f{1}{\la\s{\e}}\int^{T_0}_0\angt^{\f{3-3\dl}{2}}\sum^{\i}_{j=0}(j+1)^3\|(O^1_{j},O^2_j)(t)\|^2_{L^2(\th_2)}dt\nn\\
\ls& \f{1}{\la\s{\e}}\int^{T_0}_0\angt^{\f{4-3\dl}{2}}\lt(\|\p_zu\|^2_{X_{\tau,5,1/2}}\|\p_xu\|^2_{X_{\tau,\kappa+1/2,1/2}}
+\|\p_zu\|^2_{X_{\tau,\kappa+3/2,1/2}}\|\p_xu\|^2_{X_{\tau,5,1/2}}\rt)dt\label{termo1f}\\
\ls_\dl& \f{ C^2_\ast\e^{3/2}}{\la}\int^{T_0}_0\lt(\angt^{-\f{1-\dl}{2}}\|u\|^2_{X_{\tau,\kappa+5/2}}+\angt^{\f{1-\dl}{2}}\|\p_zu\|^2_{X_{\tau,\kappa+2}}\rt)dt.\nn
\end{align}

{\noindent\bf Estimates of $O^3_j$.}

Noting that
{\small
\bes
|O^3|\ls_\dl \sum^{[(j+1)/2]}_{k=1} (k+1)^{-\kappa} (j-k+1)^{-1}|w_{k}| |\p_z u_{j-k}|+\sum^{j-1}_{k=[(j+1)/2]+1}  (k+1)^{-1}(j-k+1)^{-\kappa} |w_{k}| |\p_z u_{j-k}|,
\ees
}
then similar as product estimates in \eqref{product1} to \eqref{product3}, and using incompressibility and the a priori estimates in Lemma \ref{llowpoint}, we have
\begin{align}
&\f{1}{\la\s{\e}}\int^{T_0}_0\angt^{\f{3-3\dl}{2}}\sum^{\i}_{j=0}(j+1)^3\|O^3_{j}(t)\|^2_{L^2(\th_2)}dt\nn\\
\ls& \f{1}{\la\s{\e}}\int^{T_0}_0\angt^{\f{4-3\dl}{2}}\lt(\|\p_zw\|^2_{X_{\tau,5,1/2}}\|\p_zu\|^2_{X_{\tau,\kappa+1/2,1/2}}
+\|\p_zw\|^2_{X_{\tau,\kappa+1/2,1/2}}\|\p_zu\|^2_{X_{\tau,5,1/2}}\rt)dt\nn\\
\ls& \f{1}{\la\s{\e}}\int^{T_0}_0\angt^{\f{4-3\dl}{2}}\lt(\|\bl{u}\|^2_{X_{\tau,7,1/2}}\|\p_zu\|^2_{X_{\tau,\kappa+1/2}}
+\|\bl{u}\|^2_{X_{\tau,\kappa+5/2}}\|\p_zu\|^2_{X_{\tau,5,1/2}}\rt)dt\label{termo3f}\\
\ls_\dl& \f{ C^2_\ast\e^{3/2}}{\la}\int^{T_0}_0\lt(\angt^{-\f{1-\dl}{2}}\|u\|^2_{X_{\tau,\kappa+5/2}}+\angt^{\f{1-\dl}{2}}\|\p_zu\|^2_{X_{\tau,\kappa+2}}\rt)dt.\nn
\end{align}

{\noindent\bf Estimates of $P^1_j,\, P^2_j$.}

Since $P^1_j$ and $P^2_j$ share the same estimate, we only care about $P^1_j$. For term $P^1_j$, using the a priori estimates in \eqref{lowpoint} to obtain that

\begin{align}
|P^1_j|\ls_\dl &C_\ast\s{\e}\angt^{-\f{2+3\dl}{4}}\sum^{[(j+1)/2]}_{k=1} (k+1)^{-\kappa} (j-k+1)^{-3}|u_{k}| \lt|\int^\i_z\p_h\mathcal{H}_{j-k-1}d\bar{z}\rt|\nn\\
 &C_\ast\s{\e}\angt^{-\f{2+3\dl}{4}}\sum^{j-1}_{k=[(j+1)/2]+1}  (j-k+1)^{-\kappa} (k+1)^{-2}|u_{k}| \lt|\int^\i_z\p_h\mathcal{H}_{j-k-1}d\bar{z}\rt|.\nn
\end{align}
Similar as product estimates in \eqref{product1} to \eqref{product3} and using the a priori estimates in Lemma \ref{llowpoint}, we have
\begin{align}
&\f{1}{\la\s{\e}}\int^{T_0}_0\angt^{\f{3-3\dl}{2}}\sum^{\i}_{j=0}(j+1)^3\|(P^1_{j}, P^2_j)(t)\|^2_{L^2(\th^2)}dt\nn\\
\ls_\dl &\f{ C^2_\ast \s{\e}}{\la}\int^{T_0}_0\angt\lt(\|u\|^2_{X_{\tau,5,1/2}}\|\p_h\mH\|^2_{X_{\tau,\kappa-3/2,1/2}}+\|u\|^2_{X_{\tau,\kappa-1/2,1/2}}\|\p_h\mH\|^2_{X_{\tau,5,1/2}}\rt)dt\nn\\
 \ls_\dl&\f{ C^4_\ast\e^{5/2}}{\la}\int^{T_0}_0\angt^{-\f{1-\dl}{2}}\angt^\dl \lt(\|\mH\|^2_{X_{\tau,\kappa+1/2}}+\|u\|^2_{X_{\tau,\kappa+5/2}}\rt)dt\label{termp1f}\\
  \ls_\dl&\f{ C^2_\ast}{\la}\s{\e}\int^{T_0}_0\angt^{-\f{1-\dl}{2}} \lt(\|\mH\|^2_{X_{\tau,\kappa+1/2}}+\|u\|^2_{X_{\tau,\kappa+5/2}}\rt)dt,\nn
\end{align}
where at the last line, we have used \eqref{timelifespan}, which indicates that $\angt^\dl\s{\e}\ls_\dl 1.$

{\noindent\bf Estimates of $P^3_j$.}

For term $P^3_j$, using the a priori estimates in Lemma \eqref{llowpoint} to obtain that
\begin{align}
|P^3_j|\ls_\dl &C_\ast\s{\e}\angt^{-\f{2+3\dl}{4}}\sum^{[(j+1)/2]}_{k=1} (k+1)^{-\kappa} (j-k+1)^{-3}|w_{k}| \lt|\mathcal{H}_{j-k-1}\rt|\nn\\
 &C_\ast\s{\e}\angt^{-\f{2+3\dl}{4}}\sum^{j-1}_{k=[(j+1)/2]+1}  (j-k+1)^{-\kappa} (k+1)^{-2}|w_{k}| \lt|\mathcal{H}_{j-k-1}\rt|.\nn
\end{align}
Similar as product estimates in \eqref{product1} to \eqref{product3}, using incompressibility and the a priori estimates in Lemma \ref{llowpoint}, we have
\begin{align}
&\f{1}{\la\s{\e}}\int^{T_0}_0\angt^{\f{3-3\dl}{2}}\sum^{\i}_{j=0}(j+1)^3\|P^3_j(t)\|^2_{L^2(\th_2)}dt\nn\\
\ls&\f{ C^2_\ast \s{\e} }{\la}\int^{T_0}_0\angt\lt(\|\p_z w\|^2_{X_{\tau,5,1/2}}\|\mH\|^2_{X_{\tau,\kappa-3/2,1/2}}+\|\p_zw\|^2_{X_{\tau,\kappa-1/2,1/2}}\|\mH\|^2_{X_{\tau,5,1/2}}\rt)dt\nn\\
\ls_\dl&\f{ C^2_\ast \s{\e} }{\la}\int^{T_0}_0\angt\lt(\|\bl{u}\|^2_{X_{\tau,7,1/2}}\|\mH\|^2_{X_{\tau,\kappa+1/2}}+\|\bl{u}\|^2_{X_{\tau,\kappa+5/2,1/2}}\|\mH\|^2_{X_{\tau,5,1/2}}\rt)dt\label{termp3f}\\
 \ls&\f{ C^4_\ast \e^{5/2} }{\la}\int^{T_0}_0\angt^{-\f{1-\dl}{2}}\angt^\dl \lt(\|\mH\|^2_{X_{\tau,\kappa+1/2}}+\|\bl{u}\|^2_{X_{\tau,\kappa+5/2}}\rt)dt\nn\\
 \ls&\f{ C^2_\ast }{\la}\s{\e}\int^{T_0}_0\angt^{-\f{1-\dl}{2}} \lt(\|\mH\|^2_{X_{\tau,\kappa+1/2}}+\|\bl{u}\|^2_{X_{\tau,\kappa+5/2}}\rt)dt,\nn
\end{align}
where at the last line of the above inequality, we have used \eqref{timelifespan0}.

{\noindent\bf Estimates of $P^4_j$.}

Using the a prior estimates \eqref{lowpoint} and \eqref{poincare}, we can obtain that
\bes
\bali
&\lt\| P^4_j \rt\|_{L^2(\th^2)}\leq \lt\|\f{\lt[\p_yv\p_zu-\p_yu\p_zv\rt]}{\s{\e}\angt^{\dl-1}}\rt\|_{L^\i}\lt\|M_j\int^\i_z\p^{j-1}_x \mathcal{H}d\bar{z}\rt\|_{L^2(\th_2)}\\
\ls_\dl & C^2_\ast\e^{3/2}\angt^{-\f{2+\dl}{2}}(j+1)^{-2}\lt\| \mathcal{H}_j\rt\|_{L^2(\th_2)}.
\eali
\ees
Then it is easy to see that
\be \label{termp4f}
\bali
&\f{1}{\la\s{\e}}\int^{T_0}_0\angt^{\f{3-3\dl}{2}}\sum^{\i}_{j=0}(j+1)^3\|P^4_{j}(t)\|^2_{L^2(\th_2)}dt\\
\ls_\dl &\f{C^2_\ast}{\la} \s{\e}\int^{T_0}_0\angt^{-\f{1-\dl}{2}}\lt\| \mathcal{H}\rt\|^2_{X_{\tau,\kappa}}dt.
\eali
\ee

{\noindent\bf Estimates of $P^5_j$.}

Using the a prior estimates \eqref{lowpoint}, we can obtain that
\bes
\bali
&\lt\| P^5_j \rt\|_{L^2(\th_2)}\leq \lt\|\f{\p^2_z u}{\s{\e}\angt^{\dl-1}}\rt\|_{L^\i}\lt\|M_j\p^{j-1}_x \mathcal{H}\rt\|_{L^2(\th_2)}\\
\ls_\dl & C_\ast\s{\e}\angt^{-\f{4+3\dl}{4}}(j+1)^{-2}\lt\| \mathcal{H}_{j-1}\rt\|_{L^2(\th_2)}.
\eali
\ees
Then it is easy to see that
\be\label{termp5f}
\bali
&\f{1}{\la\s{\e}}\int^{T_0}_0\angt^{\f{3-3\dl}{2}}\sum^{\i}_{j=0}(j+1)^3\|P^5_{j}(t)\|^2_{L^2(\th^2)}dt\\
\ls_\dl & \f{C^2_\ast}{\la}\s{\e}\int^{T_0}_0\angt^{-\f{1-\dl}{2}}\lt\|\mathcal{H}\rt\|^2_{X_{\tau,\kappa}}dt.
\eali
\ee

Combining estimates in \eqref{psi2}, \eqref{termo1f}, \eqref{termo3f}, \eqref{termp1f}, \eqref{termp3f}, \eqref{termp4f} and \eqref{termp5f}, we obtain \eqref{psinon}. \qed

{\noindent\bf Proof of Proposition \ref{punknown}.}

From Lemma \ref{lpsi} and Lemma \ref{lpsinonlinear}, we have achieved that
\begin{align}
&\angt^{\f{1-\dl}{2}}\|\psi(t)\|^2_{X_{\tau,\kappa+2}}+\dl\int^{T_0}_0\angt^{\f{1-\dl}{2}}\|\p_z \psi(t)\|^2_{X_{\tau,\kappa+2}}dt+\la \s{\e} \int^{T_0}_0\angt^{\f{1-\dl}{2}}{\eta}(t)\|\psi(t)\|^2_{_{X_{\tau,\kappa+5/2}}}dt\nn\\
\ls_\dl& \|{u}(0)\|^2_{X_{\tau_0,\kappa+2}}+\f{C^2_\ast}{\la}\s{\e} \int^{T_0}_0 \angt^{-\f{1-\dl}{2}} \lt(\lt\|\bl{u}(t)\rt\|^2_{X_{\tau,\kappa+5/2}}+\lt\|\bl{\mH}(t)\rt\|^2_{X_{\tau,\kappa+1/2}}\rt)dt\label{equiv0}\\
     &+\f{C^2_\ast}{\la}\s{\e} \int^{T_0}_0 \angt^{\f{1-\dl}{2}} \lt(\lt\|\p_z\bl{u}(t)\rt\|^2_{X_{\tau,\kappa+2}}+\lt\|\p_z\bl{\mH}(t)\rt\|^2_{X_{\tau,\kappa}}\rt)dt.\nn
\end{align}
Besides, from the definition of $\psi_j$ in \eqref{defpsi} and using \eqref{solutiongevrey} in Lemma \ref{solutiongevrey}, we see that, for $\t{\kappa}>0$,
\begin{align}
\|u(t)\|^2_{X_{\tau,\t{\kappa}+2}}\ls& \|\psi(t)\|^2_{X_{\tau,\t{\kappa}+2}}+\angt^{5/2-2\dl}\e^{-1}\|\p_z u\|^2_{L^\i}\|\mH(t)\|^2_{X_{\tau,\t{\kappa}}}\nn\\
                          \ls & \|\psi(t)\|^2_{X_{\tau,\t{\kappa}+2}}+C^2_\ast\e\|\mH(t)\|^2_{X_{\tau,\t{\kappa}}}. \label{equiv1}
\end{align}
Similarly, we can obtain that
\begin{align}
\|\p_zu(t)\|^2_{X_{\tau,\t{\kappa}+2}}\ls_\dl \|\p_z\psi(t)\|^2_{X_{\tau,\t{\kappa}+2}}+C^2_\ast\e\|\p_z\mH(t)\|^2_{X_{\tau,\t{\kappa}}}. \label{equiv2}
\end{align}
Inserting \eqref{equiv1} and \eqref{equiv2} into \eqref{equiv0} and by letting $\e$ is sufficiently small, we can achieve \eqref{uesti} in Proposition \ref{punknown}. \qed

\section{Estimates of the good unknowns $\bl{g}$}\label{secg}

In this section, we focus on Gevrey-2 estimates of the linearly good unknowns $\bl{g}$ and its $z-$ derivatives up to second order. It will induce faster decay rate for low order Gevrey-2 energy of the unknowns $\bl{u}$ as displayed in \eqref{solutiongevrey}.

Below we set
\bes
\kappa_0=12,\q \kappa_1=10,\q \kappa_2=8.
\ees

\subsection{Estimates of $\bl{g}$}

\begin{lemma}\label{lemg}

Under the assumption of \eqref{gassump}, for sufficiently small $\e$, there exits a constant $C_\dl$ such that for any $t\in (0,T_0]$, we have the following estimate.
\begin{align}
&\angt^{\f{5-\dl}{2}} \lt\| \bl{g} (t)\rt\|^2_{X_{\tau,\kappa_0}}+\dl\int^{T_0}_0\angt^{\f{5-\dl}{2}}\lt\|\p_z\bl{g}(t)\rt\|^2_{X_{\tau,\kappa_0}}dt+\la\s{\e} \int^{T_0}_0\angt^{\f{5-\dl}{2}}{\eta}(t) \lt\|\bl{g}(t)\rt\|^2_{X_{\tau,\kappa_0+1/2}}dt\nn\\
\leq & C_\dl \lt\| \bl{g} (0)\rt\|^2_{X_{\tau_0,\kappa_0}}+C_\dl \f{C^2_\ast}{\la} \e^{3/2}\int^{T_0}_0 \angt^{\f{1-\dl}{2}}\lt\|\p_z\bl{u}(t)\rt\|^2_{X_{\tau,\kappa+2}}dt\label{goodg}\\
&  +C_\dl \f{C^2_\ast}{\la} \e^{3/2}\int^{T_0}_0 \lt(\angt^{\f{3+\dl}{2}}\lt\|\bl{g}(t)\rt\|^2_{X_{\tau,\kappa_0+1/2}}+\angt^{\f{5-\dl}{2}}\lt\|\p_z\bl{g}(t)\rt\|^2_{X_{\tau,\kappa_0}}\rt)dt.\nn
\end{align}
\end{lemma}

\pf We only show estimates of $g$, since $\t{g}$ follows the same line. By direct computation, we see that $g$ satisfies the following
\be\label{3dgun1}
\lt\{
\bali
&\partial_{t} g +\f{1}{\angt}g-\p^2_z g=-\left({u} \p_x+v \partial_y+w\p_z \right) g+\f{wu}{2\angt}+\lt(\p_zu\p_yv-\p_zv\p_y u\rt), \\
&\qq\qq\qq :=\sum^6_{i=1} R^i,\\
&\p_z g\big|_{z=0}=0, \quad \lim_{z \rightarrow+\infty} g=0,\q g\big|_{t=0}=g_{0}=\p_z u_0+\f{z}{2}u_0.
\eali
\rt.
\ee
Now applying $M_{j,\kappa_0}\p^j_x$  to the first equation of \eqref{3dgun1} and denoting $f_{j,x,\kappa_0}$ by $f_{j,\kappa_0}$ for a function $f$, we can obtain that
\bes
\bali
&\partial_{t} g_{j,\kappa_0}+\la\s{\e}\eta(t)(j+1)g_{j,\kappa_0}-\partial_{z}^{2} g_{j,\kappa_0}+\f{1}{\angt}g_{j,\kappa_0}=\sum^6_{i=1}R^i_{j,\kappa_0}.
\eali
\ees

Performing energy estimates similar as \eqref{auxilh5}, we have
{\small
\be
\bali
&\angt^{\f{5-\dl}{2}}\|g_{j,\kappa_0}(t)\|^2_{L^2(\th_2)}+\dl\int^{T_0}_0\angt^{\f{5-\dl}{2}}\|\p_z g_{j,\kappa_0}(t)\|^2_{L^2(\th_2)}dt+\la\s{\e}(j+1) \int^{T_0}_0\angt^{\f{5-\dl}{2}}{\eta}(t)\|g_{j,\kappa_0}(t)\|^2_{L^2(\th_2)}dt\nn\\
\ls&\f{1}{\la\s{\e}}\int^{T_0}_0\angt^{\f{7-3\dl}{2}}(j+1)^{-1}\sum^6_{i=1}\|R^i_{j,\kappa_0}(t)\|^2_{L^2(\th_2)}dt.\nn
\eali
\ee
}
Summing the above inequality over $j\in\bN$, we can obtain that
\begin{align}
&\angt^{\f{5-\dl}{2}} \lt\|{g} (t)\rt\|^2_{X_{\tau,\kappa_0}}+\dl\int^{T_0}_0\angt^{\f{5-\dl}{2}}\lt\|\p_z{g}(t)\rt\|^2_{X_{\tau,\kappa_0}}dt+\la\s{\e} \int^{T_0}_0\angt^{\f{5-\dl}{2}}{\eta}(t) \lt\|{g}(t)\rt\|^2_{X_{\tau,\kappa_0+1/2}}dt\nn\\
\ls_\dl & \f{1}{\la\s{\e}}\int^{T_0}_0\angt^{\f{7-3\dl}{2}}(j+1)^{-1}\sum^\i_{j=0}\sum^6_{i=1}\|R^i_{j,\kappa_0}(t)\|^2_{L^2(\th_2)}dt.\nn
\end{align}

Before we proceed to estimate the nonlinear terms, we make the following claim.

\begin{claim}\label{claim}
For $0<\t{\kappa}\in\bR$, $0\leq \nu\leq 1$ and $1<m\in\bN$, we have the following embedding in Gevrey-2 spaces.
\begin{align}
\|\p_hf\|^2_{X_{\tau,\t{\kappa},\nu }} \ls_m\|f\|^{\f{2}{m}}_{X_{\tau,\t{\kappa}+2m,\nu }}\|f\|^{2-\f{2}{m}}_{X_{\tau,\t{\kappa},\nu }}. \label{gnineq1}
\end{align}
\pf
Using Gagliardo-Nirenberg inequality in $(x,y)$ variables, we have
\begin{align}
\|\p_h f_{j,\t{\kappa}}\|_{L^2(\th_{2\nu})}\ls& \|\p^m_h f_{j,\t{\kappa}}\|^{\f{1}{m}}_{L^2(\th_{2\nu})}\|f_{j,\t{\kappa}}\|^{1-\f{1}{m}}_{L^2(\th_{2\nu})}\nn\\
       \ls_m& \lt((j+m+1)^{2m}\| f_{j+m,\t{\kappa}}\|\rt)^{\f{1}{m}}_{L^2(\th_{2\nu})}\|f_{j,\t{\kappa}}\|^{1-\f{1}{m}}_{L^2(\th_{2\nu})}.\label{gnineq}
\end{align}

Then using discrete H\"{o}lder inequality, we can obtain that
\begin{align}
\|\p_hf\|^2_{X_{\tau,\t{\kappa},\nu }}=&\sum^\i_{j=0} \|\p_h f_{j,\t{\kappa}}\|_{L^2(\th_{2\nu})}\nn\\
\ls&  \lt(\sum^\i_{j=0}(j+m+1)^{2m}\| f_{j+m,\t{\kappa}}\|^2_{L^2(\th_{2\nu})}\rt)^{\f{1}{m}}\lt(\sum^\i_{j=0}\|f_{j,\t{\kappa}}\|^2_{L^2(\th_{2\nu})}\rt)^{1-\f{1}{m}}\nn\\
\ls&\|f\|^{\f{2}{m}}_{X_{\tau,\t{\kappa}+2m,\nu }}\|f\|^{2-\f{2}{m}}_{X_{\tau,\t{\kappa},\nu }}.\nn
\end{align}
which is \eqref{gnineq}.  \qed
\end{claim}

{\noindent\bf Estimate of term $R^1_j$ and $R^2_j$.}

by using \eqref{poincare1} and the definition of $g$, we have the fact that
\bes
\|\p_h g\|^2_{X_{\tau,\kappa_0-1/2,1/2}}\ls \|\p_z \p_h u\|^2_{X_{\tau,\kappa_0-1/2,1/2}}.
\ees

Using the product estimate in \eqref{product2} and the above inequality, we can have
\begin{align}
&\sum^\i_{j=0}(j+1)^{-1}\|(R^{1},R^2_j)_{j,\kappa_0}(t)\|^2_{L^2(\th_2)}\nn\\
\ls &\angt^{1/2}\lt(\|\p_zu\|^2_{X_{\tau,5,1/2}}\|\p_h g\|^2_{X_{\tau,\kappa_0-1/2,1/2}}+\|\p_zu\|^2_{X_{\tau,\kappa_0-1/2,1/2}}\|\p_h g\|^2_{X_{\tau,5,1/2}}\rt)\label{termr1s}\\
\ls &\angt^{1/2}\lt(\|\p_zu\|^2_{X_{\tau,5,1/2}}\|\p_z\p_h u\|^2_{X_{\tau,\kappa_0-1/2,1/2}}+\|\p_zu\|^2_{X_{\tau,\kappa_0-1/2,1/2}}\| g\|^2_{X_{\tau,7,1/2}}\rt).\nn
\end{align}
From \eqref{gnineq1}, we have
\begin{align}
&\|\p_h\p_zu\|^2_{X_{\tau,\kappa_0-1/2,1/2}}\ls \|\p_zu\|^{\f{2(m-1)}{m}}_{X_{\tau,\kappa_0-1/2,1/2}}\|\p_zu\|^{\f{2}{m}}_{X_{\tau,\kappa_0-1/2+2m,1/2}}.\nn
\end{align}
Inserting the above inequality into \eqref{termr1s} and using the a priori estimates in \eqref{solutiongevrey}, we can obtain that
\begin{align}
&\sum^\i_{j=0}(j+1)^{-1}\|(R^{1},R^2_j)_{j,\kappa_0}(t)\|^2_{L^2(\th_2)}\nn\\
\ls &\angt^{1/2}\lt(\|\p_zu\|^2_{X_{\tau,5,1/2}}\|\p_zu\|^{\f{2(m-1)}{m}}_{X_{\tau,\kappa_0-1/2,1/2}}\|\p_zu\|^{\f{2}{m}}_{X_{\tau,\kappa_0-1/2+2m,1/2}}+\|\p_zu\|^2_{X_{\tau,\kappa_0-1/2,1/2}}\| g\|^2_{X_{\tau,7,1/2}}\rt) \nn\\
\ls &C^2_\ast\e^2\angt^{-\f{4-\dl}{2}}\lt(\|g\|^{\f{2(m-1)}{m}}_{X_{\tau,\kappa_0-1/2}}\|\p_zu\|^{\f{2}{m}}_{X_{\tau,\kappa_0-1/2+2m}}
+\|g\|^2_{X_{\tau,\kappa_0-1/2}}\rt).\nn
\end{align}
Here we have used the fact that
\bes
\|\p_zu\|_{X_{\tau,\kappa_0-1/2,1/2}}\ls \|g\|^2_{X_{\tau,\kappa_0-1/2}}.
\ees
Then by using Young inequality, we can obtain that
\begin{align}
&\f{1}{\la\s{\e}}\int^{T_0}_0\angt^{\f{7-3\dl}{2}}\sum^\i_{j=0}(j+1)^{-1}\|(R^{1},R^2)_{j,\kappa_0}(t)\|^2_{L^2(\th^2)}dt\nn\\
\ls & \f{C_\dl C^2_\ast\e^{3/2}}{\la} \int^{T_0}_0\angt^{\f{3-2\dl}{2}}\lt(\|g\|^{\f{2(m-1)}{m}}_{X_{\tau,\kappa_0-1/2}}\|\p_z u\|^{\f{2}{m}}_{X_{\tau,\kappa_0-1/2+2m}}
+\|g\|^2_{X_{\tau,\kappa_0-1/2}}\rt)dt\nn\\
 \ls&\f{C^2_\ast\e^{3/2}}{\la}\int^{T_0}_0\angt^{\lt(\f{3-2\dl}{2}-\f{1-\dl}{2m}\rt)\f{m}{m-1}}\|g\|^2_{X_{\tau,\kappa_0-1/2}}dt
 +\f{C^2_\ast\e^{3/2}}{\la}\int^{T_0}_0\angt^{\f{1-\dl}{2}}\|\p_zu\|^2_{X_{\tau,\kappa_0-1/2+2m}}dt\label{termr1fourth}\\
    &+\f{C^2_\ast\e^{3/2}}{\la}\int^{T_0}_0\angt^{\f{3-2\dl}{2}}\|g\|^2_{X_{\tau,\kappa_0-1/2}}dt.\nn
\end{align}
By choosing $m:=\lt[\f{1}{\dl}\rt]$, we can have
\bes
\lt(\f{3-2\dl}{2}-\f{1-\dl}{2m}\rt)\f{m}{m-1}\leq \f{3+\dl}{2}, \q \kappa_0-1/2+2m\leq \kappa+2.
\ees
Then inserting the above into \eqref{termr1fourth} implies that
\begin{align}
&\f{1}{\la\s{\e}}\int^{T_0}_0\angt^{\f{7-3\dl}{2}}\sum^\i_{j=0}(j+1)^{-1}\|(R^{1},R^2)_{j,\kappa_0}(t)\|^2_{L^2(\th^2)}dt\nn\\
 \ls&\f{C^2_\ast\e^{3/2}}{\la}\int^{T_0}_0\angt^{\f{3+\dl}{2}}\|g\|^2_{X_{\tau,\kappa_0+1/2}}dt+\f{C^2_\ast\e^{3/2}}{\la}\int^{T_0}_0\angt^{\f{1-\dl}{2}}\|\p_z u\|^2_{X_{\tau,\kappa+2}}dt.\label{termr1third}
\end{align}

{\noindent\bf Estimate of term $R^3_j$.}

Using the product estimate in \eqref{product2}, the incompressibility, we can obtain that
\begin{align}
&\sum^\i_{j=0}(j+1)^{-1}\|R^3_{j,\kappa_0}(t)\|^2_{L^2(\th_2)}\nn\\
\ls &\angt^{1/2}\lt(\|\p_zw\|^2_{X_{\tau,5,1/2}}\|\p_z g\|^2_{X_{\tau,\kappa_0-1/2,1/2}}+\|\p_zw\|^2_{X_{\tau,\kappa_0-1/2,1/2}}\|\p_z g\|^2_{X_{\tau,5,1/2}}\rt)\nn\\
\ls &\angt^{1/2}\lt(\|\bl{u}\|^2_{X_{\tau,7,1/2}}\|\p_z g\|^2_{X_{\tau,\kappa_0-1/2,1/2}}+\|\p_h\bl{u}\|^2_{X_{\tau,\kappa_0-1/2,1/2}}\| \p_zg\|^2_{X_{\tau,5,1/2}}\rt).  \label{termr3second}
\end{align}
From \eqref{gnineq1}, we have
\begin{align}
&\|\p_h\bl{u}\|^2_{X_{\tau,\kappa_0-1/2,1/2}}\ls \|\bl{u}\|^{\f{2(m-1)}{m}}_{X_{\tau,\kappa_0-1/2,1/2}}\|\bl{u}\|^{\f{2}{m}}_{X_{\tau,\kappa_0-1/2+2m,1/2}}.\nn
\end{align}
Inserting the above inequality into \eqref{termr3second} and using the a priori estimates in \eqref{solutiongevrey}, we can obtain that
{\small
\begin{align}
&\sum^\i_{j=0}(j+1)^{-1}\|R^3_{j,\kappa_0}(t)\|^2_{L^2(\th^2)}\nn\\
\ls &\angt^{1/2}\lt(\|\bl{u}\|^{\f{2(m-1)}{m}}_{X_{\tau,\kappa_0-1/2,1/2}}\|\bl{u}\|^{\f{2}{m}}_{X_{\tau,\kappa_0-1/2+2m,1/2}}\| \p_zg\|^2_{X_{\tau,5}}
+\|\bl{u}\|^2_{X_{\tau,7,1/2}}\| \p_zg\|^2_{X_{\tau,\kappa_0}}\rt)  \nn\\
\ls &C^2_\ast\e^2\lt(\angt^{-\f{4-\dl}{2}}\|g\|^{\f{2(m-1)}{m}}_{X_{\tau,\kappa_0-1/2}}\|\p_z\bl{u}\|^{\f{2}{m}}_{X_{\tau,\kappa_0-1/2+2m}}
+\angt^{-\f{2-\dl}{2}}\|\p_z g\|^2_{X_{\tau,\kappa_0}}\rt),\nn
\end{align}
}
where at the last line, we have used the fact that
\bes
\|\bl{u}\|_{X_{\tau,\kappa_0-1/2,1/2}}\ls \angt^{1/2}\|\bl{g}\|_{X_{\tau,\kappa_0-1/2}},\q \|\bl{u}\|_{X_{\tau,\kappa_0-1/2,1/2}}\ls \angt^{1/2}\|\p_z\bl{u}\|_{X_{\tau,\kappa_0-1/2,1/2}}.
\ees
Then the rest is the same as estimate of \eqref{termr1third}, we can obtain that
\begin{align}
&\f{1}{\la\s{\e}}\int^{T_0}_0\angt^{\f{7-3\dl}{2}}\sum^\i_{j=0}(j+1)^{-1}\|(R^3_{j,\kappa_0}(t)\|^2_{L^2(\th^2)}dt\nn\\ \ls&\f{C^2_\ast\e^{3/2}}{\la}\int^{T_0}_0\lt(\angt^{\f{3+\dl}{2}}\|g\|^2_{X_{\tau,\kappa_0+1/2}}
+\angt^{\f{5-\dl}{2}}\|\p_zg\|^2_{X_{\tau,\kappa_0}}\rt)dt+\f{C^2_\ast\e^{3/2}}{\la}\int^{T_0}_0\angt^{\f{1-\dl}{2}}\|\p_z u\|^2_{X_{\tau,\kappa+2}}dt.\label{termr3third}
\end{align}

{\noindent\bf Estimate of term $R^4_j$.}

Using the product estimate in \eqref{product2}, the incompressibility and \eqref{poincare}, we can obtain that
\begin{align}
&\sum^\i_{j=0}(j+1)^{-1}\|R^4_{j,\kappa_0}(t)\|^2_{L^2(\th_2)}\nn\\
\ls &\angt^{1/2}\lt(\|\p_zw\|^2_{X_{\tau,5,1/2}}\|\angt^{-1} u\|^2_{X_{\tau,\kappa_0-1/2,1/2}}+\|\p_zw\|^2_{X_{\tau,\kappa_0-1/2,1/2}}\|\angt^{-1} u\|^2_{X_{\tau,5,1/2}}\rt)\nn\\
\ls &\angt^{1/2}\lt(\|\p_zw\|^2_{X_{\tau,5,1/2}}\|\p^2_z u\|^2_{X_{\tau,\kappa_0-1/2,1/2}}+\|\p_zw\|^2_{X_{\tau,\kappa_0-1/2,1/2}}\|\p^2_zu\|^2_{X_{\tau,5,1/2}}\rt)\nn\\
\ls &\angt^{1/2}\lt(\|\bl{u}\|^2_{X_{\tau,7,1/2}}\|\p_z g\|^2_{X_{\tau,\kappa_0-1/2,3/4}}+\|\p_h\bl{u}\|^2_{X_{\tau,\kappa_0-1/2,1/2}}\| \p_zg\|^2_{X_{\tau,5}}\rt),\nn
\end{align}
where at the last line, we have used the fact that
\bes
\|\p^2_z{u}\|_{X_{\tau,\kappa_0-1/2,1/2}}\ls \|\p_z{g}\|_{X_{\tau,\kappa_0-1/2,3/4}},\q \|\p^2_z{u}\|_{X_{\tau,5,1/2}}\ls \|\p_z{g}\|_{X_{\tau,5}}.
\ees
Then the rest is the same as \eqref{termr3second}.

{\noindent\bf Estimate of term $R^5_j$ and $R^6_j$.}\\

Using the product estimate in \eqref{product2}, we can have
\begin{align}
&\sum^\i_{j=0}(j+1)^{-1}\|(R^5,R^6)_{j,\kappa_0}(t)\|^2_{L^2(\th^2)}\nn\\
\ls &\angt^{1/2}\lt(\|\p_h\bl{u}\|^2_{X_{\tau,5,1/2}}\|\p^2_z \bl{u}\|^2_{X_{\tau,\kappa_0-1/2,1/2}}+\|\p_h\bl{u}\|^2_{X_{\tau,\kappa_0-1/2,1/2}}\|\p^2_z \bl{u}\|^2_{X_{\tau,5,1/2}}\rt)\nn\\
\ls &\angt^{1/2}\lt(\|\bl{u}\|^2_{X_{\tau,7,1/2}}\|\p_z g\|^2_{X_{\tau,\kappa_0-1/2,3/4}}+\|\p_h\bl{u}\|^2_{X_{\tau,\kappa_0-1/2,1/2}}\| \p_zg\|^2_{X_{\tau,5}}\rt).\nn
\end{align}
Then the rest is the same as \eqref{termr3second}.

From \eqref{termr1third} and \eqref{termr3third}, we can obtain \eqref{goodg} in Lemma \ref{lemg}. \qed

\subsection{Estimates of $\p_z\bl{g}$}

\begin{lemma}\label{lzg}

Under the assumption of \eqref{gassump}, for sufficiently small $\e$, there exits a constant $C_\dl$ such that for any $t\in (0,T_0]$, we have the following estimate.
\begin{align}
&\angt^{\f{7-\dl}{2}} \lt\| \p_z\bl{g}(t)\rt\|^2_{X_{\tau,\kappa_1}}+\dl\int^{T_0}_0\angt^{\f{7-\dl}{2}}\lt\|\p^2_z\bl{g}(t)\rt\|^2_{X_{\tau,\kappa_1}}dt+\la\s{\e} \int^{T_0}_0\angt^{\f{7-\dl}{2}}{\eta}(t) \lt\|\p_z\bl{g}(t)\rt\|^2_{X_{\tau,\kappa_1+1/2}}dt\nn\\
\leq & C_\dl \lt\| \p_z\bl{g}(0)\rt\|^2_{X_{\tau,\kappa_1}}+\int^{T_0}_0 \angt^{\f{5-\dl}{2}}\lt\|\p_z\bl{g}(t)\rt\|^2_{X_{\tau,\kappa_1}}dt\label{goodzg}\\
      &+C_\dl\f{C^2_\ast}{\la} \e^{3/2}\int^{T_0}_0 \lt(\angt^{\f{3+\dl}{2}}\lt\|\bl{g}(t)\rt\|^2_{X_{\tau,\kappa_0+1/2}}+\angt^{\f{5-\dl}{2}}\lt\|\p_z\bl{g}(t)\rt\|^2_{X_{\tau,\kappa_0}}\rt)dt.\nn
\end{align}
\end{lemma}
\pf Now applying $M_{j,\kappa_1}\p_z\p^j_x$  to the first equation of \eqref{3dgun1} and denoting $f_{j,x,\kappa_1}$ by $f_{j,\kappa_1}$ for a function $f$, we can obtain that
\bes
\bali
&\lt[\partial_{t} +\la\s{\e}(j+1)-\partial_{z}^{2}+\f{1}{\angt} \rt] \p_z g_{j,\kappa_1}=\sum^{6}_{i=1}\p_z R^i_{j,\kappa_1}.
\eali
\ees
Performing space variable energy estimates, we can have
\begin{align}
&\f{d}{dt}\|\p_zg_{j,\kappa_1}(t)\|^2_{L^2(\th_2)}+\dl\|\p^2_zg_{j,\kappa_1}(t)\|^2_{L_2(\th_2)}+\f{5-\dl}{2\angt}\| \p_zg_{j,\kappa_1}(t)\|^2_{L^2(\th_2)}+2(j+1)\la\s{\e}{\eta}(t)\|\p_zg_{j,\kappa_1}(t)\|^2_{L^2(\th_2)}\nn\\
\leq &2 \lt| \lt\langle \sum^6_{i=1}\p_z R^i_{j,\kappa_1},\p_z g_{j,\kappa_1}\rt\rangle_{\th_2}\rt|.\nn
\end{align}
Multiplying the above equality by $\angt^{\f{7-\dl}{2}}$ and using integration by parts for the righthand of the above inequality, and then integrating the resulted equation from $0$ to $t$ for any $t\in (0,T_0]$, we can achieve that
{\small
\be
\bali
&\angt^{\f{7-\dl}{2}}\|\p_z g_{j,\kappa_1}(t)\|^2_{L^2(\th_2)}+\dl\int^{T_0}_0\angt^{\f{7-\dl}{2}}\|\p^2_z g_{j,\kappa_1}(t)\|^2_{L^2(\th_2)}dt+2\la\s{\e}(j+1) \int^{T_0}_0\angt^{\f{7-\dl}{2}}{\eta}(t)\|\p_z g_{j,\kappa_1}(t)\|^2_{L^2(\th_2)}dt\nn\\
\leq &\int^{T_0}_0\angt^{\f{5-\dl}{2}}\|\p_z g_{j,\kappa_1}(t)\|^2_{L^2(\th_2)}dt+\int^{T_0}_0 \angt^{\f{7-\dl}{2}} \lt| \lt\langle \sum^6_{i=1}R^i_{j,\kappa_1},\p^2_z g_{j,\kappa_1}+\f{z}{2\angt}\p_z g_{j,\kappa_1}\rt\rangle_{\th_2}\rt|dt.
\eali
\ee
}
By using Cauchy inequality and \eqref{poincare1} to the righthand of above inequality, and then summing the resulted equations over $j\in\bN$, we can obtain that
\be
\bali
&\angt^{\f{7-\dl}{2}}\|\p_zg(t)\|^2_{X_{\tau,\kappa_1}}+\dl\int^{T_0}_0\angt^{\f{7-\dl}{2}}\|\p^2_z g(t)\|^2_{X_{\tau,\kappa_1}}dt+\la\s{\e}(j+1) \int^{T_0}_0\angt^{\f{7-\dl}{2}}{\eta}(t)\|\p_zg(t)\|^2_{X_{\tau,\kappa_1}}dt\\
\leq &\int^{T_0}_0\angt^{\f{5-\dl}{2}}\|\p_z g(t)\|^2_{X_{\tau,\kappa_1}}dt+C_\dl\int^{T_0}_0 \angt^{\f{7-\dl}{2}} \sum^6_{i=1}\|R^i\|^2_{X_{\tau,\kappa_1}}dt.\label{zgest1}
\eali
\ee
\begin{lemma}\label{lem7.4}
We have the following estimates
\begin{align} \sum^6_{i=1}\|R^i\|^2_{X_{\tau,\kappa_1}}\ls\angt^{3/2}\lt(\|g\|^2_{X_{\tau,7}}\|\p_zg\|^2_{X_{\tau,\kappa_1+2}}+\|\p_zg\|^2_{X_{\tau,7}}\|g\|^2_{X_{\tau,\kappa_1+2}}\rt).\label{zgesti2}
\end{align}
\end{lemma}

\pf Proof of this Lemma is repeatedly use of \eqref{poincare} and \eqref{poincare1} in Lemma \ref{lpoincare}, product estimates in \eqref{product1} to \eqref{product3} in Lemma \ref{lemproduct} and the relation between $\bl{u}$ and $\bl{g}$. Since it is a routing estimate, we omit the details. \qed

Then by using the a priori estimates in \eqref{solutiongevrey} and \eqref{zgesti2}, we see that
\begin{align}
&\int^{T_0}_0 \angt^{\f{7-\dl}{2}} \sum^6_{i=1}\|R^i\|^2_{X_{\tau,\kappa_1}}dt\ls C^2_\ast\e^2\int^{T_0}_0 \lt(\angt^{\f{3+\dl}{2}}\|g\|^2_{X_{\tau,\kappa_0+1/2}}+\angt^{\f{5-\dl}{2}}\|\p_zg\|^2_{X_{\tau,\kappa_0}}\rt)dt.\nn
\end{align}

Inserting the above inequality into \eqref{zgest1}, we can obtain \eqref{goodzg} in Lemma \ref{lzg}. \qed

\subsection{Estimates of $\p^2_z\bl{g}$}

\begin{lemma}\label{lzzg}

Under the assumption of \eqref{gassump}, for sufficiently small $\e$, there exits a constant $C_\dl$ such that for any $t\in (0,T_0]$, we have the following estimate.
\begin{align}
&\angt^{\f{9-\dl}{2}} \lt\| \p^2_z\bl{g}(t)\rt\|^2_{X_{\tau,\kappa_2}}+\dl\int^{T_0}_0\angt^{\f{9-\dl}{2}}\lt\|\p^3_z\bl{g}(t)\rt\|^2_{X_{\tau,\kappa_2}}dt+\la\s{\e} \int^{T_0}_0\angt^{\f{9-\dl}{2}}{\eta}(t) \lt\|\p^2_z\bl{g}(t)\rt\|^2_{X_{\tau,\kappa_3+1/2}}dt\nn\\
\leq &C_\dl \lt\| \p^2_z\bl{g}(0)\rt\|^2_{X_{\tau_0,\kappa_2}}+\int^{T_0}_0 \angt^{\f{7-\dl}{2}}\lt\|\p^2_z\bl{g}(t)\rt\|^2_{X_{\tau,\kappa_2}}dt.\label{goodzzg}\\
    &+C_\dl\f{C^2_\ast}{\la} \e^{3/2}\int^{T_0}_0 \lt(\angt^{\f{3+\dl}{2}}\lt\|\bl{g}(t)\rt\|^2_{X_{\tau,\kappa_0+1/2}}+\angt^{\f{5-\dl}{2}}\lt\|\p_z\bl{g}(t)\rt\|^2_{X_{\tau,\kappa_0}}++\angt^{\f{7-\dl}{2}}\lt\|\p^2_z\bl{g}(t)\rt\|^2_{X_{\tau,\kappa_1}}\rt)dt.\nn
\end{align}

\end{lemma}

Now applying $M_{j,\kappa_2}\p^2_z\p^j_x$  to the first equation of \eqref{3dgun1} and denoting $f_{j,x,\kappa_2}$ by $f_{j,\kappa_2}$ for a function $f$, we can obtain that
\be\label{zzgoodfirst}
\bali
&\lt[\partial_{t} +\la\s{\e}(j+1)-\partial_{z}^{2}+\f{1}{\angt} \rt] \p^2_z g_{j,\kappa_2}=\sum^{6}_{i=1}\p^2_z R^i_{j,\kappa_2}.\nn
\eali
\ee
Performing spacial energy estimates, we can have
\begin{align}
&\f{d}{dt}\|\p^2_zg_{j,\kappa_2}(t)\|^2_{L^2(\th_2)}+\dl\|\p^3_zg_{j,\kappa_2}(t)\|^2_{L^2(\th_2)}+\f{5-\dl}{2\angt}\| \p^2_zg_{j,\kappa_2}(t)\|^2_{L^2(\th_2)}+2(j+1)\la\s{\e}{\eta}(t)\|\p^2_zg_{j,\kappa_2}(t)\|^2_{L^2(\th_2)}\nn\\
\leq &2 \lt| \lt\langle \sum^6_{i=1}\p^2_z R^i_{j,\kappa_2},\p^2_z g_{j,\kappa_2}\rt\rangle_{\th_2}\rt|.\nn
\end{align}
Multiplying the above equality by $\angt^{\f{9-\dl}{2}}$ and then integrating the resulted equation from $0$ to $t$ for any $t\in (0,T_0]$, we can achieve that
\bes
\bali
&\angt^{\f{9-\dl}{2}}\|\p^2_zg_{j,\kappa_2}(t)\|^2_{L^2(\th_2)}+\dl\int^{T_0}_0\angt^{\f{9-\dl}{2}}\|\p^3_z g_{j,\kappa_2}(t)\|^2_{L^2(\th_2)}dt+2\la\s{\e}(j+1) \int^{T_0}_0\angt^{\f{9-\dl}{2}}{\eta}(t)\|\p^2_zg_{j,\kappa_2}(t)\|^2_{L^2(\th_2)}dt\\
\leq &\int^{T_0}_0\angt^{\f{7-\dl}{2}}\|\p^2_z g_{j,\kappa_2}(t)\|^2_{L^2(\th_2)}dt+\int^{T_0}_0 \angt^{\f{9-\dl}{2}} \lt| \lt\langle \sum^6_{i=1}\p^2_zR^i_{j,\kappa_2},\p^2_z g_{j,\kappa_2}\rt\rangle_{\th_2}\rt|dt.
\eali
\ees
By using Cauchy inequality and \eqref{poincare1} to the righthand of above inequality, and then summing the resulted equations over $j\in\bN$, we can obtain that
\be
\bali
&\angt^{\f{9-\dl}{2}}\|\p^2_zg(t)\|^2_{X_{\tau,\kappa_2}}+\dl\int^{T_0}_0\angt^{\f{9-\dl}{2}}\|\p^3_z g(t)\|^2_{X_{\tau,\kappa_2}}dt+\la\s{\e}(j+1) \int^{T_0}_0\angt^{\f{9-\dl}{2}}{\eta}(t)\|\p^2_zg(t)\|^2_{X_{\tau,\kappa_2}}dt\\
\leq &\int^{T_0}_0\angt^{\f{7-\dl}{2}}\|\p^2_z g(t)\|^2_{X_{\tau,\kappa_2}}dt+C_\dl\int^{T_0}_0 \angt^{\f{11-\dl}{2}}\sum^6_{i=1}\|\p^2_zR^i\|^2_{X_{\tau,\kappa_2}}dt.
\eali
\ee
Similar as Lemma \ref{lem7.4}, we have the following lemma.
\begin{lemma}
We have the following estimates
\begin{align}
 &\sum^6_{i=1}\|\p^2_zR^i\|^2_{X_{\tau,\kappa_2}}\ls\angt^{3/2}\lt(\|g\|^2_{X_{\tau,7}}\|\p^3_zg\|^2_{X_{\tau,\kappa_2}}+\|\p_zg\|^2_{X_{\tau,7}}\|\p^2_zg\|^2_{X_{\tau,\kappa_2+2}}
+\|\p^2_zg\|^2_{X_{\tau,7}}\|\p_zg\|^2_{X_{\tau,\kappa_2+2}}\rt). \label{zzgest1}
\end{align}
\end{lemma}
\pf Proof of this Lemma is also repeatedly use of \eqref{poincare} and \eqref{poincare1} in Lemma \ref{lpoincare}, product estimates in \eqref{product1} to \eqref{product3} in Lemma \ref{lemproduct} and the relation between $\bl{u}$ and $\bl{g}$. Since it is a routing estimate, we omit the details. \qed

Then by using the a priori estimates in \eqref{solutiongevrey} and \eqref{zzgest1}, we see that
\begin{align}
&\int^{T_0}_0 \angt^{\f{11-\dl}{2}}\sum^6_{i=1}\|\p^2_zR^i\|^2_{X_{\tau,\kappa_1}}dt\ls C^2_\ast\e^2\int^{T_0}_0 \lt(\angt^{\f{9-\dl}{2}}\|\p^3_zg\|^2_{X_{\tau,\kappa_2}}+\angt^{\f{5-\dl}{2}}\|\p_zg\|^2_{X_{\tau,\kappa_0}}+\angt^{\f{7-\dl}{2}}\|\p^2_zg\|^2_{X_{\tau,\kappa_1}}\rt)dt.\nn
\end{align}
Inserting the above inequality into \eqref{zzgoodfirst}, we can obtain \eqref{goodzzg} in Lemma \ref{lzzg}. \qed

\section{Estimates of the good unknowns $\bl{\fH}$} \label{secfh}

After we obtain the faster decay rate for low order Gevrey-2 energy of the unknowns $\bl{u}$ through the linearly good unknowns $\bl{g}$. In this section, we focus on the Gevrey-2 estimates of the linearly good unknowns $\bl{\fH}$ and its $z-$ derivative. It will induce faster decay rate for low order Gevrey-2 energy of the auxiliary functions $\bl{\mH}$ and $\bl{\mG}$ as displayed in \eqref{solutiongevrey}.

Below we set
\bes
\kappa_3=9,\q \kappa_4=7.
\ees

\subsection{Estimates of $\bl{\fH}$}
\begin{lemma}\label{lemfh}
Under the assumption of \eqref{gassump}, for sufficiently small $\e$, there exits a constant $C_\dl$ such that for any $t\in (0,T_0]$, we have the following estimate.
\begin{align}
&\angt^{\f{3-\dl}{2}} \lt\| \bl{\fH} (t)\rt\|^2_{X_{\tau,\kappa_3,7/8}}+\int^{T_0}_0\angt^{\f{3-\dl}{2}}\lt\|\p_z\bl{\fH}(t)\rt\|^2_{X_{\tau,\kappa_3,7/8}}dt+\la\s{\e} \int^{T_0}_0\angt^{\f{3-\dl}{2}}{\eta}(t) \lt\|\bl{\fH}(t)\rt\|^2_{X_{\tau,\kappa_3+1/2,7/8}}dt\nn\\
\leq  &C_\dl \f{C^2_\ast}{\la} \e^{3/2}\int^{T_0}_0 \lt(\angt^{\f{1+\dl}{2}}\lt\|\bl{\fH}(t)\rt\|^2_{X_{\tau,\kappa_3+1/2,7/8}}+\angt^{\f{3-\dl}{2}}\lt\|\p_z\bl{\fH}(t)\rt\|^2_{X_{\tau,\kappa_3,7/8}}
+\angt^{-\f{1-\dl}{2}}\lt\|\bl{\mH}(t)\rt\|^2_{X_{\tau,\kappa}}\rt)dt\label{goodfh}\\
        &+\f{\s{\e}}{\la}\int^{T_0}_0 \angt^{\f{3+\dl}{2}}\lt\|\bl{g}(t)\rt\|^2_{X_{\tau,\kappa_0+1/2}}dt.\nn
\end{align}
\end{lemma}

We only show estimates for $\fH$, since that of $\wt{\fH}$ follows the same line.
First we derive of the equation satisfied by $\fH$. By multiplying $\f{z}{2\angt}$ to the first equation of \eqref{auxih}, we see that
\begin{align}
&\lt[\partial_{t}-\partial_{z}^{2} \rt]\f{z}{2\angt}\int^{+\i}_z \mathcal{H} d\bar{z}-\f{1}{\angt}\fH=-\f{z}{2\angt}\left({u} \p_x+v \partial_y+w\p_z \right)\int^{+\i}_z \mathcal{H} d\bar{z}+\f{z\s{\e}}{2\angt} \angt^{\dl-1} \p_x w. \label{fh1}
\end{align}
Then subtracting \eqref{fh1} from \eqref{auxih1}, we have
\begin{align}
\lt[\partial_{t}-\partial_{z}^{2}+\f{1}{\angt}\rt]\fH=&\s{\e}\angt^{\dl-1}(\p^2_xu+\p^2_{xy}v)-\f{z\s{\e}}{2\angt} \angt^{\dl-1} \p_x w\nn\\
 &-\f{z}{2\angt}w \mathcal{H}+(\p_xu+\p_yv)\mH\nn\\
 &+(\p_zu\p_x+\p_zv\p_y)\int^{+\i}_z \mathcal{H} d\bar{z}+\f{z}{2\angt}\left({u} \p_x+v \partial_y\right)\int^{+\i}_z \mathcal{H} d\bar{z}\nn\\
:=&Q^1+Q^2+Q^3.\label{fhest1}
\end{align}
with
\bes
\p_z\fH\big|_{z=0}=0, \quad \lim_{z \rightarrow+\infty} \fH=0.
\ees

Now applying $M_{j,\kappa_3}\p^j_x$  to \eqref{fhest1} and denoting $f_{j,x,\kappa_3}$ by $f_{j,\kappa_3}$ for a function $f$, we can obtain that
\bes
\bali
&\partial_{t} \fH_{j,\kappa_3}+\la\s{\e}\eta(t)(j+1)\fH_{j,\kappa_3}-\partial_{z}^{2} \fH_{j,\kappa_3}+\f{1}{\angt}\fH_{j,\kappa_3}=\sum^3_{i=1} Q^i_{j,\kappa_3}.
\eali
\ees
Similar as \eqref{hj5}, for $0<\nu<1$, we can have
\begin{align}
&\lt\langle \lt[\partial_{t} +\la\s{\e}{\eta}(t)(j+1) -\partial_{z}^{2} \rt]\fH_{j,\kappa_3}, \fH_{j,\kappa_3}(t)\rt\rangle_{\th_{2\nu}}\nn\\
=&\f{1}{2}\f{d}{dt}\|{\fH_{j,\kappa_3}}(t)\|^2_{L^2(\th_{2\nu})}+\|\p_z \fH_{j,\kappa_3}(t)\|^2_{L^2(\th_{2\nu})}+\f{4-\nu}{4\angt}\| \fH_{j,\kappa_3}(t)\|^2_{L^2(\th_{2\nu})}+\f{\nu-\nu^2}{8}\lt\|\f{z}{\angt} \fH_{j,\kappa_3}(t)\rt\|^2_{L^2(\th_{2\nu})}\nn\\
 &+(j+1)\la\s{\e}{\eta}(t)\|{\fH_{j,\kappa_3}}(t)\|^2_{L^2(\th_{2\nu})}.\nn
\end{align}
Using \eqref{poincare} in Lemma \ref{lpoincare}, we have
\begin{align}
&2\lt\langle \lt[\partial_{t} +\la\s{\e}{\eta}(t)(j+1)-\partial_{z}^{2} \rt]\fH_{j,\kappa_3}, \fH_{j,\kappa_3}(t)\rt\rangle_{\th_{2\nu}}\nn\\
\geq &\f{d}{dt}\|\fH_{j,\kappa_3}(t)\|^2_{L^2(\th_{2\nu})}+\|\p_z \fH_{j,\kappa_3}(t)\|^2_{L^2(\th_{2\nu})}+\f{2}{\angt}\| \fH_{j,\kappa_3}(t)\|^2_{L^2(\th_{2\nu})}+2(j+1)\la\s{\e}\eta(t)\|\fH_{j,\kappa_3}(t)\|^2_{L^2(\th_{2\nu})}.\nn
\end{align}
Performing energy estimates as before and using Cauchy inequality, we have
\bes
\bali
&\angt^{\f{3-\dl}{2}}\|\fH_{j,\kappa_3}(t)\|^2_{L^2(\th_{2\nu})}+\int^{T_0}_0\angt^{\f{3-\dl}{2}}\|\p_z \fH_{j,\kappa_3}(t)\|^2_{L^2(\th_{2\nu})}dt\\
&+\la\s{\e}(j+1) \int^{T_0}_0\angt^{\f{3-\dl}{2}}{\eta}(t)\|\fH_{j,\kappa_3}(t)\|^2_{L^2(\th_{2\nu})}dt\ls\f{1}{\la\s{\e}}\int^{T_0}_0\angt^{\f{5-3\dl}{2}}(j+1)^{-1}\sum^3_{i=1}\|Q^i_{j,\kappa_3}\|^2_{L^2(\th_{2\nu})}dt.
\eali
\ees
Letting $\nu=7/8$ and summing the above inequality over $j\in\bN$, we can achieve that
\begin{align}
&\angt^{\f{3-\dl}{2}} \lt\|{\fH} (t)\rt\|^2_{X_{\tau,\kappa_3,7/8}}+\int^{T_0}_0\angt^{\f{3-\dl}{2}}\lt\|\p_z{\fH}(t)\rt\|^2_{X_{\tau,\kappa_3,7/8}}dt+\la\s{\e} \int^{T_0}_0\angt^{\f{3-\dl}{2}}{\eta}(t) \lt\|{\fH}(t)\rt\|^2_{X_{\tau,\kappa_3+1/2,7/8}}dt\nn\\
\leq  &\f{1}{\la\s{\e}}\int^{T_0}_0\angt^{\f{5-3\dl}{2}}\sum^3_{i=1}\|Q^i\|^2_{X_{\tau,\kappa_3-1/2,7/8}}dt.\label{fhest2}
\end{align}
{\noindent\bf Estimates of $Q^1$}

This is direct. By using incompressibility and \eqref{poincare1}, we have
\begin{align}
&\|Q^1\|^2_{X_{\tau,\kappa_3-1/2,7/8}}\ls \e \angt^{2\dl-2}\|\p^2_h\bl{u}\|^2_{X_{\tau,\kappa_3-1/2,7/8}}\ls  \e \angt^{2\dl-2}\|\bl{u}\|^2_{X_{\tau,\kappa_3+7/2,7/8}}\ls \e \angt^{2\dl-1}\|g\|^2_{X_{\tau,\kappa_0+1/2}}.\label{termq1s}
\end{align}
From this, we can obtain
\begin{align}
&\f{1}{\la\s{\e}}\int^{T_0}_0\angt^{\f{5-3\dl}{2}}\|Q^1\|^2_{X_{\tau,\kappa_3-1/2,7/8}}dt\ls \f{\s{\e}}{\la}\int^{T_0}_0\angt^{\f{3+\dl}{2}}\|\bl{g}\|^2_{X_{\tau,\kappa_0+1/2}}dt. \label{termq1f}
\end{align}
{\noindent\bf Estimates of $Q^2$}

 By using \eqref{poincare1} in Lemma \ref{lpoincare}, \eqref{product1} to \eqref{product3} in Lemma \ref{lemproduct}, incompressibility, and the a priori estimates in Lemma \ref{llowpoint}, we have
\begin{align}
\|Q^2\|^2_{X_{\tau,\kappa_3-1/2,\nu}}\ls &\lt\|\p_z(w\mH)\rt\|^2_{X_{\tau,\kappa_3-1/2,\nu}}+\lt\|\p_h\bl{u}\mH \rt\|^2_{X_{\tau,\kappa_3-1/2,\nu}}\nn\\
\ls & \angt^{1/2}\lt(\|\p_h\bl{u}\|^2_{X_{\tau,5,\f{\nu+1}{4}}}\|\p_z\mH\|^2_{X_{\tau,\kappa_3-1/2,\f{\nu+1}{4}}}
+\|\p_h\bl{u}\|^2_{X_{\tau,\kappa_3-1/2,\f{\nu+1}{4}}}\|\p_z\mH\|^2_{X_{\tau,5,\f{\nu+1}{4}}}\rt)\label{termq2second}\\
\ls & C^2_\ast\e^2\lt(\angt^{-\f{2-\dl}{2}}\|\p_z\fH\|^2_{X_{\tau,\kappa_3,\nu}}
+\angt^{-\f{2-\dl}{2}}\|g\|^2_{X_{\tau,\kappa_0+1/2}}\rt).\nn
\end{align}
Here at the last line of the above inequality, we have used the fact
\bes
\|\p_z\mH\|_{X_{\tau,\kappa_3-1/2,\f{\nu+1}{4}}}\ls\|\p_z\fH\|^2_{X_{\tau,\kappa_3-1/2,\nu}},\q \|\p_h\bl{u}\|^2_{X_{\tau,\kappa_3-1/2,\f{\nu+1}{4}}}\ls \|\bl{g}\|^2_{X_{\tau,\kappa_3+3/2}}.
\ees
From this, we can obtain
\begin{align}
&\f{1}{\la\s{\e}}\int^{T_0}_0\angt^{\f{5-3\dl}{2}}\|Q^2\|^2_{X_{\tau,\kappa_3-1/2,7/8}}dt\nn\\
\ls&C_\dl \f{C^2_\ast}{\la} \e^{3/2}\int^{T_0}_0 \lt(\angt^{\f{3-\dl}{2}}\lt\|\p_z\bl{\fH}(t)\rt\|^2_{X_{\tau,\kappa_3,7/8}}
+\angt^{\f{3+\dl}{2}}\lt\|\bl{g}(t)\rt\|^2_{X_{\tau,\kappa_0+1/2}}\rt)dt.\label{termq2f}
\end{align}

{\noindent\bf Estimates of $Q^3$}

Using \eqref{poincare1} in Lemma \ref{lpoincare} and product estimates \eqref{product1} to \eqref{product3} in Lemma \ref{lemproduct}, we have
\begin{align}
\|Q^3\|^2_{X_{\tau,\kappa_3-1/2,\nu}}\ls &\lt\|\p_z(u,v)\int^\i_z \p_h\mH d\bar{z}\rt\|^2_{X_{\tau,\kappa_3-1/2,\nu}}+\lt\|\p_z\lt((u,v)\int^\i_z \p_h\mH d\bar{z}\rt)\rt\|^2_{X_{\tau,\kappa_3-1/2,\nu}}\nn\\
\ls & \angt^{1/2}\lt(\|\p_zu\|^2_{X_{\tau,5,\f{\nu+1}{4}}}\|\p_h\mH\|^2_{X_{\tau,\kappa_3-1/2,\f{\nu+1}{4}}}
+\|\p_zu\|^2_{X_{\tau,\kappa_3-1/2,\f{\nu+1}{4}}}\|\p_h\mH\|^2_{X_{\tau,5,\f{\nu+1}{4}}}\rt).\label{termq3third}
\end{align}
Then using the a priori estimates in \eqref{solutiongevrey} in Lemma \ref{llowpoint} and \eqref{gnineq1} in Claim \ref{claim}, we can obtain that
\begin{align}
\|Q^3\|^2_{X_{\tau,\kappa_3-1/2,\nu}}\ls&  \angt^{1/2}\lt(\|g\|^2_{X_{\tau,5}}\|\mH\|^{\f{2(m-1)}{m}}_{X_{\tau,\kappa_3-1/2,\f{\nu+1}{4}}}
\|\mH\|^{\f{2}{m}}_{X_{\tau,\kappa_3-1/2+2m,\f{\nu+1}{4}}}
+\|g\|^2_{X_{\tau,\kappa_0}}\|\fH\|^2_{X_{\tau,7,\nu}}\rt)\label{termq3f}\\
\ls & C^2_\ast\e^2\lt(\angt^{-\f{4-\dl}{2}}\|\fH\|^{\f{2(m-1)}{m}}_{X_{\tau,\kappa_3,\nu}}
\|\mH\|^{\f{2}{m}}_{X_{\tau,\kappa}}
+\angt^{-\f{2-\dl}{2}}\|g\|^2_{X_{\tau,\kappa_0}}\rt).\nn
\end{align}
Here at the last line of the above inequality, we have used the fact
\bes
\|\mH\|_{X_{\tau,\kappa_3-1/2,\f{\nu+1}{4}}}\ls\|\fH\|^2_{X_{\tau,\kappa_3-1/2,\nu}},\q \kappa_3-1/2+2m\leq \kappa,
\ees
by our choice of $m=\lt[\dl^{-1}\rt]$.

Then by using Young inequality and \eqref{termq3f}, we can obtain that
\begin{align}
&\f{1}{\la\s{\e}}\int^{T_0}_0\angt^{\f{5-3\dl}{2}}\|Q^3\|^2_{X_{\tau,\kappa_3-1/2,7/8}}dt\nn\\
\ls &\f{C^2_\ast\e^{3/2}}{\la}\int^{T_0}_0\lt(\angt^{\lt(\f{1-2\dl}{2}+\f{1-\dl}{2m}\rt)\f{m}{m-1}}\|\fH\|^2_{X_{\tau,\kappa_3,7/8}}+
\angt^{-\f{1-\dl}{2}}\|\mH\|^2_{X_{\tau,\kappa}}+\angt^{\f{3+\dl}{2}}\|g\|^2_{X_{\tau,\kappa_0}}\rt)dt\label{termq3s}\\
\ls &\f{C^2_\ast\e^{3/2}}{\la}\int^{T_0}_0\lt(\angt^{\f{1+\dl}{2}}\|\fH\|^2_{X_{\tau,\kappa_3,7/8}}+
\angt^{-\f{1-\dl}{2}}\|\mH\|^2_{X_{\tau,\kappa}}+\angt^{\f{3+\dl}{2}}\|g\|^2_{X_{\tau,\kappa_0}}\rt)dt.\nn
\end{align}
Here at the last line of \eqref{termq3s}, by our choice of $m$, we have
\bes
\lt(\f{1-2\dl}{2}+\f{1-\dl}{2m}\rt)\f{m}{m-1}\leq \f{1+\dl}{2}.
\ees
Inserting \eqref{termq1f}, \eqref{termq2f} and \eqref{termq3s} into \eqref{fhest2}, we can obtain \eqref{goodfh} in Lemma \ref{lemfh}. \qed

\subsection{Estimates of $\p_z\bl{\fH}$}

\begin{lemma}\label{lemzfh}
Under the assumption of \eqref{gassump}, for sufficiently small $\e$, there exits a constant $C_\dl$ such that for any $t\in (0,T_0]$, we have the following estimate.
\begin{align}
&\angt^{\f{5-\dl}{2}} \lt\| \p_z\bl{\fH}(t)\rt\|^2_{X_{\tau,\kappa_4,7/8}}+\int^{T_0}_0\angt^{\f{5-\dl}{2}}\lt\|\p^2_z\bl{\fH}(t)\rt\|^2_{X_{\tau,\kappa_4,7/8}}dt+\la\s{\e} \int^{T_0}_0\angt^{\f{5-\dl}{2}}{\eta}(t) \lt\|\p_z\bl{\fH}(t)\rt\|^2_{X_{\tau,\kappa_4+1/2,7/8}}dt\nn\\
\leq  &C_\dl \f{C^2_\ast}{\la} \e^{3/2}\int^{T_0}_0 \lt(\angt^{\f{1+\dl}{2}}\lt\|\bl{\fH}(t)\rt\|^2_{X_{\tau,\kappa_3+1/2,7/8}}+\angt^{\f{3-\dl}{2}}\lt\|\p_z\bl{\fH}(t)\rt\|^2_{X_{\tau,\kappa_3,7/8}}\rt)dt\label{goodzfh}\\
        &+\f{\s{\e}}{\la}\int^{T_0}_0 \angt^{\f{3+\dl}{2}}\lt\|\bl{g}(t)\rt\|^2_{X_{\tau,\kappa_0+1/2}}dt+\int^{T_0}_0\angt^{\f{3-\dl}{2}}\lt\|\p_z\bl{\fH}(t)\rt\|^2_{X_{\tau,\kappa_4,7/8}}dt.\nn
\end{align}
\end{lemma}

\pf Now applying $M_{j,\kappa_4}\p_z\p^j_x$  to \eqref{fhest1} and denoting $f_{j,x,\kappa_4}$ by $f_{j,\kappa_4}$ for a function $f$, we can obtain that
\be\label{zfhfirst}
\bali
&\partial_{t} \p_z\fH_{j,\kappa_4}+\la\s{\e}\eta(t)(j+1)\p_z\fH_{j,\kappa_4}-\partial_{z}^{2} \p_z\fH_{j,\kappa_4}+\f{1}{\angt}\p_z\fH_{j,\kappa_4}=\sum^3_{i=1} \p_zQ^i_{j,\kappa_4}.
\eali
\ee
Similar as \eqref{hj5}, we can have
{\small
\begin{align}
&\lt\langle \lt[\partial_{t} +\la\dl\s{\e}{\eta}(t)(j+1) -\partial_{z}^{2} \rt]\p_z\fH_{j,\kappa_4}, \p_z\fH_{j,\kappa_4}(t)\rt\rangle_{\th_{2\nu}}\nn\\
=&\f{1}{2}\f{d}{dt}\|\p_z{\fH_{j,\kappa_4}}(t)\|^2_{L^2(\th_{2\nu})}+\|\p^2_z \fH_{j,\kappa_4}(t)\|^2_{L^2(\th_{2\nu})}+\f{4-\nu}{4\angt}\|\p_z \fH_{j,\kappa_4}(t)\|^2_{L^2(\th_{2\nu})}+\f{\nu-\nu^2}{8}\lt\|\f{z}{\angt} \p_z\fH_{j,\kappa_4}(t)\rt\|^2_{L^2(\th_{2\nu})}\nn\\
 &+(j+1)\la\s{\e}{\eta}(t)\|\p_z{\fH_{j,\kappa_4}}(t)\|^2_{L^2(\th_{2\nu})}.\nn
\end{align}
}
Using the inequality \eqref{poincare} in Lemma \ref{lpoincare}, we have
{\small
\begin{align}
&2\lt\langle \lt[\partial_{t} +\la\dl\s{\e}{\eta}(t)(j+1)-\partial_{z}^{2} \rt]\p_z\fH_{j,\kappa_4}, \p_z\fH_{j,\kappa_4}(t)\rt\rangle_{\th_{2\nu}}\nn\\
\geq &\f{d}{dt}\|\p_z\fH_{j,\kappa_4}(t)\|^2_{L^2(\th_{2\nu})}+\|\p^2_z \fH_{j,\kappa_4}(t)\|^2_{L^2(\th_{2\nu})}+\f{2}{\angt}\| \p_z\fH_{j,\kappa_4}(t)\|^2_{L^2(\th_{2\nu})}+2(j+1)\la\s{\e}\eta(t)\|\p_z\fH_{j,\kappa_4}(t)\|^2_{L^2(\th_{2\nu})}.\nn
\end{align}
}
Performing space and time energy estimates to \eqref{zfhfirst} as before and using integration by parts for $\p_z Q^i$ with $z-$ variable an then Cauchy inequality, we have
{\small
\bes
\bali
&\angt^{\f{5-\dl}{2}}\|\p_z\fH_{j,\kappa_4}(t)\|^2_{L^2(\th_{2\nu})}+\int^{T_0}_0\angt^{\f{5-\dl}{2}}\|\p^2_z \fH_{j,\kappa_4}(t)\|^2_{L^2(\th_{2\nu})}dt\\
&+\la\s{\e}(j+1) \int^{T_0}_0\angt^{\f{5-\dl}{2}}{\eta}(t)\|\p_z\fH_{j,\kappa_4}(t)\|^2_{L^2(\th_{2\nu})}dt\\
\ls&\int^{T_0}_0\angt^{\f{5-\dl}{2}}\sum^3_{i=1}\|Q^i_{j,\kappa_4}\|^2_{L^2(\th_{2\nu})}dt+\int^{T_0}_0\angt^{\f{3-\dl}{2}}\|\p_z \fH_{j,\kappa_4}(t)\|^2_{L^2(\th_{2\nu})}dt.
\eali
\ees
}
Letting $\nu=7/8$ and summing the above inequality over $j\in\bN$, we can achieve that
\begin{align}
&\angt^{\f{5-\dl}{2}}\|\p_z\fH(t)\|^2_{X_{\tau,\kappa_4,7/8}}+\int^{T_0}_0\angt^{\f{5-\dl}{2}}\|\p^2_z \fH(t)\|^2_{X_{\tau,\kappa_4,7/8}}dt+\la\s{\e}\int^{T_0}_0\angt^{\f{5-\dl}{2}}{\eta}(t)\|\p^2_z\fH(t)\|^2_{X_{\tau,\kappa_4,7/8}}dt\nn\\
\ls&\int^{T_0}_0\angt^{\f{5-\dl}{2}}\sum^3_{i=1}\|Q^i\|^2_{X_{\tau,\kappa_4,7/8}}dt+\int^{T_0}_0\angt^{\f{3-\dl}{2}}\|\p_z \fH(t)\|^2_{X_{\tau,\kappa_4,7/8}}dt.\label{zfhs}
\end{align}

{\noindent\bf Estimates of $Q^1$} This is direct. Same as \eqref{termq1s}, we have
\begin{align}
&\|Q^1\|^2_{X_{\tau,\kappa_4,7/8}}\ls \e \angt^{2\dl-2}\|\p^2_h\bl{u}\|^2_{X_{\tau,\kappa_4,7/8}}\ls  \e \angt^{2\dl-2}\|\bl{u}\|^2_{X_{\tau,\kappa_4+4,7/8}}\ls \e \angt^{2\dl-1}\|\bl{g}\|^2_{X_{\tau,\kappa_0+1/2}}.\nn
\end{align}
From this, we can obtain
\begin{align}
\int^{T_0}_0\angt^{\f{5-\dl}{2}}\|Q^1\|^2_{X_{\tau,\kappa_4,7/8}}dt\ls& \e\int^{T_0}_0\angt^{\f{3+\dl}{2}}\angt^\dl\|\bl{g}\|^2_{X_{\tau,\kappa_0+1/2}}dt\ls\f{\s{\e}}{\la}\int^{T_0}_0\angt^{\f{3+\dl}{2}}\|\bl{g}\|^2_{X_{\tau,\kappa_0+1/2}}dt. \label{termzq1s}
\end{align}
Here at the last line, we have used \eqref{timelifespan0}.

{\noindent\bf Estimates of $Q^2$}
 Similar as estimate in \eqref{termq2second}, we can obtain that
\begin{align}
&\|Q^2\|^2_{X_{\tau,\kappa_3,7/8}}\ls  C^2_\ast\e^2\lt(\angt^{-\f{2-\dl}{2}}\|\p_z\fH\|^2_{X_{\tau,\kappa_4,7/8}}
+\angt^{-\f{2-\dl}{2}}\|g\|^2_{X_{\tau,\kappa_0+1/2}}\rt).\nn
\end{align}
From this, we can obtain
\begin{align}
&\int^{T_0}_0\angt^{\f{5-\dl}{2}}\|Q^2\|^2_{X_{\tau,\kappa_4,7/8}}dt\nn\\
\ls&C^2_\ast\e^2\int^{T_0}_0 \lt(\angt^{\f{3}{2}}\lt\|\p_z\bl{\fH}(t)\rt\|^2_{X_{\tau,\kappa_4,7/8}}
+\angt^{\f{3}{2}}\lt\|\bl{g}(t)\rt\|^2_{X_{\tau,\kappa_0+1/2}}\rt)dt\label{termzq2f}\\
\ls&\f{C^2_\ast}{\la}\e^{3/2}\int^{T_0}_0 \lt(\angt^{\f{3-\dl}{2}}\lt\|\p_z\bl{\fH}(t)\rt\|^2_{X_{\tau,\kappa_4,7/8}}
+\angt^{\f{3+\dl}{2}}\lt\|\bl{g}(t)\rt\|^2_{X_{\tau,\kappa_0+1/2}}\rt)dt.\nn
\end{align}
{\noindent\bf Estimates of $Q^3$} Similar as estimate in \eqref{termq3third} and using the a priori estimate in \eqref{solutiongevrey}, we can obtain that
\begin{align}
\|Q^3\|^2_{X_{\tau,\kappa_4,7/8}}\ls & \angt^{1/2}\lt(\|\p_zu\|^2_{X_{\tau,5,\f{15}{32}}}\|\p_h\mH\|^2_{X_{\tau,\kappa_4,\f{15}{32}}}
+\|\p_zu\|^2_{X_{\tau,\kappa_4,\f{15}{32}}}\|\p_h\mH\|^2_{X_{\tau,5,\f{15}{32}}}\rt)\nn\\
\ls &\angt^{1/2}\lt(\|\p_zu\|^2_{X_{\tau,5,\f{15}{32}}}\|\mH\|^2_{X_{\tau,\kappa_3,\f{15}{32}}}
+\|\p_zu\|^2_{X_{\tau,\kappa_4,\f{15}{32}}}\|\mH\|^2_{X_{\tau,7,\f{15}{32}}}\rt)\nn\\
\ls &C^2_\ast\e^2\lt(\angt^{-\f{4-\dl}{2}}\|\fH\|^2_{X_{\tau,\kappa_3}}+\angt^{-\f{2-\dl}{2}}\|\bl{g}\|^2_{X_{\tau,\kappa_0+1/2}}\rt).\nn
\end{align}

From this, we can obtain
\begin{align}
&\int^{T_0}_0\angt^{\f{5-\dl}{2}}\|Q^3\|^2_{X_{\tau,\kappa_4,7/8}}dt\nn\\
\ls&C^2_\ast\e^2\int^{T_0}_0 \lt(\angt^{\f{1}{2}}\lt\|\bl{\fH}(t)\rt\|^2_{X_{\tau,\kappa_3,7/8}}
+\angt^{\f{3}{2}}\lt\|\bl{g}(t)\rt\|^2_{X_{\tau,\kappa_0+1/2}}\rt)dt\label{termzq3f}\\
\ls&\f{C^2_\ast}{\la}\e^{3/2}\int^{T_0}_0 \lt(\angt^{\f{1+\dl}{2}}\lt\|\bl{\fH}(t)\rt\|^2_{X_{\tau,\kappa_3,7/8}}
+\angt^{\f{3+\dl}{2}}\lt\|\bl{g}(t)\rt\|^2_{X_{\tau,\kappa_0+1/2}}\rt)dt.\nn
\end{align}
Inserting \eqref{termzq1s}, \eqref{termzq2f} and \eqref{termzq3f} into \eqref{zfhs}, we can obtain \eqref{goodzfh} in Lemma \ref{lemzfh}. \qed

\begin{appendix}
\begin{center} {\Large\sc Appendix}  \end{center}
\titleformat{\section}[block]{\large\bf}{\Alph{section}}{0.5em}{}[]

\section{ Estimates in Lemma \ref{llowpoint}}

\pf  We only show estimates for $u$, $\mH$ and $\mG$ since the others are complete the same.

By the definition of $g$  and boundary condition for $u$, we have
\be\label{phig}
\lt\{
\bali
&\p_zu+\f{ z}{2\langle t\rangle} u=g,\\
&u \big|_{z=0}=0.
\eali
\rt.
\ee
Solving this ODE, we get
\bes
u(t,x_h,z)=\exp\lt(-\f{z^2}{4\angt}\rt)\int^z_0 g(t,x_h,\bar{z})\exp\lt(\f{\bar{z}^2}{4\angt}\rt) d\bar{z}.
\ees
For any $0\leq\nu<1$, by multiplying the above equality with $\th_{\nu}M_{k,\kappa}$ and taking $\p^k_h$, we have
\be\label{ethetaphi}
\bali
\th_{\la} M_{k,\kappa}\p^k_h u= \th_{\la-1}(z) \int^z_0 \th(\bar{z})M_{k,\kappa}\p^k_h g(\bar{z})\exp\lt(\f{1}{8\angt}(\bar{z}^2-z^2)\rt) d\bar{z}.
\eali
\ee

Using the fact that for any $\beta\geq 0$,
\bes
\sup_{z\geqq 0} \zeta^\beta e^{-\zeta^2}\leq C_\beta,
\ees
we have
\bes
\lt| \lt(\f{z}{\s{\angt}}\rt)^\beta \th_{\nu-1}\rt| \leq C_{\nu,\beta}.
\ees
Moreover, by considering $0\leq\zeta\leq 1$ and $\zeta>1$, it is not hard to check that
\bes
e^{-\zeta^2}\int^\zeta_0 e^{\bar{\zeta}^2}d\bar{\zeta}\leq \f{2}{1+\zeta}.
\ees
Then a change of variable indicates that
\bes
 \int^z_0 \exp\lt(\f{1}{4\angt}(\bar{z}^2-z^2)\rt) d\bar{z}\leq \f{C}{1+\zeta} \s{\angt}.
\ees
In \eqref{ethetaphi}, by using H\"{o}lder inequality on $z$, we have

\be\label{ethetaphi1}
\bali
\lt|\th_\nu M_{k,\kappa}\p^k_h u\rt|\leq& \th_{\nu-1}\|\th M_{k,\kappa}\p^k_hg\|_{L^{2}_z} \lt( \int^z_0 \exp\lt(\f{1}{4\angt}(\bar{z}^2-z^2)\rt) d\bar{z} \rt)^{1/2}\\
           \ls & \th_{\nu-1}\|\th M_{k,\kappa}\p^k_h g\|_{L^{2}_z}\angt^{1/4} (1+\zeta)^{-1/2}\\
           \ls & \th_{\nu-1}\|\th M_{k,\kappa}\p^k_h  g\|_{L^{2}_z}\angt^{1/4}.
\eali
\ee
Also, from the definition of $\fH$ in \eqref{defng} and the fact in \eqref{antimh}, we see that
\be\label{fhmh}
\lt\{
\bali
&\p_z\lt(\int^\i_z\mH\bar{z}\rt)+\f{ z}{2\langle t\rangle} \lt(\int^\i_z\mH\bar{z}\rt)=-\fH,\\
& \int^\i_z\mH\bar{z}\Big|_{z=0}=0.
\eali
\rt.
\ee
Similar as \eqref{ethetaphi1}, we can have for $0<\nu<7/8$
\be\label{fhmh1}
\lt|\th_\nu M_{k,\kappa}\p^k_h \int^\i_z\mH d\bar{z}\rt|\ls \th_{\nu-7/8}\|\th_{7/8} M_{k,\kappa}\p^k_h  \fH\|_{L^{2}_z}\angt^{1/4}.
\ee
From \eqref{ethetaphi1}, by using the a priori assumption \eqref{gassump}, it is easy to see that
\be\label{eug1}
\bali
\angt^{-1/2}\lt\|\th_\nu M_{k,\kappa}\p^k_h u\rt\|_{L^2}  \ls& \angt^{-1/2} \|\th_{\nu-1}\|_{L^2_z}\|\th M_{k,\kappa}\p^k_h g\|_{L^2}\angt^{1/4}\\
                       \ls& \|\th M_{k,\kappa}\p^k_h g\|_{L^2}\ls C_\ast\e \angt^{-\f{5-\dl}{4}}.
\eali
\ee
By squaring \eqref{eug1} and summing the resulted equation over $k\in \bN$, we obtain that
\be\label{egu1p}
\angt^{-1}\|u\|^2_{X_{\tau,12,\nu}}\ls \|g\|^2_{X_{\tau,12}}\ls C_\ast\e \angt^{-\f{5-\dl}{2}}.
\ee
By using \eqref{fhmh1} and same as \eqref{egu1p}, we have, for $0<\nu<7/8$
\be\label{fhmh2}
\angt^{-1}\lt\|\int^\i_z\mH d\bar{z}\rt\|^2_{X_{\tau,9,\nu}}\ls \|\fH\|^2_{X_{\tau,9,7/8}}\ls C_\ast\e \angt^{-\f{3-\dl}{2}}.
\ee
Applying $\th_\nu \p^k_h$, $\th_\nu\p_z\p^k_h$ and $\th_\nu\p^2_z\p^k_h$ to the first equation of \eqref{phig} respectively give that
\bes
\bali
&\th_{\nu}\p_z \p^k_hu=\th_{\nu}\p^k_hg-\th_{\nu}\f{z}{2\langle t\rangle} \p^k_h u,\\
&\th_{\nu}\p^2_z \p^k_h u =\th_{\nu}\p_z \p^k_hg-\th_{\nu}\f{1}{2\angt}\p^k_h u-\th_{\nu}\f{z}{2\angt}\p_z\p^k_h u,\\
&\th_{\nu}\p^3_z \p^k_h u =\th_{\nu}\p^2_z \p^k_h g-\th_{\nu}\f{1}{\angt}\p^k_h \p_z u-\th_{\nu}\f{z}{2\angt}\p^2_z\p^k_h u.
\eali
\ees
From the above equalities, we can easily deduce that
\begin{align}
&\|\p_z u\|_{X_{\tau,12,\nu}}\ls \|g\|_{X_{\tau,12,\nu}}+\angt^{-1/2}\| u\|_{X_{\tau,12,\f{\nu+1}{2}}}\ls_\nu C_\ast\e \angt^{-\f{5-\dl}{4}},\label{ug2}\\
&\s{\dl}\angt^{1/2}\|\p^2_z u\|_{X_{\tau,10,\nu}}\ls \s{\dl}\angt^{1/2} \|\p_zg\|_{X_{\tau,10,\nu}}+\s{\dl}\angt^{-1/2}\| u\|_{X_{\tau,10,\nu}}+\s{\dl}\| \p_z u\|_{X_{\tau,10,\f{\nu+1}{2}}}\nn\\
&\ls_\nu C_\ast\e \angt^{-\f{5-\dl}{4}}, \label{ug3}\\
&\dl\angt\|\p^3_z u\|_{X_{\tau,8,\nu}}\ls \dl\angt\|\p^2_zg\|_{X_{\tau,8,\nu}}+\dl\| \p_zu\|_{X_{\tau,8,\nu}}+\dl\angt^{1/2}\| \p^2_z u\|_{X_{\tau,8,\f{\nu+1}{2}}}\nn\\
&\ls_\nu C_\ast\e \angt^{-\f{5-\dl}{4}}. \label{ug4}
\end{align}
The above estimates in \eqref{egu1p}, \eqref{ug2}, \eqref{ug3} and \eqref{ug4} indicates the first one in \eqref{solutiongevrey}.

Applying $\th_\nu \p^k_h$ and $\th_\nu\p_z\p^k_h$ to the first equation of \eqref{fhmh} respectively give that
\bes
\bali
&\th_{\nu} \p^k_h \mH=\th_{\nu}\p^k_h \fH+\th_{\nu}\f{z}{2\langle t\rangle} \p^k_h \int^\i_z\mH d\bar{z},\\
&\th_{\nu}\p_z \p^k_h \mH =\th_{\nu}\p^k_h \p_z\fH+\f{\th_{\nu}}{2\langle t\rangle} \p^k_h \int^\i_z\mH d\bar{z}-\f{z}{2\angt}\th_\nu\p^k_h\mH.
\eali
\ees
From the above equalities and by using \eqref{fhmh2}, we can easily deduce that
\begin{align}
&\|\mH\|_{X_{\tau,\kappa,3/4}}\ls \|\fH\|_{X_{\tau,\kappa,3/4}}+\angt^{-1/2}\lt\|\int^\i_z\mH d\bar{z}\rt\|_{X_{\tau,\kappa,\f{13}{16}}}\ls_\nu C_\ast\e \angt^{-\f{3-\dl}{4}},\label{fhmh3}\\
&\angt^{1/2}\|\p_z \mH\|_{X_{\tau,\kappa,3/4}}\ls \angt^{1/2} \|\p_z\fH\|_{X_{\tau,\kappa,3/4}}+\angt^{-1/2}\lt\|\int^\i_z\mH d\bar{z}\rt\|_{X_{\tau,\kappa,3/4}}+\lt\|\mH\rt\|_{X_{\tau,\kappa,\f{13}{16}}}\nn\\
&\ls_\nu C_\ast\e \angt^{-\f{3-\dl}{4}}.\label{fhmh4}
\end{align}
The above estimates in \eqref{fhmh3} and \eqref{fhmh4} indicates the second one in \eqref{solutiongevrey}.

For the estimate of $\mG$, we need to use a product estimate in \eqref{product2} in Lemma \eqref{lemproduct}. From the representation of $\mG$, we have that
\begin{align}
\|\mG\|_{X_{\tau,9,3/4}}\ls& \|\p_x u\|_{X_{\tau,9,3/4}}+\f{\angt^{1-\dl}}{\s{\e}}\lt\|\p_z u\int^\i_z\mH d\bar{z}\rt\|_{X_{\tau,9,3/4}}\nn\\
\ls&  \|u\|_{X_{\tau,11,3/4}}+\f{\angt^{5/4-\dl}}{\s{\e}}\lt(\lt\|\p_z u\rt\|_{X_{\tau,5,3/4}}\lt\|\mH \rt\|_{X_{\tau,9,3/4}}+\lt\|\p_z u\rt\|_{X_{\tau,9,3/4}}\lt\|\mH\rt\|_{X_{\tau,5,3/4}}\rt)\nn\\
\ls& C_\ast \e \angt^{-\f{3-\dl}{4}}+C^2_\ast \e^{3/2} \angt^{-\f{3+2\dl}{4}}\ls C_\ast \e \angt^{-\f{3-\dl}{4}}, \label{mgest1}
\end{align}
by letting $\e$ is small such that $C_\ast\s{\e}\leq 1$.  Also by using \eqref{product2}, we can see that
\begin{align}
&\s{\dl}\angt^{1/2}\|\mG\|_{X_{\tau,7,3/4}}\nn\\
\ls& \s{\dl}\angt^{1/2}\|\p_x \p_z u\|_{X_{\tau,7,3/4}}+\s{\dl}\f{\angt^{3/2-\dl}}{\s{\e}}\lt(\lt\|\p^2_z u\int^\i_z\mH d\bar{z}\rt\|_{X_{\tau,7,3/4}}+\lt\|\p_z u\mH\rt\|_{X_{\tau,7,3/4}}\rt)\nn\\
\ls& \s{\dl} \angt^{1/2}\|\p_zu\|_{X_{\tau,9,3/4}}+\s{\dl}\f{\angt^{7/4-\dl}}{\s{\e}}\lt(\lt\|\p^2_z u\rt\|_{X_{\tau,5,3/4}}\lt\|\mH \rt\|_{X_{\tau,7,3/4}}+\lt\|\p^2_z u\rt\|_{X_{\tau,7,3/4}}\lt\|\mH\rt\|_{X_{\tau,5,3/4}}\rt)\nn\\
\ls& C_\ast \e \angt^{-\f{3-\dl}{4}}.  \label{mgesti2}
\end{align}
Combining estimates in \eqref{mgest1} and \eqref{mgesti2}, we obtain the third line of \eqref{solutiongevrey}.

%
%

Now, we using Sobolev embedding and \eqref{solutiongevrey} to show the low order estimates in \eqref{lowpoint}.  First by using Sobolev embedding in $x,y$ directions and \eqref{solutiongevrey}, we can easily obtain that for $0\leq k\leq 51$ and $0\leq \nu<1$, we have

\be\label{lowpoint1}
\bali
&\angt^{-\f{1}{2}}\lt\|\th_\nu\p^k_h (u,v)\rt\|_{L^\i_h L^2_z}+\lt\|\th_\nu\p^k_h\p_z (u,v)\rt\|_{L^\i_h L^2_z}\\
&+\s{\dl} \angt^{\f{1}{2}}\lt\|\th_\nu\p^2_z\p^k_h(u,v)\rt\|_{L^\i_h L^2_z}
+\dl \angt\lt\|\th_\nu\p_z^3\p^k_h(u,v)\rt\|_{L^\i_h L^2_z}\\
\ls& C_\ast\e \angt^{-\f{5-\dl}{4}}.
\eali
\ee
Also for any $f$ which decay fast enough at $z$ infinity, we have
\begin{align}
|\th_\nu(z) f(z)|=&\th_\nu(z) \lt|\int^\i_z \p_zfd\bar{z}\rt|\leq \lt | \int^\i_z \th_{\f{\nu-1}{2}}(\bar{z})\th_{\f{\nu+1}{2}}(\bar{z})\p_zfd\bar{z}\rt|\nn\\
  \leq & \lt\| \th_{\f{\nu-1}{2}}\rt\|_{L^2_z} \lt\|\th_{\f{\nu+1}{2}}(z)\p_z f\rt\|_{L^2_z}\ls_\nu  \angt^{1/4}  \lt\|\th_{\f{\nu+1}{2}}(z)\p_zf\rt\|_{L^2_z}.\nn
\end{align}
This indicate that for $0\leq\nu<1$, we have
\be\label{lowpoint3}
\|\th_\nu(z) f(z)\|_{L^\i}\ls \angt^{1/4}  \lt\|\th_{\f{\nu+1}{2}}(z)\p_zf\rt\|_{L^\i_h L^2_z}.
\ee

Combining \eqref{lowpoint1} and \eqref{lowpoint3}, we see that
\be\label{lowpoint4}
\bali
\angt^{-\f{1}{4}}\lt\|\th_\nu\p^k_h (u,v)\rt\|_{L^\i}+\s{\dl} \angt^{\f{1}{4}}\lt\|\th_\nu\p_z\p^k_h(u,v)\rt\|_{L^\i}
+\dl \angt^{\f{3}{4}}\lt\|\th_\nu\p_z^2\p^k_h(u,v)\rt\|_{L^\i}\ls C_\ast\e \angt^{-\f{5-\dl}{4}}.
\eali
\ee
Also from \eqref{lowpoint3} and incompressibility, we have
\begin{align}
\angt^{-3/4}\|\th_\nu \p^k_hw\|_{L^\i}\ls \angt^{-1/2}  \lt\|\th_{\f{\nu+1}{2}}(z)\p_z\p^k_hw\rt\|_{L^\i_h L^2_z}\ls \angt^{-1/2}\lt\|\th_{\f{\nu+1}{2}}(z)\p^{k+1}_h\bl{u}\rt\|_{L^\i_h L^2_z}\ls C_\ast\e \angt^{-\f{5-\dl}{4}}. \label{lowpoint5}
\end{align}
Combining \eqref{lowpoint4} and \eqref{lowpoint5}, we can obtain \eqref{lowpoint} in Lemma \ref{llowpoint}.
\qed
\end{appendix}

\section*{Acknowledgments}
\addcontentsline{toc}{section}{Acknowledgments}

\q  X. Pan is supported by National Natural Science Foundation of China (No. 11801268, No. 12031006) and C. J. Xu is supported by National Natural Science Foundation of China (No. 12031006) and the Fundamental Research Funds for the Central Universities of China.

\medskip

 {\footnotesize

 {\sc X. Pan:  School of Mathematics and Key Laboratory of MIIT, Nanjing University of Aeronautics and Astronautics, Nanjing 211106, China}

  {\it E-mail address:}  xinghong\_87@nuaa.edu.cn

  \medskip

{\sc C. J. Xu:  School of Mathematics and Key Laboratory of MIIT, Nanjing University of Aeronautics and Astronautics, Nanjing 211106, China}

  {\it E-mail address:}  xuchaojiang@nuaa.edu.cn

}
\end{document}